\documentclass[reqno,11pt]{amsart}

\usepackage{latexsym}
\usepackage{float}
\usepackage{amssymb}
\usepackage{amsmath}
\usepackage{amsthm}
\usepackage{mathrsfs}
\usepackage{euscript}
\usepackage[cmtip,all]{xy}

\usepackage{bbold}
\usepackage{tabu}
\usepackage{enumerate}

\usepackage{graphicx}

\newcommand{\Mod}[1]{\ (\mathrm{mod}\ #1)}

\makeatletter
\g@addto@macro \normalsize {%
 \setlength\abovedisplayskip{7pt}%
 \setlength\belowdisplayskip{6pt}%
}
\makeatother

\newtheorem{thm}[equation]{Theorem}
\newtheorem{lem}[equation]{Lemma}
\newtheorem{cor}[equation]{Corollary}
\newtheorem{prop}[equation]{Proposition}

\newtheorem{prob}[equation]{Problem}

\theoremstyle{remark}
\newtheorem{rem}[equation]{Remark}
\newtheorem{fact}[equation]{Fact}

\newtheorem{example}[equation]{Example}
\newtheorem{defn}[equation]{Definition}

\numberwithin{equation}{section}

\newcommand{\gb}{\beta}
\newcommand{\ga}{\alpha}

\newcommand{\gD}{\Delta}

\newcommand{\gt}{\theta}

\newcommand{\gS}{\Sigma}

\newcommand{\eps}{\varepsilon}

\newcommand{\fa}{{\mathfrak a}}             
\newcommand{\fb}{{\mathfrak b}}
\newcommand{\fg}{{\mathfrak g}}
\newcommand{\fh}{{\mathfrak h}}
\newcommand{\fk}{{\mathfrak k}}

\newcommand{\fm}{{\mathfrak m}}
\newcommand{\fn}{{\mathfrak n}}

\newcommand{\fp}{{\mathfrak p}}

\newcommand{\ft}{{\mathfrak t}}

\newcommand{\fs}{{\mathfrak s}}

\newcommand{\f}{\mathfrak}

\newcommand{\R}{\mathbb{R}}          
\newcommand{\C}{\mathbb{C}}          
           
\newcommand{\Z}{\mathbb{Z}}

\newcommand{\Ad}{\mathrm{Ad}}
\newcommand{\Cal}{\mathcal}

\newcommand{\Hom}{\operatorname{Hom}}

\newcommand{\Ind}{\mathrm{Ind}}

\newcommand{\Sol}{\mathrm{Sol}}
\newcommand{\CSol}{\Cal{S}ol}
\newcommand{\ESol}{\EuScript{S}\mathrm{ol}}

\newcommand{\Pol}{\mathrm{Pol}}
\newcommand{\To}{\longrightarrow}

\newcommand{\Diff}{\mathrm{Diff}}
\newcommand{\Irr}{\mathrm{Irr}}

\newcommand{\triv}{\mathrm{triv}}
\newcommand{\fin}{\mathrm{fin}}

\newcommand{\D}{\Cal{D}}

\newcommand{\wrho}{\widetilde{\rho}}
\newcommand{\wSL}{\widetilde{SL}}
\newcommand{\wA}{\widetilde{A}}

\newcommand{\dpin}{d\pi_n}

\newcommand{\xid}{\mathrm{id}_{\xi}}

\newcommand{\id}{\mathrm{id}}

\newcommand{\pp}{\textnormal{\mbox{\smaller($+$,$+$)}}}
\newcommand{\pmi}{\textnormal{\mbox{\smaller($+$,$-$)}}}
\newcommand{\mip}{\textnormal{\mbox{\smaller($-$,$+$)}}}
\newcommand{\mm}{\textnormal{\mbox{\smaller($-$,$-$)}}}

\newcommand{\sgn}{\mathrm{sgn}}

\newcommand{\Sylv}{\mathrm{Sylv}}

\newcommand{\Cay}{\mathrm{Cay}}

\newcommand{\Spec}{\mathrm{Spec}}

\providecommand*{\donothing}[1]{}

\makeatletter
\@namedef{subjclassname@2020}{%
  \textup{2020} Mathematics Subject Classification}
\makeatother

\begin{document}

\baselineskip=16pt
\tabulinesep=1.2mm


\title[]{
Classification of $K$-type formulas for the Heisenberg ultrahyperbolic 
operator $\square_s$ for $\wSL(3,\R)$ and
tridiagonal determinants for local Heun functions
 }

\author{Toshihisa Kubo}
\author{Bent {\O}rsted}

\address{Faculty of Economics, 
Ryukoku University,
67 Tsukamoto-cho, Fukakusa, Fushimi-ku, Kyoto 612-8577, Japan}
\email{toskubo@econ.ryukoku.ac.jp}

\address{Department of Mathematics, 
Aarhus University,
Ny Munkegade 118 DK-8000 Aarhus C Denmark}
\email{orsted@imf.au.dk}

\date{\today}

\maketitle


\begin{abstract} 
The $K$-type formulas of the space 
of $K$-finite solutions to the Heisenberg ultrahyperbolic equation
$\square_sf=0$
for the non-linear group $\wSL(3,\R)$ 
are classified. 
This completes a previous study of Kable 
for the linear group $SL(m,\R)$ in the case of $m=3$, as well as 
generalizes our earlier results on a certain second order differential operator.
As a by-product we also show several properties of certain
sequences $\{P_j(x;y)\}_{j=0}^\infty$ and $\{Q_j(x;y)\}_{j=0}^\infty$
of tridiagonal determinants, whose generating functions are given by
local Heun functions.
In particular, 
it is shown that 
these sequences satisfy
a certain arithmetic-combinatorial property, 
which we refer to as a \emph{palindromic property}.
We further show that 
classical sequences of
Cayley continuants  
$\{\Cay_j(x;y)\}_{j=0}^\infty$ 
and Krawtchouk polynomials
$\{\Cal{K}_j(x;y)\}_{j=0}^\infty$ 
also admit this property.
In the end a new proof of Sylvester's formula
for certain tridiagonal determinant $\Sylv(x;n)$ is provided
from a representation theory point of view.
\end{abstract}

\setcounter{tocdepth}{1}
\tableofcontents

 
\section{Introduction}\label{sec:intro}

The representation theory of 
a reductive Lie group $G$  
is intimately tied to the analysis on corresponding flag manifolds
$G/P$, where $P$ is a parabolic subgroup. For a representation 
of $P$, one considers the corresponding homogeneous
vector bundle over the flag manifold $G/P$ 
and the space of sections either in
the smooth sense or in some square integrable sense. 
Understanding the structure
of this space of sections is crucial, for example, whether there are interesting
subspaces invariant under the left regular action of $G$. This might for example
happen if there is given a $G$-invariant differential equation on the space of
sections, so that its solutions form an invariant subspace. In order to analyze
such a space of solutions, it is convenient to use the analogue of Fourier series,
namely, via the maximal compact subgroup $K$ in $G$; here $G = KP$, so the
flag manifold $G/P$ 
is also  a homogeneous space for $K$, and sections can be described
according to how they transform under the action of $K$.

In the case that we consider in this paper, 
there will be a one-parameter family $\square_s$
of natural invariant differential operators 
and explicit spaces of solutions.
It turns out that these will be related to classical function theory, namely,
hypergeometric functions and local Heun functions, and
some of the relevant identities analogous to 
classical tridiagonal determinants of Sylvester type and Cayley type.
Further, the representations obtained from the solution space to the 
differential equation $\square_sf=0$ 
include ones, which are to be thought of as minimal representations (in some sense).
We provide a new aspect on a connection
between the representation theory of reductive groups,
ordinary differential equations in the complex domains,
and sequences of polynomials.

The aim of this paper is threefold. The first is the classification of $K$-type 
formulas of the space  
of $K$-finite solutions to the Heisenberg ultrahyperbolic equation
$\square_sf=0$ for $\wSL(3,\R)$.
The second is a study of certain sequences 
$\{P_k(x;y)\}_{k=0}^\infty$ and $\{Q_k(x;y)\}_{k=0}^\infty$
of tridiagonal determinants arising from the study of the $K$-finite solutions
to $\square_sf=0$. 
The third is an application of the arguments for
$\{P_k(x;y)\}_{k=0}^\infty$ and $\{Q_k(x;y)\}_{k=0}^\infty$
to classical sequences of Cayley continuants
$\{\Cay_k(x;y)\}_{k=0}^\infty$ and Krowtchouk polynomials 
$\{\Cal{K}_k(x;y)\}_{k=0}^\infty$.
We describe these three topics in detail now.

\subsection{Heisenberg ultrahyperbolic operator $\square_s$}
For a moment,
let $\fg=\f{sl}(m,\C)$ with $m \geq 3$
and 
take a real form $\fg_0$ of $\fg$ such that 
there exists a parabolic subalgebra 
$\fp_0=\fm_0 \oplus \fa_0 \oplus \fn_0$ 
with Heisenberg nilpotent radical $\fn_0$,
namely, 
$\dim_\R [\fn_0, \fn_0] =1$. 
The Heisenberg condition on $\fn_0$ 
forces $\fg_0$ to be either $\f{sl}(m,\R)$ or $\f{su}(p,q)$ with $p+q=m$.
We write $\fp=\fm \oplus \fa \oplus \fn$ for 
the complexification of $\fp_0=\fm_0 \oplus \fa_0 \oplus \fn_0$.
We write $\Cal{U}(\fg)$ for the universal enveloping algebra of the complexified
Lie algebra $\fg = \fg_0 \otimes_{\R}\C$.

In \cite{Kable12C}, 
under the framework of $\fg$ and $\fp$ as above, 
Kable introduced a one-parameter family 
$\square_s$ of differential operators  with $s\in \C$ 
as an example of conformally invariant systems
(\cite{BKZ08, BKZ09}).
The operator $\square_s$ is referred to
as the \emph{Heisenberg ultrahyperbolic operator} (\cite{Kable12C}).
For instance, for $m=3$, it is defined as
\begin{equation}\label{eqn:HUO}
\square_s=R((XY+YX)+s[X,Y]),
\end{equation}
where $R$ denotes the infinitesimal right translation
and $X, Y$ are certain nilpotent elements in $\f{sl}(3,\R)$
(see \eqref{eqn:XY}).
The differential operator 
$\square_s$ for $m=3$ in \eqref{eqn:HUO}
in particular recovers the second order 
differential operator studied in \cite{KuOr19} as 
the case of $s=0$. 
Some algebraic and analytic properties of
the Heisenberg ultrahyperbolic operator 
$\square_s$ as well as its generalizations are investigated  
in \cite{Kable12B, Kable12C, Kable15, Kable18}
for a linear group $SL(m,\R)$.

From a viewpoint of intertwining operators,
the Heisenberg ultrahyperbolic operator $\square_s$
is an intertwining differential operator
between parabolically induced representations for $G \supset P$,
where $G$ and $P$ are Lie groups with Lie algebras
$\fg_0$ and $\fp_0$, respectively.
Thus the space $\CSol(\square_s)$ 
of smooth solutions to the equation $\square_s f =0$ 
in the induced representation 
is a subrepresentation of $G$,
and also the space $\CSol(\square_s)_K$ of $K$-finite solutions 
is a $(\fg, K)$-module.
We then consider the following problem.

\begin{prob}\label{prob:1}
Classify the $K$-type formulas for $\CSol(\square_s)_K$.
\end{prob}

In \cite{Kable12C}, 
an attempt for this direction 
was made for $G=SL(m,\R)$.
In the paper, 
Kable introduced a notion of \emph{$\EuScript{H}$-modules}
and developed an algebraic theory for them.
Although the theory is powerful to compute
the dimension of the $\fk\cap \fm$-invariant subspace of 
the space of $K$-type solutions for each $K$-type in $\CSol(\square_s)_K$,
the determination of the explicit $K$-type formulas 
was not achieved. 

In this paper, we consider $G=\wSL(3,\R)$,
the universal covering group of $SL(3,\R)$.
By making use of a technique from our earlier paper \cite{KuOr19},
we successfully classified the $K$-type formulas
of $\CSol(\square_s)_K$ for $\wSL(3,\R)$.
Our method is different from the algebraic theory used in \cite{Kable12C};
we rather utilize differential equations.
In order to describe our results in more detail we next introduce some notation.

\subsection{$K$-type formulas}

For the rest of this introduction let $G=\wSL(3,\R)$ with  Lie algebra 
$\fg_0$. Fix a minimal parabolic subgroup $B$ of $G$ 
with Langlands decomposition
$B=MAN$.  Here the subgroup $M$ is isomorphic to 
the quaternion group $Q_8$ of order 8.
Let $K$ be a maximal compact subgroup of $G$ so that $G=KAN$
is an Iwasawa decomposition of $G$. We have $K \simeq SU(2) \simeq Spin(3)$.

Let $\Irr(M)$ and $\Irr(K)$ denote the set of equivalence classes of irreducible representations of $M$ and $K$, respectively. 
As $M \simeq Q_8$, the set $\Irr(M)$ may be given as
\begin{equation*}
\Irr(M)=\{\pp, \,  \pmi, \,  \mip,\,  \mm, \, \mathbb{H}\},
\end{equation*}
where $(\pm,\pm)$ are some characters 
(see Section \ref{subsec:IrrM} for the definition)
and $\mathbb{H}$ stands for
the unique two-dimensional genuine representation of $M$.
The character $\pp$ is, for instance, the trivial character.
Let $\left(n/2\right)$ denote the irreducible 
finite-dimensional representation of $K \simeq Spin(3)$ with 
dimension $n+1$. Then we have 
\begin{equation*}
\Irr(K) = \left\{\left(n/2\right): n\in \Z_{\geq 0}\right\}.
\end{equation*}

For $\sigma \in \Irr(M)$ and a character $\lambda$ of $A$,
we write
\begin{equation*}
I(\sigma,\lambda) = \Ind_B^G(\sigma\otimes (\lambda+\rho) \otimes \mathbb{1})
\end{equation*}
for the representation of $G$ induced from the representation
$\sigma\otimes (\lambda+\rho) \otimes \mathbb{1}$ of $B=MAN$,
where $\rho$ is half the sum of the positive roots corresponding to $B$.
We realize the induced representation $I(\sigma,\lambda)$
on the space of smooth sections for a $G$-equivariant homogeneous 
vector bundle over $G/B$. 
Let $\Diff_G(I(\sigma_1,\lambda_1), I(\sigma_2,\lambda_2))$ 
denote the space of intertwining differential operators
from
$I(\sigma_1,\lambda_1)$ to $I(\sigma_2,\lambda_2)$.

Let $\fg$ denote the complexified Lie algebra of $\fg_0 = \f{sl}(3,\R)$.
Since $\fg= \f{sl}(3,\C)$, the Heisenberg ultrahyperbolic operator 
$\square_s$ is given as in \eqref{eqn:HUO}. We then set 
\begin{equation}\label{eqn:Ds_intro}
D_s:=(XY+YX)+s[X,Y] \in \Cal{U}(\fg),
\end{equation}
so that $\square_s = R(D_s)$.
It follows from Proposition \ref{prop:square} that
we have
\begin{equation*}
R(D_s) \in \Diff_G(I(\pp, -\wrho(s)), I(\mm,\wrho(-s))),
\end{equation*}
where $\wrho(s)$ is a certain weight determined by 
$\wrho:=\rho/2$ (see \eqref{eqn:wrho(s)}).
It is further shown that, for $\sigma \in \Irr(M)$, 
\begin{equation*}
R(D_s) \otimes \id_\sigma
\in \Diff_G(I(\sigma, -\wrho(s)), I(\mm\otimes \sigma,\wrho(-s))),
\end{equation*}
where $\id_\sigma$ denotes the identity map on $\sigma$.

For notational convenience we consider 
$R(D_{\bar{s}})$ in place of $R(D_s)$,
where $\bar{s}$ denotes the complex conjugate of $s \in \C$
(see Section \ref{subsec:HUO} for the details).
 We define
\begin{align}\label{eqn:SolDs0}
\CSol(s;\sigma)&:=
\text{the space of smooth solutions to 
$(R(D_{\bar{s}}) \otimes \id_\sigma)f=0$},\nonumber\\
\CSol(s;\sigma)_K&:=
\text{the space of $K$-finite solutions to 
$(R(D_{\bar{s}}) \otimes \id_\sigma)f=0$}.
\end{align}

It follows from
a Peter--Weyl theorem
for the solution space (Theorem \ref{thm:PW1}) 
that the $(\fg, K)$-module $\CSol(s;\sigma)_K$
decomposes as
\begin{equation}\label{eqn:PWSL_intro}
\CSol(s;\sigma)_K
\simeq 
\bigoplus_{n\in \Z_{\geq 0}}
(n/2) \otimes 
\mathrm{Hom}_{M}\left(
\Sol(s;n), \sigma\right),
\end{equation}
where
$\Sol(s;n)$ is the space of $K$-type solutions to 
$D_s$ (without complex conjugation on $s$)  
on the $K$-type $(n/2)=(\delta, V_{(n/2)}) \in \Irr(K)$,
that is, 
\begin{equation}\label{eqn:Solsa}
\Sol(s;n)=\{ v \in V_{(n/2)}:d\delta(D^\flat_s)v=0\}.
\end{equation}
(see \eqref{eqn:PWSL} and \eqref{eqn:SolK3}).
Here $d\delta$ is the differential of $\delta$ and 
$D^\flat_s$ denotes the compact model of $D_s$ (see Definition \ref{def:ub}).
Then the $K$-type decomposition $\CSol(s;\sigma)_K$
for $(s, \sigma) \in \C \times \Irr(M)$ is explicitly given as follows.

\begin{thm}\label{thm:K-type0}
The following conditions on $(\sigma, s) \in \Irr(M) \times \C$ 
are equivalent.

\begin{enumerate}[{\normalfont (i)}]
\item $\CSol(s;\sigma)\neq \{0\}$.
\item One of the following conditions holds.

\begin{itemize}
\item $\sigma = \pp:$ $s \in \C$.
\item $\sigma=\mm:$ $s \in \C$.
\item $\sigma=\mip:$ $s \in 1+4\Z$.
\item $\sigma=\pmi:$ $s \in 3+4\Z$.
\item $\sigma=\mathbb{H}:$ \hspace{13pt} $s \in 2\Z$.
\end{itemize}
\end{enumerate}

Further, the $K$-type formulas for $\CSol(s;\sigma)_K$ 
are given as follows.

\begin{enumerate}[{\quad \normalfont (1)}]
\item $\sigma = \pp:$ 
\begin{equation*}
\qquad
\CSol(s;\pp)_K \simeq 
\bigoplus_{n\in \Z_{\geq 0}} \left(2n\right) 
\hspace{0.4in} \textnormal{for all $s \in \C$.}
\qquad
\end{equation*}

\item $\sigma=\mm:$
\begin{equation*}
\CSol(s;\mm)_K \simeq 
\bigoplus_{n\in \Z_{\geq 0}} \left(1+2n\right) 
\quad \textnormal{for all $s \in \C$.}
\end{equation*}

\item $\sigma=\mip:$
\begin{equation*}
\qquad \qquad
\CSol(s;\mip)_K
\simeq
\bigoplus_{n\in \Z_{\geq 0}}
((|s|+1)/2+2n) \quad \textnormal{for $s \in 1+4\Z$}.
\end{equation*}

\item $\sigma=\pmi:$
\begin{equation*}
\qquad \qquad
\CSol(s;\pmi)_K
\simeq
\bigoplus_{n\in \Z_{\geq 0}}
((|s|+1)/2+2n) \quad \textnormal{for $s \in 3+4\Z$}.
\end{equation*}

\item $\sigma=\mathbb{H}:$
\begin{equation*}
\qquad \qquad
\CSol(s;\mathbb{H})_K
\simeq
\bigoplus_{n\in \Z_{\geq 0}}
((|s|+1)/2+2n) \quad \textnormal{for $s \in 2\Z$}.
\end{equation*}

\end{enumerate}
\end{thm}

We shall deduce Theorem \ref{thm:K-type0} 
from Theorem \ref{thm:K-type}
at the end of Section \ref{sec:HGE}.
The proof is not case-by-case analysis on $\sigma \in \Irr(M)$; 
each $M$-representation $\sigma$ is treated uniformly via 
a recipe for the $K$-type decomposition of $\CSol(s;\sigma)_K$.
See Section \ref{subsec:recipe2} for the details of the recipe.
In regard to Theorem \ref{thm:K-type0}, we shall also classify
the space $\Sol(s;n)$ of $K$-type solutions for all $n \in \Z_{\geq 0}$.
This is done in Theorems \ref{thm:Sol1} and \ref{thm:SolKO1}.
Recall from \eqref{eqn:SolDs0} that 
$\CSol(s;\sigma)_K$
concerns $R(D_{\bar{s}})$.
Theorem \ref{thm:K-type0} shows that
the $K$-type decompositions
are in fact independent of 
taking the complex conjugate on the parameter $s \in \C$. 
\vspace{5pt}

There are several remarks on
the space $\Sol(s;n)$ of $K$-type solutions
and the $K$-type decompositions of $\CSol(s;\sigma)_K$.
First, as mentioned above,
the dimensions of $\fk \cap \fm$-invariant subspaces 
of the spaces of $K$-type solutions are determined 
 in \cite{Kable12C}
for $SL(m,\R)$ 
with arbitrary rank $m \geq 3$.
In particular, as $\fk \cap \fm = \{0\}$ for $m=3$,
the dimensions $\dim_\C \Sol(s;n)$ for $n \in 2\Z_{\geq 0}$
are obtained in the cited paper (\cite[Thm.\ 5.13]{Kable12C}).
We note that our normalization on $s\in \C$ is different from one 
for $z \in \C$ in \cite{Kable12C} by $s=-2z$. 
In the paper, 
factorization formulas of certain tridiagonal determinants 
are essential to compute $\dim_\C \Sol(s;n)$.  
For further details on the factorization formulas, see the remark
after \eqref{eqn:Sylv0} below.
By making use of differential equations, 
we determine the dimensions $\dim_\C \Sol(s;n)$ for all $n \in \Z_{\geq 0}$
independently to the results of \cite{Kable12C}.

Table \ref{table:tableKKO} summarizes
the results of \cite{Kable12C} and this paper, concerning the $K$-type decompositions of $\CSol(s;\sigma)_K$ for $\fg_0 = \f{sl}(3,\R)$.
\vspace{-5pt}
\begin{table}[H]
\caption{Comparison between \cite{Kable12C} 
and this paper for $\fg_0 = \f{sl}(3,\R)$}
\begin{center}
\renewcommand{\arraystretch}{1.5}
{
\begin{tabular}{c||c|c}
\hline
$\fg_0=\f{sl}(3,\R)$ & $\dim_\C\Sol(s;n)$ 
& $K$-type decomposition of $\CSol(s;\sigma)_K$ \\
\hline
\cite{Kable12C} & Done for $SL(3,\R)$ & Not obtained\\
\hline
this paper & Done for $\wSL(3,\R)$ & Done for $\wSL(3,\R)$\\
\hline 
\end{tabular}
}
\end{center}
\label{table:tableKKO}
\end{table}
\vspace{-5pt}

Secondly, Tamori recently investigates in \cite{Tamori20}
the representations on the space of smooth solutions 
to the same differential equation as 
the Heisenberg ultrahyperbolic equation $R(D_s)f=0$,
as part of his thorough study on minimal representations.
In particular, the $K$-type formula
$\CSol(s;\mathbb{H})_K$
for the case $\sigma=\mathbb{H}$ is 
determined in \cite[Prop.\ 5.4.6]{Tamori20}.
The representations on $\CSol(s;\mathbb{H})$ for $s \in 2\Z$ 
are genuine representations 
and not unitarizable unless $s=0$ (\cite[Rem.\ 5.4.5]{Tamori20}).
We remark that, both in \cite{Tamori20} and in this paper,
the $K$-type formula for $\CSol(s;\mathbb{H})_K$ is obtained by
realizing the $K$-type solutions in $\Sol(s;n)$ 
as hypergeometric polynomials;
nevertheless, the methods are rather different. For instance,
some combinatorial computations are carried out in \cite{Tamori20},
whereas we simply solve the hypergeometric equation.
Further, we also give the $K$-type solutions by Heun polynomials.
We illustrate our methods in detail in Section \ref{subsec:method}.

Lastly, the $K$-type formulas
$\CSol(0;\sigma)_K$
for the case of $s=0$ are 
previously classified by ourselves in \cite[Thm.\ 1.6]{KuOr19}.
In this case
only three
$M$-representations
$\sigma=\pp$, $\mm$, $\mathbb{H}$ 
contribute to $\CSol(0;\sigma)\neq \{0\}$ with explicit  $K$-type formulas
\begin{align*}
\CSol(0;\pp)_K &\simeq \bigoplus_{n\in \Z_{\geq 0}} \left(2n\right), \\[5pt]
\CSol(0;\mm)_K &\simeq \bigoplus_{n\in \Z_{\geq 0}} \left(1+2n\right), \\[5pt]
\CSol(0;\mathbb{H})_K &\simeq \bigoplus_{n\in \Z_{\geq 0}} \left((1/2)+2n\right). 
\end{align*}
It was quite mysterious that only the series of $(3/2)+2\Z_{\geq 0}$ does not 
appear in the $K$-type formulas. 
Theorem \ref{thm:K-type0} now gives an answer for this question.

The representations realized on $\CSol(0;\sigma)$ for 
$\sigma=\pp$, $\mm$, $\mathbb{H}$ are known to be 
unitarizable and the resulting representations 
are the ones attached to the minimal nilpotent orbit (\cite{RS82}). 
For instance, the representation realized on $\CSol(0;\mathbb{H})$ 
is the genuine representation so-called Torasso's representation 
(\cite{Torasso83}).
Here are two remarks on a recent progress 
on the unitarity of the representations on
$\CSol(0;\sigma)$.
First, Dahl recently constructed in \cite{Dahl19} the unitary structures 
for the three unitarizable representations 
on $\CSol(0;\sigma)$ by using the Knapp--Stein intertwining operator
and the Fourier transform on the Heisenberg group. 
Furthermore, Frahm also recently gives the $L^2$-models
of these unitary representations in \cite{Frahm20},
as part of his intensive study on the $L^2$-realizations of the minimal 
representations realized on the solution space of conformally invariant
systems.

\subsection{Hypergeometric and Heun's differential equations}
\label{subsec:method}

The main idea for accomplishing Theorem \ref{thm:K-type0}
is use of the decomposition formula \eqref{eqn:PWSL_intro} 
and a refinement of the method applied for the case $s=0$ in \cite{KuOr19}.
The central problem is classifying the space $\Sol(s;n)$ of $K$-type solutions.
In order to do so,
one needs to proceed the following two steps.
\begin{enumerate}[]
\item Step 1: Classify $(s,n) \in \C \times \Z_{\geq 0}$ 
so that $\Sol(s;n)\neq \{0\}$.
\vspace{3pt}
\item Step 2: Classify the $M$-representations on $\Sol(s;n) \neq \{0\}$.
\end{enumerate}

In \cite{KuOr19} (the case $s=0$),
via the polynomial realization
\begin{equation}\label{eqn:Irr(K)_intro}
\Irr(K)\simeq \{(\pi_n, \Pol_n[t]):n\in \Z_{\geq 0}\}
\end{equation}
of $\Irr(K)$ (see \eqref{eqn:IrrK}),
we carried out the two steps by
realizing $\Sol(0;n)$ in \eqref{eqn:Solsa} as 
\begin{equation*}
\Sol(0;n)=\{p(t) \in \Pol_n[t]: d\pi^{\textnormal{I}}_n(D^\flat_0)p(t)=0\},
\end{equation*}
the space of polynomial solutions to an ordinary differential equation 
$d\pi^{\textnormal{I}}_n(D^\flat_0)p(t)=0$ 
via a certain identification $\Omega^\textnormal{I} \colon
\fk \stackrel{\sim}{\to} \f{sl}(2,\C)$
(see Sections \ref{subsec:so3} and \ref{subsec:Omega12}).
It turned out that the differential equation 
$d\pi_n^\textnormal{I}(D^\flat_0)p(t)=0$ 
is a hypergeometric differential equation.
Steps 1 and 2 were then easily carried out by observing 
the Gauss hypergeometric functions arising from 
$d\pi_n^\textnormal{I}(D^\flat_0)p(t)=0$.

In this paper we apply the same idea for the general case $s \in \C$.
Nonetheless, in the general case, Steps 1 and 2 do not become as simple as 
the case of $s=0$, as the differential equation 
$d\pi_n^\textnormal{I}(D^\flat_s)p(t)=0$ 
is turned out to be Heun's differential equation
\begin{equation}\label{eqn:HeunOp_intro}
\D_H(-1, -\frac{ns}{4}; -\frac{n}{2}, -\frac{n-1}{2}, \frac{1}{2}, \frac{1-n-s}{2};t^2)
p(t)=0
\end{equation}
with the $P$-symbol
\begin{equation*}
P\left\{
\begin{matrix}
0 & 1 & -1 & \infty & &\\
0 & 0 & 0 & -\frac{n}{2} & t^2 & -\frac{ns}{4}\\
\frac{1}{2} & \frac{1+n+s}{2} & \frac{1+n-s}{2} & -\frac{n-1}{2} & &
\end{matrix}
\right\}.
\end{equation*}
\vskip 0.1in
\noindent
(See 
Section \ref{subsec:Heun}
for the definition of 
$\D_H(a,q;\ga, \gb, \gamma,\delta;z)$
and the notation of the $P$-symbol.)
As such, one needs to deal with local Heun functions at 0
(\cite{Ron95});
these are not as easy to handle as hypergeometric functions for our purpose.
For instance, in Step 1, one needs to classify $(s,n) \in \C \times \Z_{\geq 0}$ 
for which the local Heun functions
in consideration are polynomials; however, 
as opposed to hypergeometric functions,
classifying such parameters is not an easy problem at all,
since 
only known are
necessary conditions in which 
local Heun functions are reduced to polynomials.
To resolve this problem, inspired by a work \cite{Tamori20} of Tamori,
we use a different identification
$\Omega^\textnormal{II} \colon \fk \stackrel{\sim}{\to} \f{sl}(2,\C)$ 
so that the differential equation $d\pi_n^\textnormal{II}(D^\flat_s)p(t)=0$ 
becomes again a hypergeometric equation, namely,
\begin{equation}\label{eqn:HGEop_intro}
D_F(-\frac{n}{2},-\frac{n+s-1}{4},\frac{3-n+s}{4};t^2)p(t)=0.
\end{equation}
\vskip 0.1in
\noindent
(See \eqref{eqn:HGEop} for the definition of 
$\D_F(a,b,c;z)$.) For the detail, see
Sections \ref{subsec:so3} and \ref{subsec:Omega12}.
We accomplished the $K$-type formulas of $\CSol(s;\sigma)_K$
by using the hypergeometric model 
$d\pi_n^\textnormal{II}(D^\flat_s)p(t)=0$
(see Theorem \ref{thm:K-type}).

The Heun model 
$d\pi_n^\textnormal{I}(D^\flat_s)p(t)=0$
and
the hypergeometric model 
$d\pi_n^\textnormal{II}(D^\flat_s)p(t)=0$
are related by a Cayley transform
$\pi_n(k_0)$ given by an element $k_0$ of $K$ (\eqref{eqn:k0}).
Namely, we have
\begin{equation*}
d\pi_n^\textnormal{II}(D^\flat_s)p(t)=0
\quad
\Longleftrightarrow
\quad
d\pi_n^\textnormal{I}(D^\flat_s)\pi_n(k_0)p(t)=0
\end{equation*}
(see Proposition \ref{prop:KOT}).
For instance, put
\begin{equation}\label{eqn:Heun_intro}
u_{[s;n]}(t):=
Hl(-1, -\frac{ns}{4}; -\frac{n}{2}, -\frac{n-1}{2}, \frac{1}{2}, \frac{1-n-s}{2};t^2)
\end{equation}
and
\begin{equation*}
a_{[s;n]}(t):= F(-\frac{n}{2}, -\frac{n+s-1}{4}, \frac{3-n+s}{4};t^2),
\end{equation*}
where $Hl(a,q;\ga, \gb,\gamma,\delta;z)$ denotes the local Heun function
at $z=0$ and $F(a,b,c;z)\equiv {_2}F_1(a,b,c;z)$ is the Gauss hypergeometric
function. Then, for 
\begin{equation*}
\D_H^{[s;n]}
:=\D_H(-1, -\frac{ns}{4}; -\frac{n}{2}, -\frac{n-1}{2}, \frac{1}{2}, \frac{1-n-s}{2};t^2)
\end{equation*}
in \eqref{eqn:HeunOp_intro} and
\begin{equation*}
\D_F^{[s;n]}:=D_F(-\frac{n}{2},-\frac{n+s-1}{4},\frac{3-n+s}{4};t^2)
\end{equation*}
in \eqref{eqn:HGEop_intro},
we have
\begin{equation}\label{eqn:DF0}
\D_H^{[s;n]}u_{[s;n]}(t)=0
\quad
\text{and}
\quad
\D_F^{[s;n]}a_{[s;n]}(t)=0.
\end{equation}
It will be shown in Lemma \ref{lem:n0} that,
for appropriate $(s,n) \in \C\times \Z_{\geq 0}$
at which $u_{[s;n]}(t)$ and $a_{[s;n]}(t)$ are both polynomials,
the two functions $u_{[s;n]}(t)$ and $a_{[s;n]}(t)$ are related as
\begin{equation}\label{eqn:au0}
\pi_n(k_0)a_{[s;n]}(t) \in \C u_{[s;n]}(t),
\end{equation}
equivalently,
\begin{equation*}
(1-\sqrt{-1}t)^n
a_{[s;n]}\left(\frac{1+\sqrt{-1}t}{1-\sqrt{-1}t}\cdot\sqrt{-1}\right)
\in \C u_{[s;n]}(t).
\end{equation*}
\vskip 0.1in
\noindent
For the case of $s=0$, the Gauss-to-Heun transformation 
by $\pi_n(k_0)$ from $a_{[s;n]}(t)$ to $u_{[s;n]}(t)$ can be 
reduced to a Gauss-to-Gauss transformation, as the local Heun function
$u_{[s;n]}(t)$ can be reduced to a hypergeometric function for $s=0$.
(See Remarks \ref{rem:S0F} and \ref{rem:S0F2}.)

The proportional constant for \eqref{eqn:au0} may be given by 
a ratio of shifted factorials.
Indeed, let $I^-_0$ and $J_0$ be certain subsets of $\Z_{\geq 0}$
(see \eqref{eqn:IJ}).
Then, for $n \equiv 0 \Mod 4$ with $s \in \C \backslash (I^-_0 \cup J_0)$,
we have
\begin{equation*}
\pi_n(k_0)a_{[s;n]}(t) = 
\frac{\left(\frac{2-n}{4},\frac{n}{4}\right)}
        {\left(\frac{3-n+s}{4}, \frac{n}{4}\right)}
u_{[s;n]}(t)
\end{equation*}
with $(\ell,m):=\frac{\Gamma(\ell+m)}{\Gamma(\ell)}$.
(As the proportional constants do not play any role for our main purposes,
we shall not discuss them in this paper.)

We remark that
a Cayley transform,
which is slightly different from ours $\pi_n(k_0)$,
is used also in \cite{GH16} under the name of 
the MacWilliams transform,
to construct the generating function of
the discrete Chebyshev polynomials from Jacobi polynomials
via the Heun equation.
Further, Gauss-to-Heun transformations are studied 
in, for instance, \cite{Maier05, VF13, VF14}.
In the cited papers, 
considered are the cases in which
local Heun functions are pulled back from
a hypergeometric function.
As opposed to the references,
this paper concerns some cases that,
for suitable $(s;n) \in \C \times \Z_{\geq 0}$,
a certain linear combination of
$a_{[s;n]}(t)$ and the second solution
 $b_{[s;n]}(t)$ to \eqref{eqn:HGEop_intro}
 is transformed to $u_{[s;n]}(t)$. Namely,
put
\begin{equation*}
b_{[s;n]}(t):=t^{\frac{1+n-s}{2}}F(-\frac{n+s-1}{4},-\frac{s-1}{2}, \frac{5+n-s}{4};t^2).
\end{equation*}
Then, for appropriate $(s,n) \in \C\times \Z_{\geq 0}$, 
we have
\begin{equation*}
\pi_n(k_0)
\left(a_{[s;n]}(t)\pm C(s;n)b_{[s;n]}(t)\right) \in \C u_{[s;n]}(t),
\end{equation*}
where $C(s;n)$ is some constant, to be defined in \eqref{eqn:Csn}.
For more details, see Section \ref{subsec:SolTKO}.
Via the transformation $\pi_n(k_0)$, we also give in detail
the space $\Sol(s;n)$ of $K$-type solutions 
for the Heun model $d\pi_n^\textnormal{I}(D^\flat_s)p(t)=0$
(Section \ref{sec:Heun}).

\subsection{Palindromic property of a sequence 
$\{p_k(x;y)\}_{k=0}^\infty$ of polynomials}
\label{subsec:PPS}

In the last half of this paper, 
we shall discuss certain arithmetic-combinatorial 
properties of four sequences of polynomials of two variables, namely.
\begin{equation}\label{eqn:seq0}
\{P_k(x;y)\}_{k=0}^{\infty}, \quad
\{Q_k(x;y)\}_{k=0}^{\infty}, \quad
\{\Cay_k(x;y)\}_{k=0}^{\infty}, \quad \text{and} \quad
\{\Cal{K}_k(x;y)\}_{k=0}^{\infty},
\end{equation}
where 
$\{P_k(x;y)\}_{k=0}^{\infty}$
and
$\{Q_k(x;y)\}_{k=0}^{\infty}$
are obtained from the study of  the Heun model 
$d\pi_n^\textnormal{I}(D^\flat_0)p(t)=0$.
Since we could not find the property in consideration, we refer to 
it as a \emph{palindromic property}.
As such, we shall digress for a moment 
from the main results of this paper.

To state the general definition,
let $\{p_k(x;y)\}_{k=0}^\infty$ be a sequence of polynomials
of two variables with $p_0(x;y)=1$, and
$\{a_k \}_{k=0}^\infty$ a sequence of non-zero numbers.
Let $d\colon \Z_{\geq 0} \to \R$ be a given map.
For $n \in \Z_{\geq 0}$, we put
\begin{equation*}
\ESol_k(p;n) :=\{s\in \C : p_k(s;n) = 0\}.
\end{equation*}

\begin{defn}[palindromic property]\label{def:palindromic}
A pair $(\{p_k(x;y)\}_{k=0}^\infty, \{a_k\}_{k=0}^\infty)$
 is said to be a \emph{palindromic pair} 
 with \emph{degree} $d(n)$
if, for each $n \in d^{-1}(\Z_{\geq 0})$,
there exists a map
\begin{equation*}
\gt_{(p;n)} \colon \ESol_{d(n)+1}(p;n)
 \to \{\pm 1\}
 \end{equation*}
 such that,
for all $s \in \ESol_{d(n)+1}(p;n)$, we have
$p_k(s;n)=0$ for $k \geq d(n)+1$ and
\begin{equation}\label{eqn:intro_palin0}
\frac{p_k(s;n)}{a_k} = \gt_{(p;n)}(s)\frac{p_{d(n)-k}(s;n)}{a_{d(n)-k}}
\quad   \text{for $k \leq d(n)$}.
\end{equation}

We call the identity \eqref{eqn:intro_palin0}
the \emph{palindromic identity} and
$\gt_{(p;n)}(s)$ 
its \emph{sign factor}. 
Further, for a given palindromic pair
$(\{p_k(x;y)\}_{k=0}^\infty, \{a_k\}_{k=0}^\infty)$,
the sequence $\{a_k\}_{k=0}^\infty$ is said to be
an \emph{associated sequence} of $\{p_k(x;y)\}_{k=0}^\infty$.
If there exits a sequence $\{a_k\}_{k=0}^\infty$ such that 
$(\{p_k(x;y)\}_{k=0}^\infty, \{a_k\}_{k=0}^\infty)$ is a palindromic pair,
then we say that $\{p_k(x;y)\}_{k=0}^\infty$ admits a 
\emph{palindromic property}.
\end{defn}

We note that our notion of palindromic property does not 
concern polynomials themselves but sequences of polynomials;
for instance, each $p_k(x;n)$ is not necessarily a palindromic polynomial.

\begin{example}\label{example:introPalin}
Here are some simple examples.
\begin{enumerate}[(a)]

\item 
If $p_0(x;y)=1$ and $p_k(x;y)=0$ for all $k \geq 1$,
then, for any sequence $\{a_k\}_{k=0}^\infty$ of non-zero numbers,
the pair $(\{p_k(x;y)\}_{k=0}^\infty, \{a_k\}_{k=0}^\infty)$ is a palindromic pair
with degree $0$.
\vskip 0.1in

\item 
If $p_0(x;y)=p_1(x;y)=1$ and $p_k(x;y)=0$ for all $k \geq 2$,
then, for any sequence $\{a_k\}_{k=0}^\infty$ of non-zero numbers
with $a_0=\pm a_1$,
the pair $(\{p_k(x;y)\}_{k=0}^\infty, \{a_k\}_{k=0}^\infty)$ is a palindromic pair
with degree $1$.
\vskip 0.1in

\item If $p_k(x;y)=x^k$, then, for any sequence $\{a_k\}_{k=0}^\infty$ of
non-zero numbers, 
the pair $(\{x^k\}_{k=0}^\infty, \{a_k\}_{k=0}^\infty)$ is a palindromic pair
with degree $0$.
\vskip 0.1in
\end{enumerate}
We shall provide four examples for which the degree $d(n)$ is not constant.
(See Sections \ref{subsec:intro_palin}, \ref{subsec:introCayley}, 
and \ref{subsec:introKraw}.)
\end{example}

Now suppose that $(\{p_k(x;y)_{k=0}^\infty, \{a_k\}_{k=0}^\infty)$
is a palindromic pair 
with degree $d(n)$, 
where the associated sequence 
$\{a_k\}_{k=0}^\infty$ is of the form
$a_k = \left(c_k\right)!$ for some
sequence $\{c_k\}_{k=1}^\infty$ of non-negative integers with $c_0=1$.
Then, as $p_0(x;y)=1$ by definition, 
the palindromity of $(\{p_k(x;y)_{k=0}^\infty, \{a_k\}_{k=0}^\infty)$
implies that, for $n \in d^{-1}(\Z_{\geq 0})$, the $d(n)$th term
$p_{d(n)}(x;y)$ satisfies the identity
\begin{equation}\label{eqn:intro_factorial}
p_{d(n)}(s;n)=\gt_{(p;n)}(s)\left(c_{d(n)}\right)!
\quad 
\text{for $s \in \ESol_{d(n)+1}(p;n)$}.
\end{equation}
We refer to the identity \eqref{eqn:intro_factorial} as 
the \emph{factorial identity} of $p_{d(n)}(x;n)$ on 
$\ESol_{d(n)+1}(p;n)$.
Further, we mean by the \emph{refinement} of a factorial 
(resp.\ palindromic) identity
a factorial (resp.\ palindromic) 
identity for each explicit value of $s \in \ESol_{d(n)+1}(p;n)$.

In order to show the refinements of a palindromic property and factorial identity
for $\{p_k(x;y)\}_{k=0}^\infty$, it is important to understand 
the zero set $\ESol_{d(n)+1}(p;n)$ of $p_{d(n)+1}(x;n)$.
For the purpose we also consider the \emph{factorization formula}
of the polynomial $p_{d(n)+1}(x;n)$.

In summary, 
we shall investigate the following properties
for the sequences of polynomials $\{p_k(x;y)\}_{k=0}^\infty$ 
in \eqref{eqn:seq0}:
\vskip 0.1in
\begin{itemize}
\item Factorization formula of $p_{d(n)+1}(x;n)$;
\vskip 0.05in

\item 
Palindromic property of $\{p_k(x;y)\}_{k=0}^{\infty}$;
\vskip 0.05in

\item 
Factorial identity of $p_{d(n)}(x;n)$;

\item Refinement of a factorial identity of $p_{d(n)}(x;n)$.
\end{itemize} 
\vskip 0.1in
It is remarked that factorization formulas for 
$\{P_k(x;y)\}_{k=0}^\infty$ and 
$\{Q_k(x;y)\}_{k=0}^\infty$ will be given in a more general situation.
We are going to describe these topics in detail now.

\subsection{
Sequences
$\{P_k(x;y)\}_{k=0}^{\infty}$ and $\{Q_k(x;y)\}_{k=0}^{\infty}$
of tridiagonal determinants}

Recall from \eqref{eqn:DF0} that 
the local Heun function $u_{[s;n]}(t)$ in \eqref{eqn:Heun_intro}
is a local solution at $t=0$
to the Heun differential equation \eqref{eqn:HeunOp_intro}.
Let $v_{[s;n]}(t)$ be the second solution (see \eqref{eqn:v}).
It will be shown in Proposition \ref{prop:PQ} that
$u_{[s;n]}(t)$ and $v_{[s;n]}(t)$ may be given in
power series representations as
\begin{equation*}
u_{[s;n]}(t)=\sum_{k=0}^\infty P_k(s;n)\frac{t^{2k}}{(2k)!}
\qquad \text{and} \qquad 
v_{[s;n]}(t)=\sum_{k=0}^\infty Q_k(s;n)\frac{t^{2k+1}}{(2k+1)!},
\end{equation*}
where $P_k(x;y)$ and $Q_k(x;y)$ are certain $k \times k$ tridiagonal determinants
(see Section \ref{subsec:PQ}). 
For instance, 
the $\frac{n}{2}\times \frac{n}{2}$ tridiagonal determinant
$P_{\frac{n}{2}}(x;n)$ for $y=n \in 2\Z_{\geq 0}$
 is given as 
$P_0(x;0)=1$, $P_1(x;2)=2x$, and for $n \in 4+2\Z_{\geq 0}$, 

\begin{equation}\label{eqn:P_intro}
\begin{small}
P_{\frac{n}{2}}(x;n) 
= 
\begin{vmatrix}
nx& 1\cdot 2    &              &            &             &       \\
-(n-1)n& (n-4)x    &    3\cdot 4       &            &             &       \\
  & -(n-3)(n-2) &     (n-8)x       &  5\cdot 6        &             &       \\
  &       &  \dots  &  \dots & \dots  &       \\
  &       &              &-5 \cdot 6   &-(n-8)x            & (n-3)(n-2) \\
  &       &              &            &-3\cdot 4     & -(n-4)x    \\
\end{vmatrix}.\\[5pt]
\end{small}
\end{equation}

\noindent
Similarly,
the $\frac{n-2}{2}\times \frac{n-2}{2}$ tridiagonal determinant
$Q_{\frac{n-2}{2}}(x;n)$ for $y=n \in 2(1+\Z_{\geq 0})$ is given as 
$Q_0(x;2)=1$,
$Q_1(x;4)=2x$, and
for $n \in 6+2\Z_{\geq 0}$, 

\begin{equation}\label{eqn:Q_intro}
\begin{small}
Q_{\frac{n-2}{2}}(x;n) = 
\begin{vmatrix}
(n-2)x& 2\cdot 3    &              &            &             &       \\
-(n-2)(n-1)& (n-6)x    &     4\cdot 5       &            &             &       \\
  & -(n-4)(n-3) &     (n-10)x       &  6 \cdot 7        &             &       \\
  &       &  \dots  &  \dots & \dots  &       \\
  &       &              &-6\cdot 7  &-(n-10)x           & (n-4)(n-3) \\
  &       &              &            &-4\cdot 5     & -(n-6)x    \\
\end{vmatrix}.\\[5pt]
\end{small}
\end{equation}
\noindent

\subsection{Factorization formulas of 
$P_{[(n+2)/2]}(x;n)$ and $Q_{[(n+1)/2]}(x;n)$}

In 1854, Sylvester observed that an $(n+1)\times (n+1)$ 
centrosymmetric tridiagonal determinant

\begin{equation}\label{eqn:Sylv_intro}
\Sylv(x;n) :=
\begin{small}
\begin{vmatrix}
x& 1    &              &            &             &       &\\
n& x    &     2       &            &             &       &\\
  & n-1 &     x       &  3        &             &       &\\
  &       &  \dots  &  \dots & \dots  &       &\\
  &       &              &3          &x            & n-1 &\\
  &       &              &            &2            & x    & n\\
  &       &              &            &              &  1   & x
\end{vmatrix}\\[5pt]
\end{small}
\end{equation}
satisfies the following  formula (\cite{Sylv54}).

\begin{equation}\label{eqn:Sylv0}
\begin{aligned}
\Sylv(x;n)
=\begin{cases}
\;\;(x^2-1^2)(x^2-3^2)\cdots(x^2-(n-2)^2)(x^2-n^2) &\text{if $n$ is odd},\\
x(x^2-2^2)(x^2-4^2)\cdots(x^2-(n-2)^2)(x^2-n^2) &\text{if $n$ is even}.\\
\end{cases}
\end{aligned}
\end{equation}
\vskip 0.05in
\noindent
By utilizing some results 
for the Heun model 
$d\pi_n^\textnormal{I}(D^\flat_0)p(t)=0$,
we show that 
$P_{[(n+2)/2]}(x;n)$ and $Q_{[(n+1)/2]}(x;n)$ enjoy similar but 
more involved factorization formulas.
See Theorems \ref{thm:factorPQ} and \ref{thm:factorP} for the details.
Based on the factorization formula, we also express
$P_{(n+1)/2}(x;n)$ for $n$ odd in terms of the Sylvester determinant
$\Sylv(x;n)$ (Corollary \ref{cor:PS}).

Here are some remarks on the factorization formulas in order.
First, the factorization formulas
$P_{(n+2)/2}(x;n)$ and 
$Q_{n/2}(x;n)$ for $n\in 2\Z$
also follow from a more general formula in \cite[Prop.\ 5.11]{Kable12C},
which is obtained via some clever trick based on a change of variables
and some general formula on tridiagonal determinants
(see, for instance, \cite[p.\ 52]{Clement59} and \cite[Lem.\ 5.10]{Kable12C}).
As mentioned in the remark after Theorem \ref{thm:K-type0} above,
the factorization formulas 
for $P_{(n+2)/2}(x;n)$ and $Q_{n/2}(x;n)$
 play an essential role in \cite{Kable12C}
to determine $\dim_\C\Sol(s;n)$. 
In this paper we somewhat do this process backwards including the case of 
$n$ odd. 
Schematically, the difference is described as follows.
\vspace{3pt}
\begin{itemize}
\item \cite{Kable12C}: 
\qquad \quad
``factorization $\to$ $\dim_\C\Sol(s;n)$'' 
\quad \,
($n \in 2\Z_{\geq 0}$)
\vspace{3pt}
\item this paper: ``$\Sol(s;n)$ $\to$ Heun $\to$ factorization''
($n \in \Z_{\geq 0}$)
\end{itemize}
\vspace{3pt}
(It is recalled that,
as the linear group $SL(m,\R)$ is considered in \cite{Kable12C}, 
only even $n \in 2\Z_{\geq 0}$ appear for the space
$\Sol(s;n)$ of $K$-type solutions for $SO(3) \subset SL(3,\R)$ 
in the case of $m=3$, 
whereas we handle all $n \in \Z_{\geq 0}$ in this paper,
as $Spin(3) \subset \wSL(3,\R)$ is in consideration.)

The techniques used in \cite{Kable12C} can also be applied for the case 
of $n$ odd. Nevertheless, it requires some involved computations;
for instance, the trick used in the cited paper 
does not make the situation as simple as 
for the case of $n$ even, and further, the general formula 
 (\cite[p.\ 52]{Clement59} and \cite[Lem.\ 5.10]{Kable12C})
cannot be applied either. By simply applying the results 
for the Heun model $d\pi_n^\textnormal{I}(D^\flat_0)p(t)=0$,
we successfully avoid such computations.

\subsection{Palindromic properties for
$\{P_k(x;y)\}_{k=0}^{\infty}$ and $\{Q_k(x;y)\}_{k=0}^{\infty}$}
\label{subsec:intro_palin}

We next describe the palindromic properties 
and factorial identities
for $\{P_k(x;y)\}_{k=0}^{\infty}$ and $\{Q_k(x;y)\}_{k=0}^{\infty}$.
The palindromic properties are given as follows.

\begin{thm}
[Theorems \ref{thm:PalinP} and \ref{thm:PalinQ}]
\label{thm:intro_PalinPQ}
The pair $(\{P_k(x;y)\}_{k=0}^\infty, \{(2k)!\}_{k=0}^\infty)$ 
is a palindromic pair 
with degree $\frac{n}{2}$.
Similarly, 
$(\{Q_k(x;y)\}_{k=0}^\infty,\{(2k+1)!\}_{k=0}^\infty)$ 
is a palindromic pair 
with degree $\frac{n-2}{2}$. 
\end{thm}
For the case for $n$ odd, see Propositions 
\ref{prop:PalinPodd} and \ref{prop:PalinQodd}. 
The key idea of proofs of Theorem \ref{thm:intro_PalinPQ} 
is to use a symmetry of the generating functions
$u_{[s;n]}(t)$ and $v_{[s;n]}(t)$
with respect to the non-trivial Weyl group element 
$m^{\textnormal{I}}_2$ (\eqref{eqn:mKO}) of $SU(2)$
in the form of the inversion.
See Section \ref{subsec:PalinPQ} for the details.

It follows from Theorem \ref{thm:intro_PalinPQ} that
$P_{\frac{n}{2}}(x;y)$ (\eqref{eqn:P_intro})
and 
$Q_{\frac{n-2}{2}}(x;y)$ (\eqref{eqn:Q_intro})
satisfy
factorial identities (see Corollaries \ref{cor:PalinP} and \ref{cor:PalinQ}).
Corollaries \ref{cor:P0} and \ref{cor:Q0} 
below are the refinements of the identities
via the factorization formulas 
(Theorem \ref{thm:factorPQ})
of $P_{\frac{n+2}{2}}(x;n)$ and $Q_{\frac{n}{2}}(x;n)$.

\begin{cor}[Refinement of the factorial identity of $P_{\frac{n}{2}}(x;n)$]
\label{cor:P0}
Let $n \in 2\Z_{\geq 0}$.
Then the following hold.
\begin{enumerate}[{\quad \normalfont (1)}]
\item $n \equiv 0 \Mod 4:$
The values of 
$P_{\frac{n}{2}}(s;n)$ are
\begin{equation*}
P_{\frac{n}{2}}(s;n)=n! \quad \textnormal{for all $s \in \C$}.
\end{equation*}

\item $n \equiv 2 \Mod 4:$
The values of 
$P_{\frac{n}{2}}(s;n)$ 
for $s=\pm1, \pm 5,\ldots, \pm (n-2)$ are given as 
\begin{equation}\label{eqn:PalinPintro}
P_{\frac{n}{2}}(s;n) 
= \begin{cases}
n! &\textnormal{if $s = 1, 5, 9 \dots, n-2$},\\
-n! &\textnormal{if $s = -1, -5, -9 \dots, -(n-2)$}.
\end{cases}
\end{equation}
\end{enumerate}
\end{cor}

\begin{cor}[Refinement of the factorial identity
of $Q_{\frac{n-2}{2}}(x;n)$]
\label{cor:Q0}
Let $n \in 2(1+\Z_{\geq 0})$.
Then the following hold.
\begin{enumerate}[{\quad \normalfont (1)}]
\item 
$n \equiv 0 \Mod 4:$ 
The values of
$Q_{\frac{n-2}{2}}(s;n)$
for $s=\pm3, \pm 7,\ldots, \pm (n-1)$ are given as
\begin{equation}\label{eqn:PalinQintro}
Q_{\frac{n-2}{2}}(s;n) 
= \begin{cases}
(n-1)! &\textnormal{if $s = 3, 7, 11, \ldots, n-1$},\\
-(n-1)! &\textnormal{if $s=-3, -7, -11, \ldots, -(n-1)$}.
\end{cases}
\end{equation}

\item $n \equiv 2 \Mod 4:$ 
The values of
$Q_{\frac{n-2}{2}}(s;n)$ are
\begin{equation*}
Q_{\frac{n-2}{2}}(s;n) = (n-1)! 
\quad \textnormal{for all $s \in \C$.}
\end{equation*}
\end{enumerate}
\end{cor}

We shall give proofs of Corollaries \ref{cor:P0} and \ref{cor:Q0} 
in Sections \ref{subsubsec:PalinP} and \ref{subsubsec:PalinQ}.

\subsection{Sequence $\{\Cay_k(x;y)\}_{k=0}^\infty$ of Cayley continuants}
\label{subsec:introCayley}
In 1858,
four years after the observation of Sylvester on \eqref{eqn:Sylv0},
Cayley considered in \cite{Cayley58} 
a sequence $\{\Cay_k(x;y)\}_{k=0}^\infty$
of $k \times k$ tridiagonal determinants $\Cay_k(x;y)$
and expressed each $\Cay_k(x;n)$ in terms of 
the Sylvester determinant $\Sylv(x;n)$ 
(see, for instance, \cite{Cayley58, MT05} and \cite[p.\ 429]{Muir06})).
The first few terms are given as
\begin{equation*}
\Cay_0(x;y)=1,\quad
\Cay_1(x;y)=x, \quad
\Cay_2(x;y)=
\begin{vmatrix}
x & 1\\
y & x
\end{vmatrix}, \quad
\Cay_3(x;n)=
\begin{small}
\begin{vmatrix}
x & 1 &\\
y & x &2\\
   &y-1 &x
\end{vmatrix}, \quad
\end{small}
\ldots.
\end{equation*} 
\noindent
The $(n+1)$th term $\Cay_{n+1}(x;n)$ with $y=n$ is nothing but
the Sylvester determinant
$\Cay_{n+1}(x;n)=\Sylv(x;n)$.
Further,
the $n$th term $\Cay_{n}(x;n)$ may be thought of as
the ``almost'' Sylvester determinant,
as it is given as
\begin{equation}\label{eqn:Cayn}
\Cay_{n}(x;n) =
\begin{small}
\begin{vmatrix}
x& 1    &              &            &             &      \\
n& x    &     2       &            &             &       \\
  & n-1 &     x       &  3        &             &       \\
  &       &  \dots  &  \dots & \dots  &       \\
  &       &              &3          &x            & n-1 \\
  &       &              &            &2            & x    
\end{vmatrix}\\[5pt]
\end{small}
\end{equation}
\vskip 0.05in
\noindent
(compare \eqref{eqn:Cayn} with \eqref{eqn:Sylv_intro}).
Following \cite{MT05}, we refer to each
$\Cay_k(x;y)$ 
as a Cayley continuant. 

As part of the third aim of this paper,
we shall show 
the palindromic property of
the Cayley continuants 
$\{\Cay_k(x;y)\}_{k=0}^\infty$ as follows.

\begin{thm}
[Theorem \ref{thm:recip}]
\label{thm:intro_PalinC}
The pair
$(\{\Cay_k(x;y)\}_{k=0}^\infty, \{k!\}_{k=0}^\infty)$
is a palindromic pair 
with degree $n$.
\end{thm}

For the palindromic identity, 
see Theorem \ref{thm:recip}.
As for the cases of 
$\{P_k(x;y)\}_{k=0}^\infty$ and 
$\{Q_k(x;y)\}_{k=0}^\infty$,
we prove Theorem \ref{thm:intro_PalinC} 
by observing a symmetry of the generating function
of the Cayley continuants $\{\Cay_k(x;y)\}_{k=0}^\infty$
with respect to the inversion
(see Section \ref{subsec:PalinC}).

Theorem \ref{thm:intro_PalinC} implies that
the almost Sylvester determinant
$\Cay_n(x;y)$ (\eqref{eqn:Cayn}) also satisfies a factorial identity.
Corollary \ref{cor:Cay0} below is
the refinement of the factorial identity 
(see Theorem \ref{thm:recip}) of $\Cay_n(x;n)$
on $\ESol_{n+1}(\Cay;n)$
via Sylvester's factorization formula \eqref{eqn:Sylv0}
of $\Sylv(x;n)=\Cay_{n+1}(x;n)$.

\begin{cor}
[Refinement of the factorial identity of $\Cay_n(x;n)$]
\label{cor:Cay0}
For $n$ even, 
the values of the almost Sylvester determinant $\Cay_{n}(s;n)$
for $s = 0, \pm 2, \dots, \pm (n-2), \pm n$
are given as follows.

\begin{itemize}
\item
$n\equiv 0 \Mod 4:$
\begin{equation*}
\Cay_n(s;n) 
= \begin{cases}
n! &\textnormal{if $s = 0, \pm 4, \dots, \pm (n-4), \pm n$},\\
-n! &\textnormal{if $s=\pm 2, \pm 6, \dots, \pm(n-2)$}.
\end{cases}
\end{equation*}
\vskip 0.1in

\item
$n\equiv 2 \Mod 4:$ 
\begin{equation*}
\Cay_n(s;n) 
= \begin{cases}
-n! &\textnormal{if $s= 0, \pm 4, \dots, \pm (n-4), \pm n$},\\
n! &\textnormal{if $s=\pm 2, \pm 6, \dots, \pm(n-2)$}.
\end{cases}
\end{equation*}
\end{itemize}
\vskip 0.1in
\noindent
Similarly, 
for $n$ odd, 
the values of $\Cay_{n}(s;n)$
for $s = \pm1, \pm 3, \dots, \pm (n-2), \pm n$
are given as follows.

\begin{itemize}
\item
$n\equiv 1 \Mod 4:$ 
\begin{equation*}
\Cay_n(s;n) 
= \begin{cases}
n! &\textnormal{if $s =1, -3, 5, \dots, -(n-2),n $},\\
-n! &\textnormal{if $s=-1, 3, -5, \dots, n-2,-n$}.
\end{cases}
\end{equation*}

\item
$n\equiv 3 \Mod 4:$
\begin{equation*}
\Cay_n(s;n) 
= \begin{cases}
-n! &\textnormal{if $s =1, -3, 5, \dots, n-2,-n $},\\
n! &\textnormal{if $s=-1, 3, -5, \dots, -(n-2),n$}.
\end{cases}
\end{equation*}
\end{itemize}
\end{cor}

We shall give a proof of Corollary \ref{cor:Cay0} at the end of 
Section \ref{subsec:PalinC}.
\vskip 0.1in

It has been more than 160 years since Cayley continuants 
$\{\Cay_k(x;y)\}_{k=0}^\infty$
were introduced in \cite{Cayley58}. 
We therefore suppose that 
Theorem \ref{thm:intro_PalinC} and Corollary \ref{cor:Cay0} 
are already in the literature.
Yet, in contrast to a large variety of 
a proof for Sylvester's formula \eqref{eqn:Sylv0} 
(see the remark after Theorem \ref{thm:factorP}),
we could not find them at all.
It will be quite surprising if the classical identities in
Corollary \ref{cor:Cay0} have not been found over a century.

It is also remarked that Sylvester's formula \eqref{eqn:Sylv0} readily follows
from the representation theory of $\f{sl}(2,\C)$ without any computation.
Nonetheless, it seems not in the literature either.
We shall then provide a proof of \eqref{eqn:Sylv0}
from a representation theory point of view
(see Section \ref{sec:appendixB}).
The idea of the proof is in principle the same as
the one discussed in Section \ref{subsec:method}
for the classification of $K$-type formulas.
Furthermore, by applying the idea for determining the generating
functions of $\{P(x;y)\}_{k=0}^\infty$ and $\{Q(x;y)\}_{k=0}^\infty$,
we also give in Proposition \ref{prop:genCay8}
another proof for 
Sylvester's formula \eqref{eqn:Sylv0} as well as
the generating function of 
the Cayley continuants $\{\Cay_k(x;y)\}_{k=0}^\infty$.
 (Thus we give two different 
proofs of Sylvester's formula \eqref{eqn:Sylv0} in this paper.)

\subsection{Sequence $\{\Cal{K}_k(x;y)\}_{k=0}^\infty$ of Krawtchouk polynomials}
\label{subsec:introKraw}
For $k \in \Z_{\geq 0}$, 
let $\Cal{K}_k(x;y)$ be a polynomial of homogeneous degree $k$ 
such that
$\Cal{K}_k(x;n)$ for $y=n \in \Z_{\geq 0}$ is a Krawtchouk polynomial
in the sense of \cite[p.\ 137]{MS77}, 
namely,
\begin{equation}\label{eqn:introKraw}
\Cal{K}_k(x;y) =\sum_{j=0}^k(-1)^j \binom{x}{j}\binom{y-x}{k-j}
\end{equation}
(see Definition \ref{def:Kraw}). 
In this paper we also call the polynomials $\Cal{K}_k(x;y)$ of two variables
Krawtchouk polynomials.
From the observation of 
the generating functions of $\Cay_k(x;y)$ and $\Cal{K}_k(x;y)$, 
it is immediate that 
\begin{equation}\label{eqn:CK0}
\Cal{K}_k(x;y) = \frac{\Cay_k(y-2x;y)}{k!}
\end{equation}
(see \eqref{eqn:generateU} and \eqref{eqn:genKraw}).
In the last part of this paper, via the identity \eqref{eqn:CK0}, 
we deduce the palindromic property of the Krawtchouk polynomials  
$\{\Cal{K}_k(x;y)\}_{k=0}^\infty$ from that of the Cayley continuants 
 $\{\Cay_k(x;y)\}_{k=0}^\infty$ as follows.

\begin{thm}
[Theorem \ref{thm:PalinK}]
\label{thm:intro_PalinK0}
The pair
$(\{\Cal{K}_k(x;y)\}_{k=0}^\infty, \{1\}_{k=0}^\infty)$
is a palindromic pair
with degree $n$.
\end{thm}

See Theorem \ref{thm:PalinK} for the palindromic identity.
As the associated sequence $\{1\}_{k=0}^\infty$ for 
$\{\Cal{K}_k(x;y)\}_{k=0}^\infty$ is of the form
$1 = 1!$ for all $k$, the $n$th term
$\Cal{K}_n(x;y)$ satisfies a factorial identity.
Corollary \ref{cor:Kraw0} below is the refinement of 
the factorial identity 
of $\Cal{K}_n(x;n)$ on $\ESol_{n+1}(\Cal{K},n)$
(see Theorem \ref{thm:PalinK}).

\begin{cor}
[Refinement of the factorial identity of $\Cal{K}_n(x;n)$]
\label{cor:Kraw0}
Given $n \in \Z_{\geq 0}$,
the values of $\Cal{K}_{n}(s;n)$
for $s = 0, 1, 2,\dots, n$ are given as follows.
\begin{equation}\label{eqn:factorialKraw0}
\Cal{K}_n(s;n) 
= \begin{cases}
1 &\textnormal{if $s = 0,\, 2,\, 4,\, \dots,\, n_{\textrm{even}}$},\\
-1 &\textnormal{if $s=1,\, 3,\, 5,\, \dots,\, n_{\textrm{odd}}$},
\end{cases}
\end{equation}
where $n_{\textrm{odd}}$ and $n_{\textrm{even}}$ are defined as
\begin{equation*}
n_{\textnormal{odd}}:=
\begin{cases}
n & \textnormal{if $n$ is odd},\\
n-1 & \textnormal{if $n$ is even}
\end{cases}
\qquad
\textnormal{and}
\qquad
n_{\textnormal{even}}:=
\begin{cases}
n-1 & \textnormal{if $n$ is odd},\\
n & \textnormal{if $n$ is even}.
\end{cases}
\end{equation*}
\end{cor}

We shall give a proof of Corollary \ref{cor:Kraw0} at the end of 
Section \ref{subsec:PalinK} as a corollary of Theorem \ref{thm:PalinK}. 
We remark that the identities 
\eqref{eqn:factorialKraw0} can also be easily 
shown directly from the definition \eqref{eqn:introKraw}
by an elementary observation.

The Krawtchouk polynomials $\{\Cal{K}_k(x;y)\}_{k=0}^\infty$ 
are also a classical object at least for the case of $y=n \in \Z_{\geq 0}$.
Therefore Theorem \ref{thm:intro_PalinK0} 
may be a known as, for instance, an exercise problem; 
nonetheless, we could not find it in the literature.

\subsection{Summary of the data for palindromic properties}
To close this introduction, we summarize some data 
on the palindromic property of the sequence
$\{p_k(x;y)\}_{k=0}^\infty$
of polynomials in \eqref{eqn:seq0}.
For the associated sequence $\{a_k\}_{k=0}^\infty$ 
of $\{p_k(x;y)\}_{k=0}^\infty$,
we write
\begin{equation*}
g_{[x;y]}(t) := \sum_{k=0}^\infty p_k(x;y) \frac{t^{b_k}}{a_k},
\end{equation*}
\vskip 0.05in
\noindent
where $\{b_k\}_{k=0}^\infty$ is some sequence of non-negative integers.
Table \ref{table:tableCPQ} exhibits
the generating function $g_{[x;y]}(t)$, 
the associated sequence $\{a_k\}_{k=0}^\infty$,
the sequence $\{b_k\}_{k=0}^\infty$ of exponents,
the degree $d(n)$,
and 
the sign factor $\gt_{(p;n)}(s)$
for $\{p_k(x;y)\}_{k=0}^\infty$
in \eqref{eqn:seq0}.
(For the sign factors $\gt_{(P;n)}(s)$ and $\gt_{(Q;n)}(s)$, see \eqref{eqn:thetaP} 
and \eqref{eqn:thetaQ}, respectively.)
\vspace{-10pt}
\begin{table}[H]
\caption{Summary on palindromic properties}
\begin{center}
\renewcommand{\arraystretch}{1.5}
{
\begin{tabular}{c||c|c|c|c|c}
\hline
$p_k(x;y)$ &$g_{[x;y]}(t)$& $a_k$ & $b_k$ & $d(n)$
& $\gt_{(p;n)}(s)$\\
\hline
$\Cal{K}_k(x;y)$ & $(1+t)^{y-x}(1-t)^x$ & $1$ &  $k$ & $n$
& $(-1)^s$\\
\hline
$\Cay_k(x;y)$ & $(1+t)^{\frac{y+x}{2}}(1-t)^{\frac{y-x}{2}}$ & $k!$ &  $k$ & $n$
& $(-1)^{\frac{n-s}{2}}$\\
\hline
$P_k(x;y)$ & $u_{[x;y]}(t)$ \eqref{eqn:u} & $(2k)!$ & $2k$ &\, $\frac{n}{2}$ \quad  & \eqref{eqn:thetaP} \\
\hline 
$Q_k(x;y)$ &$v_{[x;y]}(t)$ \eqref{eqn:v} & $(2k+1)!$ & $2k+1$ & $\frac{n-2}{2}$ & \eqref{eqn:thetaQ}\\
\hline 
\end{tabular}
}
\end{center}
\label{table:tableCPQ}
\end{table}

\vspace{-10pt}

\subsection{Organization}
We now outline the rest of this paper.
This paper consists of ten sections including this introduction.
In Section \ref{sec:PW}, we overview a general framework
established in \cite{KuOr19} for
the Peter--Weyl theorem \eqref{eqn:PWSL_intro} 
for the space of $K$-finite solutions to intertwining differential operators. 
In Section \ref{sec:SL3}, 
we collect necessary notation and normalizations for $\wSL(3,\R)$.
Two identifications 
$\Omega^J\colon \fk \stackrel{\sim}{\to} \f{sl}(2,\C)$
for $J = \textnormal{I}, \textnormal{II}$ are discussed in this section.

The purpose of Section \ref{sec:HO} is to recall from \cite{Kable12C}
the Heisenberg 
ultrahyperebolic operator $\square_s = R(D_s)$ for $\wSL(3,\R)$ and 
to study the associated differential equation $d\pi_n^J(D^\flat_s)p(t)=0$.
In this section we identify the equation $d\pi_n^J(D^\flat_s)p(t)=0$ with
Heun's differential equation ($J=\textnormal{I}$) 
and the hypergeometric differential equation ($J=\textnormal{II}$)
via the identifications 
$\Omega^J\colon \fk \stackrel{\sim}{\to} \f{sl}(2,\C)$.
At the end of this section we give a recipe 
to classify the space $\Sol(s;n)$ of $K$-type solutions 
to  $\square_s = R(D_s)$. 

In accordance with the recipe given in Section \ref{sec:HO},
we show the $K$-type decompositions
of $\CSol(s;\sigma)_K$ in the hypergeometric model 
$d\pi_n^{\textnormal{II}}(D^\flat_s)p(t)=0$ in Section \ref{sec:HGE}. 
This is accomplished in Theorem \ref{thm:K-type}.
We then convert the results in the hypergeometric 
model 
$d\pi_n^{\textnormal{II}}(D^\flat_s)p(t)=0$
to the Heun model 
$d\pi_n^{\textnormal{I}}(D^\flat_s)p(t)=0$
by a Cayley transform $\pi_n(k_0)$
in Section \ref{sec:Heun}.

Sections \ref{sec:PQ} and \ref{sec:U}
are devoted to applications to sequences of polynomials.
In Section \ref{sec:PQ}, by utilizing
the results in the Heun model 
$d\pi_n^{\textnormal{I}}(D^\flat_s)p(t)=0$, 
we investigate two sequences 
$\{P_k(x;n)\}_{k=0}^{\infty}$ and $\{Q_k(x;n)\}_{k=0}^{\infty}$ 
of tridiagonal determinants.
The main results of this sections are factorization formulas 
(Theorems \ref{thm:factorPQ} and \ref{thm:factorP}) and 
palindromic properties (Theorems \ref{thm:PalinP} and \ref{thm:PalinQ}).
We also show an expression of $P_{\frac{n+1}{2}}(x;n)$ for $n$ odd
in terms of the Sylvester determinant $\Sylv(x;n)$ (Corollary \ref{cor:PS}).
In Section \ref{sec:U}, we show the palindromic properties 
for Cayley continuants 
$\{\Cay_k(x;y)\}_{k=0}^\infty$ and 
Krawtchouk polynomials $\{\Cal{K}_k(x;y)\}_{k=0}^\infty$. 
These are achieved in Theorems \ref{thm:recip} and \ref{thm:PalinK},
respectively.

The last two sections are appendices.
In order to study
$\{P_k(x;n)\}_{k=0}^{\infty}$ and $\{Q_k(x;n)\}_{k=0}^{\infty}$,
we use some general facts on 
the coefficients of a power series expression of local Heun functions.
We collect those facts in Section \ref{sec:appendix}.
In Section \ref{sec:appendixB}, for future possible convenience, 
we give a proof of Sylvester's formula \eqref{eqn:Sylv0}
from an $\f{sl}(2,\C)$ point of view.


\section{Peter--Weyl theorem for the space of $K$-finite solutions}\label{sec:PW}

The aim of this section is to recall from \cite[Sect.\ 2]{KuOr19}
a general framework established in \cite{KuOr19}.
In particular, we give a Peter--Weyl theorem for the space 
of $K$-finite solutions to intertwining differential operators
between parabolically induced representations. 
This is done in Theorem \ref{thm:PW1}.

\subsection{General framework}
Let $G$ be a reductive Lie group with Lie algebra $\fg_0$. 
Choose a Cartan involution $\gt: \fg_0 \to \fg_0$ 
and write $\fg_0 = \fk_0 \oplus \fs_0$
for the Cartan decomposition of $\fg_0$ with respect to $\gt$. 
Here $\fk_0$ and $\fs_0$ stand for the $+1$ and $-1$ eigenspaces of $\gt$, respectively.
We take maximal abelian subspaces $\fa^{\min}_0 \subset \fs_0$
and $\ft_0^{\min} \subset \fm_0^{\min} := Z_{\fk_0}(\fa_0^{\min})$
such that $\fh_0 :=\fa_0^{\min}  \oplus  \ft_0^{\min}$ 
is a Cartan subalgebra of $\fg_0$.

For a real Lie algebra $\f{y}_0$,
we denote by $\f{y}$ the complexification of $\f{y}_0$.
For instance, 
the complexifications of  
$\fg_0$, $\fh_0$, $\fa_0^{\min}$, and $\fm^{\min}_0$
are denoted by $\fg$, $\fh$, $\fa$, and $\fm^{\min}$, respectively.
We write $\Cal{U}(\f{y})$ for the universal enveloping algebra 
of a Lie algebra $\f{y}$.

Let $\gD \equiv \gD(\fg, \fh)$ be the set of roots 
with respect to the Cartan subalgebra $\fh$
and $\gS \equiv \gS(\fg_0, \fa_0)$ denote
the set of restricted roots with respect to $\fa_0$.
We choose a positive system $\gD^+$ and $\gS^+$ in such a way that
$\gD^+$ and $\gS^+$ are compatible. 
We denote by $\rho$ for half the sum of the positive roots.

Let $\fn^{\min}_0$ be the nilpotent subalgebra of 
$\fg_0$ corresponding to $\gS^+$,
so that  $\fp^{\min}_0:=\fm^{\min}_0 \oplus \fa^{\min}_0 \oplus \fn^{\min}_0$ 
is a Langlands decomposition of a minimal parabolic subalgebra 
$\fp_0^{\min}$ of $\fg_0$. 
Fix a standard parabolic subalgebra $\fp_0 \supset \fp_0^{\min}$ with 
Langlands decomposition $\fp_0 = \fm_0 \oplus \fa_0 \oplus \fn_0$.
Let $P$ be a parabolic subalgebra of $G$ with Lie algebra $\fp_0$.
We write $P=MAN$ for the Langlands decomposition of $P$ corresponding to 
$\fp_0=\fm_0\oplus \fa_0 \oplus \fn_0$.

For $\mu \in \fa^* \simeq \Hom_\R(\fa_0,\C)$,
we define a one-dimensional representation $\C_{\mu}$ of $A$ as
$a \mapsto e^{\mu}(a):= e^{\mu(\log a)}$ for $a\in A$.
Then, for a finite-dimensional representation 
$W_\sigma = (\sigma, W)$ of $M$ and weight $\lambda \in \fa^*$,
we define an $MA$-representation $W_{\sigma,\lambda}$ as
\begin{equation*}
W_{\sigma,\lambda}:=W_\sigma \otimes \C_{\lambda}.
\end{equation*}
(We note that the definition  
of $W_{\sigma,\lambda}$ was slightly different from the one 
in \cite{KuOr19}.)
As usual, by letting $N$ act on 
$W_{\sigma,\lambda}$ trivially, we regard
$W_{\sigma,\lambda}$ as a representation of $P$.
We identify 
the Fr\' echet space 
$C^\infty\left(G/P, \mathcal{W}_{\sigma,\lambda}\right)$ 
of smooth sections for the
$G$-equivariant homogeneous vector bundle
$\mathcal{W}_{\sigma,\lambda}:=
G\times_{P}W_{\sigma,\lambda} \to G/P$
as $C^\infty\left(G/P, \mathcal{W}_{\sigma,\lambda}\right) 
\simeq (C^\infty(G)\otimes W_{\sigma,\lambda})^P$, that is,
\begin{align*}
C^\infty\left(G/P, \mathcal{W}_{\sigma,\lambda}\right) 
\simeq \left\{f\in C^\infty(G)\otimes W_{\sigma,\lambda}: 
f(gman) = \sigma(m)^{-1}e^{-\lambda}(a)f(g)\; 
\text{for all $man \in MAN$}
\right\},
\end{align*}
where $G$ acts by left translation.
Then we realize a parabolically induced representation
\begin{equation}\label{eqn:Ind}
I_P(\sigma, \lambda) 
:= \text{Ind}_P^G(\sigma\otimes(\lambda+\rho)\otimes \mathbb{1})
\end{equation}
of $G$ on 
$C^\infty\left(G/P, \mathcal{W}_{\sigma,\lambda+\rho}\right)$.

For pairs $(\sigma, \lambda), (\eta,\nu)$
of finite-dimensional representations $\sigma,\eta$ of $M$
and weights $\lambda,\nu \in \fa^*$, we write
$\mathrm{Hom}_G(I_P(\sigma, \lambda), I_P(\eta,\nu))$ 
for the space of intertwining operators from
$I_P(\sigma, \lambda)$ to $I_P(\eta, \nu)$.
Then we set
\begin{equation*}
\mathrm{Diff}_{G}
(I_P(\sigma, \lambda), I_P(\eta,\nu)):=
\mathrm{Diff}
(I_P(\sigma, \lambda), I_P(\eta,\nu)) 
\cap
\mathrm{Hom}_G
(I_P(\sigma, \lambda), I_P(\eta,\nu)),
\end{equation*}
where 
$\mathrm{Diff}
(I_P(\sigma, \lambda), I_P(\eta,\nu)) $
is the space of differential operators from 
$I_P(\sigma, \lambda)$ to $I_P(\eta, \nu)$.

Let $\Irr(M)_{\fin}$ be the set of equivalence classes of
irreduicible
finite-dimensional representations of $M$.
As $M$ is not connected in general,
we write $M_0$ for the identity component of $M$.
Let $\Irr(M/M_0)$ denote the set of 
irreducible representations of  the component group $M/M_0$.
Via the surjection $M \twoheadrightarrow M/M_0$, 
we regard $\Irr(M/M_0)$ as a subset of $\Irr(M)_{\fin}$.
Let $\xid$ denote the identity map on $\xi \in \Irr(M/M_0)$.
Lemma \ref{lem:tensor} below shows that the tensored operator 
$\D\otimes \id_\xi$ to $\D \in \mathrm{Diff}_{G}
(I_P(\sigma, \lambda), I_P(\eta,\nu)$
is also an intertwining differential operator.

\begin{lem}[{\cite[Lems.\ 2.17, 2.21]{KuOr19}}]\label{lem:tensor}
Let $\D \in 
\Diff_G(I_P(\sigma, \lambda), I_P(\eta,\nu))$. 
For any $\xi \in \Irr(M/M_0)$, we have
\begin{equation*}
\D \otimes \id_{\xi} 
\in 
\Diff_G(I_P(\sigma\otimes \xi, \lambda), I_P(\eta\otimes \xi,\nu)).
\end{equation*}
\end{lem}

If the $M$-representation $\sigma$ is the trivial character
$\sigma=\chi_\triv$, then Lemma \ref{lem:tensor} in particular shows that
differential operator 
$\D \in \Diff_G(I_P(\chi_{\triv}, \lambda), I_P(\eta,\nu))$ 
yields 
\begin{equation}\label{eqn:dxi}
\D \otimes \id_{\xi} 
\in \Diff_G(I_P(\xi, \lambda), I_P(\eta\otimes \xi,\nu))
\quad
\text{for  $\xi \in \Irr(M/M_0)$}.
\end{equation}

\subsection{Peter--Weyl theorem for $\Cal{S}ol_{(u; \lambda)}(\xi)_K$}
For a later purpose,
we specialize 
the representations
$\sigma, \eta$ of $M$ 
to be $\sigma= \chi_{\triv}$ and $\eta = \chi$,
where $\chi$ is some character of $M$. 
In this situation, it follows from
\eqref{eqn:dxi} that, for
$\D \in \Diff_G(I_P(\chi_{\triv}, \lambda), I_P(\chi,\nu))$ 
and $\xi \in \Irr(M/M_0)$,
we have
\begin{equation*}
\D \otimes \id_{\xi} 
\in \Diff_G(I_P(\xi, \lambda), I_P(\chi\otimes \xi,\nu)).
\end{equation*}

By the duality theorem between intertwining differential operators
and homomorphisms between generalized Verma modules 
 (see, for instance, \cite{CS90, KP1, Kostant75}),
any intertwining differential operator
$\Cal{D} \in \Diff_G(I_P(\chi_\triv, \lambda), I_P(\chi, \nu))$ 
is of the form $\Cal{D} = R(u)$
for some $u \in \Cal{U}(\bar \fn)$, where $R$ denotes the infinitesimal 
right translation of $\Cal{U}(\fg)$ and 
$\bar \fn$ is
the opposite nilpotent radical to $\fn$, namely,
\begin{equation*}
\Diff_G(I_P(\chi_\triv, \lambda), I_P(\chi, \nu))
=\{\text{$R(u)$ for some $u \in \Cal{U}(\bar{\fn})$}\}. 
\end{equation*}
In particular,
intertwining differential operators $\Cal{D}$ are
determined by the complex Lie algebra $\fg$ and independent 
of its real forms.

It follows from
\cite[Lem.\ 2.1]{Kable11} and \cite[Lem.\ 2.24]{KuOr19}
that there exists a $(\Cal{U}(\fk \cap \fm), K\cap M)$-isomorphism
\begin{equation}\label{eqn:gk}
\mathcal{U}(\fk)\otimes_{\mathcal{U}(\fk \cap \fm)}  \C_{\chi_\triv}
\stackrel{\sim}{\To}
\mathcal{U}(\fg) \otimes_{\mathcal{U}(\fp)} 
\C_{\chi_\triv,-(\lambda+\rho)},
\quad u\otimes \mathbb{1}_{\chi_\triv} \mapsto u
\otimes (\mathbb{1}_{\chi_\triv}
\otimes \mathbb{1}_{-(\lambda+\rho)}), 
\end{equation}
where $K\cap M$ acts on 
$\mathcal{U}(\fk)\otimes_{\mathcal{U}(\fk \cap \fm)}  \C_{\chi_\triv}$
and 
$\mathcal{U}(\fg) \otimes_{\mathcal{U}(\fp)} 
\C_{\chi_\triv,-(\lambda+\rho)}$ 
diagonally.
Via the isomorphism \eqref{eqn:gk}, we define a \emph{compact model}
of $u^\flat \in \Cal{U}(\fk)$ of $u \in \Cal{U}(\bar{\fn})$ as follows.

\begin{defn}\label{def:ub}
For $R(u) \in \Diff_G(I_P(\chi_\triv, \lambda), I_P(\chi, \nu))$ with
$u \in \Cal{U}(\bar{\fn})$, 
we denote by $u^\flat \in \Cal{U}(\fk)$ 
an element of $\Cal{U}(\fk)$ such that the identity
\begin{equation}\label{eqn:ub}
u^\flat \otimes (\mathbb{1}_{\chi_\triv} \otimes \mathbb{1}_{-(\lambda+\rho)})
=
u \otimes (\mathbb{1}_{\chi_\triv} \otimes \mathbb{1}_{-(\lambda+\rho)})
\end{equation} 
holds in $\Cal{U}(\fg) \otimes_{\Cal{U}(\fp)}
\C_{\chi_\triv, -(\lambda+\rho)}$.
\end{defn}

\begin{rem}\label{rem:cpt}
We remark that compact model $u^\flat$
is not unique; 
any choice of $u^\flat$ satisfying the identity \eqref{eqn:ub} is acceptable.
\end{rem}

Let $\Irr(K)$ be the set of equivalence classes 
of irreducible representations of the maximal compact subgroup $K$
with Lie algebra $\fk_0$.
For $V_\delta:=(\delta, V) \in \Irr(K)$ and $u \in \Cal{U}(\bar \fn)$,
we define a subspace $\Sol_{(u)}(\delta)$ of $V_\delta$ by
\begin{equation}\label{eqn:SolK}
\Sol_{(u)}(\delta):= \{v \in V_\delta :
d\delta(\tau( u^\flat) )v = 0 \}.
\end{equation}
Here $d\delta$ denotes the differential of $\delta$ and 
$\tau$ denotes the conjugation $\tau\colon \fg \to \fg$ 
with respect to the real form $\fg_0$, that is, 
$\tau(X_1+\sqrt{-1}X_2) = X_1 - \sqrt{-1}X_2$
for $X_1, X_2 \in \fg_0$.

\begin{lem}
[{\cite[Lem.\ 2.41]{KuOr19}}]
\label{lem:SolKO}
The space $\Sol_{(u)}(\delta)$ is 
a $K\cap M$-representation.
\end{lem}

We set
\begin{equation*}
\Sol^{\fk \cap \fm}_{(u)}(\delta) 
:=\Sol_{(u)}(\delta)\cap V_\delta^{\fk \cap \fm},
\end{equation*}
where $V^{\fk \cap \fm}_\delta$ is the subspace of 
$\fk\cap \fm$-invariant vectors of $V_\delta$.
Clearly
$\Sol^{\fk \cap \fm}_{(u)}(\delta)$ is a $K\cap M$-subrepresentation
of $\Sol_{(u)}(\delta)$.
Further, since $\fk \cap \fm$ acts on $\xi \in \Irr(M/M_0)$ trivially,
we have 
\begin{equation}\label{eqn:hom}
\mathrm{Hom}_{K\cap M}\left(
\Sol_{(u)}(\delta), \xi\right) \neq \{ 0\}
\quad 
\text{if and only if}
\quad
\mathrm{Hom}_{K\cap M}\left(
\Sol^{\fk \cap \fm}_{(u)}(\delta), \xi\right) \neq \{ 0\}
\end{equation}
via the composition of maps
$K \cap M \hookrightarrow M \twoheadrightarrow M/M_0$.

Let $I_P(\chi_{\triv}, \lambda)_K$ denote the $(\fg,K)$-module consisting of 
the $K$-finite vectors of $I_P(\chi_{\triv}, \lambda)$.
Then, given 
$R(u) \in \Diff_G(I_P(\chi_\triv, \lambda), I_P(\chi, \nu))$
and 
$\xi \in \Irr(M/M_0)$,
we set
\begin{align*}
\Cal{S}ol_{(u; \lambda)}(\xi)_K
&:=\{f \in I_P(\xi, \lambda)_K : (R(u) \otimes \xid) f = 0\}.
\end{align*}

For the sake of future convenience,
we now give a slight modification of
a Peter--Weyl theorem 
\cite[Thm.\ 1.2]{KuOr19}
for $\Cal{S}ol_{(u; \lambda)}(\xi)_K$,
although such a modification is not necessary for the main objective of this paper.
We remark that this modified version is a more direct generalization of 
the argument given after the proof of \cite[Thm.\ 2.6]{Kable12C} than \cite[Thm.\ 1.2]{KuOr19}.

\begin{thm}
[Peter--Weyl theorem for $\Cal{S}ol_{(u; \lambda)}(\xi)_K$]
\label{thm:PW1}
Let $R(u) \in \Diff_G(I_P(\chi_\triv, \lambda), I_P(\chi, \nu))$
and $\xi \in \Irr(M/M_0)$.
Then the $(\fg, K)$-module $\Cal{S}ol_{(u; \lambda)}(\xi)_K$
can be decomposed as a $K$-representation as
\begin{equation}\label{eqn:K-Sol}
\Cal{S}ol_{(u; \lambda)}(\xi)_K
\simeq \bigoplus_{\delta \in \Irr(K)}
V_\delta \otimes 
\mathrm{Hom}_{K\cap M}\left(
\Sol^{\fk \cap \fm}_{(u)}(\delta), \xi\right).
\end{equation}
\end{thm}

\begin{proof}
It follows from \cite[Thm.\ 1.2]{KuOr19} that
the space $\Cal{S}ol_{(u; \lambda)}(\xi)_K$ can be decomposed as
\begin{equation*}
\Cal{S}ol_{(u; \lambda)}(\xi)_K
\simeq \bigoplus_{\delta \in \Irr(K)}
V_\delta \otimes 
\mathrm{Hom}_{K\cap M}\left(
\Sol_{(u)}(\delta), \xi\right).
\end{equation*}
Now the equivalence \eqref{eqn:hom} concludes the theorem.
\end{proof}

When $G$ is a split real group
and $P=MAN$ is a minimal parabolic subgroup of $G$,
we have $K\cap M = M$, $M/M_0 =M$, and $\fk \cap \fm = \{0\}$.
Thus in this case Theorem \ref{thm:PW1} is  simplified as follows.

\begin{cor}\label{cor:Sol}
Suppose that $G$ is split real and 
$P=MAN$ is minimal parabolic subgroup of $G$.
Let $R(u)\in \mathrm{Diff}_{G}
(I_P(\chi_\triv, \lambda), I_P(\chi,\nu))$
and  $\sigma \in \Irr(M)$.
Then the $(\fg, K)$-module
$\Cal{S}ol_{(u; \lambda)}(\sigma)_K$
can be decomposed as
\begin{equation}\label{eqn:split}
\Cal{S}ol_{(u; \lambda)}(\sigma)_K
\simeq 
\bigoplus_{\delta \in \Irr(K)}
V_\delta \otimes 
\mathrm{Hom}_{M}\left(
\Sol_{(u)}(\delta), \sigma\right)
\end{equation}
\end{cor}


\section{Specialization to $(\wSL(3,\R),B)$}\label{sec:SL3}

The purpose of this section is to specialize 
the general theory discussed in Section \ref{sec:PW} 
to a pair $(\wSL(3,\R), B)$, where $B$ is a minimal parabolic subgroup of 
$\wSL(3,\R)$.
In this section
we in particular give a recipe for computing $K$-type formulas
of the space $\Cal{S}ol_{(u; \lambda)}(\sigma)_K$ 
of $K$-finite solutions
for $R(u)\in \mathrm{Diff}_{G}(I_B(\chi_\triv, \lambda), I_B(\chi,\nu))$.
It is described in Section \ref{subsec:recipe}. 

\subsection{Notation and normalizations}
\label{subsec:SL3}

We begin by recalling from \cite[Sect.\ 4.1]{KuOr19} 
the notation and normalizations for $\wSL(3,\R)$.
Let $G= \widetilde{SL}(3,\R)$ with Lie algebra $\fg_0 = \f{sl}(3,\R)$.
Fix a Cartan involution $\gt\colon \fg_0 \to \fg_0$ such that
$\gt(U)=-U^t$. We write $\fk_0$ and $\fs_0$ for  the $+1$ and $-1$ eigenspaces of $\gt$, respectively, so that
$\fg_0= \fk_0 \oplus \fs_0$ is a Cartan decomposition of $\fg_0$.
We put $\fa_0:=\text{span}_\R\{E_{ii}-E_{i+1, i+1}: i=1, 2\}$
and  $\fn_0:=\text{span}_\R\{E_{12}, E_{23}, E_{13}\}$,
where $E_{ij}$ denote matrix units.
Then $\fb_0:=\fa_0 \oplus \fn_0$ is a minimal parabolic subalgebra of $\fg_0$.

Let $K$, $A$, and $N$ be the analytic subgroups of $G$ with Lie algebras 
$\fk_0$, $\fa_0$, and $\fn_0$, respectively.
Then $G=K AN$ is an Iwasawa decomposition of $G$.
We write $M = Z_{K}(\fa_0)$, so that
$B:=MAN$ is a minimal parabolic subgroup of $G$ 
with Lie algebra $\fb_0$.

We denote by $\fg$ the complexification of the Lie algebra $\fg_0$ of $G$.
A similar convention is employed also for subgroups of $G$;
for instance, $\mathfrak{b}= \mathfrak{a} \oplus \mathfrak{n}$ 
is a Borel subalgebra of $\fg= \f{sl}(3,\C)$.
We write $\bar{\fn}$ for the nilpotent radical opposite to $\fn$.

Let $\Delta\equiv \Delta(\mathfrak{g},\mathfrak{a})$ denote
the set of roots of $\mathfrak{g}$ with respect to $\mathfrak{a}$. 
We denote by $\Delta^+$ and $\Pi$ 
the positive system corresponding to $\fb$ and 
the set of simple roots of $\gD^+$, respectively.
We have $\Pi=\{\eps_1-\eps_2, \eps_2-\eps_3\}$,
where $\eps_j$ are the dual basis of $E_{jj}$ for $j=1, 2, 3$.
The root spaces $\fg_{\eps_1-\eps_2}$ and $\fg_{\eps_2-\eps_3}$ 
are then given as $\fg_{\eps_1-\eps_2} = \C E_{12}$ 
and $\fg_{\eps_2-\eps_3}=\C E_{23}$.
We write $\rho$ for half the sum of the positive roots, namely,
$\rho=\eps_1 - \eps_3$.

We define $X, Y \in \fg$ as
\begin{equation}\label{eqn:XY}
X = 
\begin{small}
\begin{pmatrix}
0 & 0 & 0\\
1 & 0 & 0\\
0 & 0 & 0
\end{pmatrix}
\end{small}
\quad \text{and} \quad
Y = 
\begin{small}
\begin{pmatrix}
0 & 0 & 0\\
0 & 0 & 0\\
0 & 1 & 0
\end{pmatrix}.
\end{small}
\end{equation}
Then $X$ and $Y$ are root vectors for
$-(\eps_1-\eps_2)$ and $-(\eps_2-\eps_3)$, respectively.
The opposite nilpotent radical $\bar{\fn}$ is thus given as
$\bar{\fn}=\text{span}\{X, Y, [X,Y]\}$.

\subsection{Two identifications for
$\f{so}(3,\C) \simeq \f{sl}(2,\C)$}
\label{subsec:so3}

As $\fk_0=\f{so}(3)\simeq \f{su}(2)$, we have $\fk= \f{so}(3,\C) \simeq \f{sl}(2,\C)$.
For later applications 
we consider
two identifications of $\f{so}(3,\C)$ with $\f{sl}(2,\C)$:
\begin{equation*}
\Omega^J \colon \f{so}(3,\C) \stackrel{\sim}{\To} \f{sl}(2,\C)
\quad
\text{for $J= \textnormal{I}, \textnormal{II}$}.
\end{equation*}
Loosely speaking, these identifications are described as follows.

\begin{enumerate}
\item 
Identification 
$\Omega^{\textnormal{I}}\colon \f{so}(3,\C) \stackrel{\sim}{\to} \f{sl}(2,\C)$:
This is an identification
via a Lie algebra isomorphism 
$\Omega^{\textnormal{I}}_0\colon \f{so}(3) \stackrel{\sim}{\to} \f{u}(2)$.
Schematically, we have
\begin{equation*}
\xymatrix@R=1pc{
\f{so}(3,\C) \ar@{.>}[r]_{\sim}^{\Omega^{\textnormal{I}}} & \f{sl}(2,\C)\\
\f{so}(3)
\ar[u]^{\otimes_{\R}\C\;}
\ar[r]^{\sim}_{\Omega^{\textnormal{I}}_0} & \f{su}(2)\ar[u]_{\; \otimes_{\R}\C}
}
\end{equation*}

\item Identification 
$\Omega^{\textnormal{II}}\colon \f{so}(3,\C) \stackrel{\sim}{\to} \f{sl}(2,\C)$:
 This is an identification independent of $\f{su}(2)$. Schematically, we have 
\begin{equation*}
\xymatrix@R=1pc{
\f{so}(3,\C)\ar[r]_{\sim}^{\Omega^{\textnormal{II}}} & \f{sl}(2,\C)\\
\f{so}(3)\ar[u]^{\otimes_{\R}\C\;} & \f{su}(2)
}
\end{equation*}
\end{enumerate}

Let $E_+$, $E_-$, and $E_0$ be the elements of $\f{sl}(2,\C)$ defined as
\begin{equation}\label{eqn:sl2E}
E_+:=
\begin{pmatrix}
0 & 1\\
0 & 0
\end{pmatrix},
\quad
E_-:=
\begin{pmatrix}
0 & 0\\
1 & 0
\end{pmatrix},
\quad
E_0:=
\begin{pmatrix}
1 & 0\\
0 & -1
\end{pmatrix}.
\end{equation}
\noindent
We now describe the identifications 
$\Omega^{\textnormal{I}}$ and $\Omega^{\textnormal{II}}$ 
in detail separately.

\subsubsection{Identification $\fk \simeq \f{sl}(2,\C)$ via $\Omega^{\textnormal{I}}$}

We start with the identification
$\Omega^{\textnormal{I}}\colon \fk \stackrel{\sim}{\to} \f{sl}(2,\C)$.
First observe that 
$\fk_0=\f{so}(3)$ is spanned by the three matrices
\begin{equation*}
B_1:=
\begin{small}
\begin{pmatrix}
0 & 0 & -1\\
0 & 0 & 0\\
1 & 0 &0
\end{pmatrix},
\end{small}
\quad
B_2:=
\begin{small}
\begin{pmatrix}
0 & 0 & 0\\
0 & 0 & -1\\
0 & 1 &0
\end{pmatrix},
\end{small} 
\quad
B_3:=
\begin{small}
\begin{pmatrix}
0 & -1 & 0\\
1 & 0 & 0\\
0 & 0 &0
\end{pmatrix}
\end{small}
\end{equation*}
with commutation relations
\begin{equation}\label{eqn:commB}
[B_1,B_2]=B_3,
\quad
[B_1,B_3]=-B_2, 
\quad
\text{and}
\quad
[B_2, B_3] = B_1.
\end{equation}
On the other hand, the Lie algebra $\f{su}(2)$ is spanned by 
\begin{equation*}
A_1:=\begin{pmatrix}
\sqrt{-1} & 0\\
0 & -\sqrt{-1}
\end{pmatrix},\quad
A_2:=
\begin{pmatrix}
0 & 1\\
-1 & 0
\end{pmatrix},\quad
A_3:=
\begin{pmatrix}
0 & \sqrt{-1}\\
\sqrt{-1} & 0
\end{pmatrix}
\end{equation*}
with commutation relations
\begin{equation*}
[A_1, A_2] = 2A_3,
\quad
[A_1, A_3] = -2 A_2, 
\quad
\text{and}
\quad
[A_2, A_3] = 2A_1.
\end{equation*}
Then one may identify $\fk_0$ with $\f{su}(2)$ via the linear map
\begin{equation}\label{eqn:omega}
\Omega_0^\textnormal{I}\colon \fk_0 \stackrel{\sim}{\To} \f{su}(2),
\quad B_j \mapsto \tfrac{1}{2}A_j \quad \text{for $j=1, 2, 3$}.
\end{equation}

Let $Z_+, Z_-, Z_0$ be the elements of $\fk=\f{so}(3,\C)$ defined as
\begin{equation}\label{eqn:Z}
Z_+ := B_2-\sqrt{-1}B_3,
\quad
Z_-:=-(B_2 + \sqrt{-1}B_3),
\quad
Z_0 := [Z_+, Z_-] = -2\sqrt{-1}B_1.
\end{equation}
Then we have
\begin{equation}\label{eqn:kZ}
\fk = \text{span}\{Z_+,\, Z_-,\, Z_0\}. 
\end{equation}
\vskip 0.05in
\noindent
As $A_2-\sqrt{-1}A_3 =2E_+$ and
$-(A_2 + \sqrt{-1}A_3)= 2E_-$,
the isomorphism \eqref{eqn:omega} yields
a Lie algebra isomorphism
\begin{equation}\label{eqn:omegaKO}
\Omega^{\textnormal{I}}\colon \fk \stackrel{\sim}{\To} \f{sl}(2,\C),
\quad
Z_j \mapsto E_j \quad \text{for $j=+, - , 0$}.
\end{equation}

\subsubsection{Identification $\fk \simeq \f{sl}(2,\C)$ via $\Omega^{\textnormal{II}}$}

We next discuss the identification
$\Omega^{\textnormal{II}}\colon \fk \stackrel{\sim}{\To} \f{sl}(2,\C)$.
Let $W_+$, $W_-$, and $W_0$ be the elements of $\f{so}(3,\C)$ defined by
\begin{equation}\label{eqn:W}
W_+:=B_1+\sqrt{-1}B_3,
\quad
W_-:=-B_1+\sqrt{-1}B_3,
\quad
W_0:= [W_+, W_-] =-2\sqrt{-1}B_2.
\end{equation}
Then we have
\begin{equation}\label{eqn:kW}
\fk = \text{span}\{W_+,\, W_-,\, W_0\}.
\end{equation}
It follows from the commutation relations \eqref{eqn:commB} 
that the triple $\{W_+,\, W_-,\, W_0\}$ forms an $\f{sl}(2)$-triple, namely,
$[W_+, W_-] =W_0$, 
$[W_0. W_+]=2W_+$, 
and 
$[W_0,W_-]=-2W_-$.
Therefore,
$\fk$ may be identified with $\f{sl}(2,\C)$ via the Lie algebra isomorphism
\begin{equation}\label{eqn:omegaT}
\Omega^{\textnormal{II}}\colon 
\fk \stackrel{\sim}{\To} \f{sl}(2,\C),
\quad 
W_j \mapsto E_j \quad \text{for $j=+, - , 0$}.
\end{equation}

\subsection{A realization of the subgroup $M=Z_K(\fa_0)$}

As $K$ is isomorphic to $SU(2)$,
we realize $M =Z_K(\fa_0)$ as a subgroup of $SU(2)$ via the isomorphism
\eqref{eqn:omegaKO}. To do so,
first observe that the adjoint action $\Ad$ of $SU(2)$ on $\f{su}(2)$ 
yields a two-to-one covering map 
$SU(2)\twoheadrightarrow \Ad(SU(2))\simeq SO(3)$
on $SO(3)$. 
We realize $\Ad(SU(2))$ as a matrix group with respect to the ordered basis
$\{A_2, A_1, A_3\}$ of $\f{su}(2)$ in such a way that 
the elements $\pm m_j^{\textnormal{I}} \in SU(2)$ with
\begin{equation}\label{eqn:mKO}
m^{\textnormal{I}}_0:=
\begin{pmatrix}
1 & 0\\
0 & 1\\
\end{pmatrix},\;
m^{\textnormal{I}}_1:=
\begin{pmatrix}
\sqrt{-1} & 0\\
0 & -\sqrt{-1}\\
\end{pmatrix},\;
m^{\textnormal{I}}_2:=
\begin{pmatrix}
0 & 1 \\
-1 & 0
\end{pmatrix},\;
m^{\textnormal{I}}_3:=
\begin{pmatrix}
0 & \sqrt{-1} \\
\sqrt{-1} & 0
\end{pmatrix}
\end{equation}
are mapped to $\pm m_j^{\textnormal{I}}  \mapsto m_j \in SO(3)$ for $j = 0, 1, 2, 3$, where
\begin{equation*}
\mbox{\smaller$
m_0=
\begin{small}
\begin{pmatrix}
1& 0 & 0\\
0 & 1 & 0\\
0 & 0 & 1
\end{pmatrix},
\end{small}
\;
m_1=
\begin{small}
\begin{pmatrix}
-1& 0 & 0\\
0 & 1 & 0\\
0 & 0 & -1
\end{pmatrix},
\end{small}
\;
m_2=
\begin{small}
\begin{pmatrix}
1& 0 & 0\\
0 & -1 & 0\\
0 & 0 & -1
\end{pmatrix},
\end{small}
\;
m_3=
\begin{small}
\begin{pmatrix}
-1& 0 & 0\\
0 & -1 & 0\\
0 & 0 & 1
\end{pmatrix}
\end{small}
$}.
\end{equation*}
\vskip 0.05in

One can easily check that the 
map $\pm m^{\textnormal{I}}_j \mapsto m_j$ respects 
the Lie algebra isomorphism 
$\Omega^{\textnormal{I}} \colon \fk \stackrel{\sim}{\to} \f{sl}(2,\C)$ in \eqref{eqn:omegaKO}
(\cite[Lem.\ 4.6]{KuOr19}).
Namely, for $Z\in \fk$, we have 
\begin{equation*}
\Omega^{\textnormal{I}}(\Ad(m_j)Z)=\Ad(m^{\textnormal{I}}_j)\Omega^{\textnormal{I}}(Z)
\quad
\text{for $j=0,1,2,3$}.
\end{equation*}
\vskip 0.05in
\noindent
As $Z_{SO(3)}(\fa_0)=\{m_0, m_1, m_2, m_3\}$,
we then realize $M$ as a subgroup of $SU(2)$ as
\begin{equation*}
M=\left\{
\pm m^{\textnormal{I}}_0,\;
\pm m^{\textnormal{I}}_1,\;
\pm m^{\textnormal{I}}_2,\;
\pm m^{\textnormal{I}}_3
\right\}.
\end{equation*}
\noindent
The subgroup $M$ is isomorphic to 
the quaternion group $Q_8$, a non-commutative group of order $8$.

\subsection{Irreducible representations $\Irr(M)$ of $M$}\label{subsec:IrrM}

In order to compute a $K$-type formula via \eqref{eqn:split},
we next discuss the sets $\Irr(M)$ and $\Irr(K)$ of 
equivalence classes of irreducible representations of $M$ and $K$, respectively.
We first consider $\Irr(M)$ via 
the isomorphism \eqref{eqn:omegaKO}.
As $M$ is isomorphic to the quaternion group $Q_8$, 
the set $\Irr(M)$ consists of four characters and one 2-dimensional irreducible 
representation.
For $\eps, \eps' \in \{\pm\}$,
we define a character $\chi^{SO(3)}_{(\eps,\eps')}\colon Z_{SO(3)}(\fa_0) 
\to \{\pm 1\}$ of $Z_{SO(3)}(\fa_0)$  as
\begin{equation*}
\chi^{SO(3)}_{(\eps, \eps')}
(\mathrm{diag}(a_1, a_2, a_3))
:=|a_1|_{\eps}\;|a_3|_{\eps'},
\end{equation*}
where
$|a|_+:=|a|$ and $|a|_-:=a$.
Via the character $\chi^{SO(3)}_{(\eps, \eps')}$ of $Z_{SO(3)}(\fa_0)$,
we define a character $\chi_{(\eps, \eps')}\colon M \to \{\pm 1\}$ of
$M$ as
\begin{equation}\label{eqn:charM}
\chi_{(\eps,\eps')}(\pm m^{\textnormal{I}}_j) := \chi_{(\eps, \eps')}(m_j) 
\quad
\textnormal{for\; $j=0, 1, 2, 3$}.
\end{equation}
We often abbreviate $\chi_{(\eps,\eps')}$ as $(\eps,\eps')$. 
The character $\pp$, for instance, is the trivial character of $M$.
Table \ref{table:char} illustrates the character table for 
$(\eps,\eps') =\chi_{(\eps,\eps')}$.
The set $\Irr(M)$ may then be described as follows:
\begin{equation}\label{eqn:char}
\Irr(M)= \{ \pp, \, \pmi, \,  \mip,\,  \mm, \,  \mathbb{H}\},
\end{equation}
where $\mathbb{H}$ is the unique genuine $2$-dimensional representation 
of $M \simeq Q_8$.

\begin{table}[t]
\caption{Character table for $(\eps, \eps')$ for $m_j^{\textnormal{I}}$}
\begin{center}
\renewcommand{\arraystretch}{1.2} 
{
\begin{tabular}{|c|c|c|c|c|}
\hline
& $\pm m_0^{\textnormal{I}}$ 
& $\pm m_1^{\textnormal{I}}$ 
& $\pm m_2^{\textnormal{I}}$ 
& $\pm m_3^{\textnormal{I}}$\\
\hline
$\pp$ & $1$ & $1$ & $1$ & $1$ \\
\hline
$\pmi$ & $1$  & $-1$  & $-1$ & $1$\\
\hline
$\mip$ & $1$ & $-1$ & $1$ & $-1$\\
\hline
$\mm$ & $1$ & $1$ & $-1$ & $-1$\\
\hline
\end{tabular}
}
\end{center}
\label{table:char}
\end{table}%

\subsection{Irreducible representations $\Irr(K)$ of $K$}
\label{subsec:K}

We next consider a polynomial realization of $\Irr(K)$.
Let $\Pol[t]$ be the space of polynomials of one variable $t$ with complex coefficients
and set 
\begin{equation*}
\Pol_n[t] := \{p(t) \in \Pol[t] : \deg p(t) \leq n\}.
\end{equation*}
Then we realize the set $\Irr(K)$ of equivalence classes of 
irreducible representations of $K \simeq SU(2)$ as
\begin{equation}\label{eqn:IrrK}
\Irr(K) \simeq \{(\pi_n, \Pol_n[t]) : n \in \Z_{\geq 0}\},
\end{equation}
where the representation $\pi_n$ of $SU(2)$ on $\Pol_n[t]$ is defined as
\begin{equation}\label{eqn:pin}
\left(\pi_n(g)p\right)(t) := (ct+d)^n p\left(\frac{at+b}{ct+d}\right)
\quad \text{for} \quad
g = \begin{pmatrix}
a & b\\ c& d
\end{pmatrix}^{-1}.
\end{equation}
It follows from \eqref{eqn:pin} that
the elements
$m_j^{\textnormal{I}}$ defined in \eqref{eqn:mKO}
act on $\Pol_n[t]$ via $\pi_n$ as
\begin{equation}\label{eqn:pi-wM}
m_1^{\textnormal{I}}\colon p(t) \mapsto  (\sqrt{-1})^n p(-t); \;\;
m_2^{\textnormal{I}}\colon p(t) \mapsto t^np\left(-\frac{1}{t}\right); \;\;
m_3^{\textnormal{I}}\colon p(t) \mapsto (-\sqrt{-1}t)^n p\left(\frac{1}{t}\right).
\end{equation}

Let $d\pi_n$ be the differential of the representation $\pi_n$.
As usual we extend $d\pi_n$ complex-linearly to $\f{sl}(2,\C)$ and also naturally to
the universal enveloping algebra $\Cal{U}(\f{sl}(2,\C))$.
It follows from \eqref{eqn:sl2E} and \eqref{eqn:pin} 
that $E_+$, $E_-$, and $E_0$ act on $\Pol_n[t]$ via $\dpin$ as
\begin{equation}\label{eqn:Epm}
\dpin(E_+) = -\frac{d}{dt}, \quad
\dpin(E_-) = -nt + t^2\frac{d}{dt},\quad \text{and} \quad 
\dpin(E_0) = -2t\frac{d}{dt}+n.
\end{equation}

\subsection{Peter--Weyl theorem for $\Cal{S}ol_{(u; \lambda)}(\sigma)_K$
with the polynomial realization of $\Irr(K)$}
\label{subsec:PWpoly}

Let $\Omega \colon \fk \stackrel{\sim}{\to} \f{sl}(2,\C)$ 
be an identification of $\fk$ with $\f{sl}(2,\C)$ such as \eqref{eqn:omegaKO} and 
\eqref{eqn:omegaT}. 
Then elements $F \in \Cal{U}(\fk)$ can act on 
$\Pol_n[t]$ via $\dpin$ as $\dpin(\Omega(F))$.
To simplify the notation we write
\begin{equation}\label{eqn:OF}
\dpin^\Omega(F)=\dpin(\Omega(F))
\quad
\text{for $F \in \Cal{U}(\fk)$}.
\end{equation}
Then, for $R(u) \in \Diff_G(I(\pp, \lambda), I(\chi, \nu))$
with $\pp$ the trivial character of $M$, we set 
\begin{equation}\label{eqn:SolK2}
\Sol_{(u)}(n):=\{ p(t) \in \Pol_n[t]:\dpin^{\Omega}\big(\tau(u^\flat)\big)p(t)=0\},
\end{equation}
where $u^\flat \in \Cal{U}(\fk)$ is a compact model of $u\in\Cal{U}(\bar{\fn})$
(see Definition \ref{def:ub}).
Then the $K$-type decomposition of $\Cal{S}ol_{(u; \lambda)}(\sigma)_K$ in \eqref{eqn:split} 
with the polynomial realization 
\eqref{eqn:IrrK} of $\Irr(K)$ becomes
\begin{equation}\label{eqn:split2}
\Cal{S}ol_{(u; \lambda)}(\sigma)_K
\simeq 
\bigoplus_{n\geq 0}
\Pol_n[t] \otimes 
\mathrm{Hom}_{M}\left(
\Sol_{(u)}(n), \sigma\right).
\end{equation}

We remark that
since $K \cap M = M$ and $\fm=\{0\}$, 
the $(\Cal{U}(\fk\cap \fm), K\cap M)$-isomorphism \eqref{eqn:gk} 
reduces to an $M$-isomorphism
$\mathcal{U}(\fk)\otimes \C_{\tiny{\pp}}
\stackrel{\sim}{\To}
\mathcal{U}(\fg) \otimes_{\mathcal{U}(\fb)} \C_{\tiny{\pp},-\lambda}$.
In particular a compact model $u^\flat \in \Cal{U}(\fk)$ 
of $u \in \Cal{U}(\bar{\fn})$ is indeed unique in the present situation
(see Remark \ref{rem:cpt}).

\subsection{Recipe for determining the $K$-type formula for 
$\Cal{S}ol_{(u; \lambda)}(\sigma)_K$}
\label{subsec:recipe}
For later convenience we summarize a recipe for computing the $K$-type
formula for $\Cal{S}ol_{(u; \lambda)}(\sigma)_K$ via \eqref{eqn:split2}.
Let $R(u) \in \Diff_G(I(\pp,\lambda), I(\chi,\nu))$.
\vskip 0.1in

\noindent
\textsf{Step 0:}
Choose an identification $\Omega \colon \fk \simeq \f{sl}(2,\C)$.
\vskip 0.1in

\noindent
\textsf{Step 1:}
Find the compact model $u^\flat$ of $u$ (see Definition \ref{def:ub}).
\vskip 0.1in

\noindent
\textsf{Step 2:}
Find the explicit formula of the differential operator 
$\dpin^\Omega\big(\tau(u^\flat)\big)$.
\vskip 0.1in

\noindent
\textsf{Step 3:}
Solve the differential equation
$\dpin^\Omega\big(\tau(u^\flat)\big)p(t)=0$
and
classify $n \in \Z_{\geq 0}$ such that $\Sol_{(u)}(n)\neq \{0\}$.
\vskip 0.1in

\noindent
\textsf{Step 4:}
For  $n \in \Z_{\geq 0}$ with $\Sol_{(u)}(n)\neq \{0\}$,
classify the $M$-representations on 
$\Sol_{(u)}(n)$.
\vskip 0.1in

\noindent
\textsf{Step 5:}
Given $\sigma \in \Irr(M)$,
classify $n \in \Z_{\geq 0}$ with $\Sol_{(u)}(n)\neq \{0\}$ such that
\begin{equation*}
\mathrm{Hom}_{M}\left(
\Sol_{(u)}(n),\, \sigma\right) \neq \{0\}.
\end{equation*}


\section{Heisenberg ultrahyperbolic operator for $\wSL(3,\R)$}
\label{sec:HO}

The aim of this section is to discuss 
a certain second order differential operator called 
the Heisenberg ultrahyperbolic operator $R(D_s)$.
In this section we apply to $R(D_s)$
the theory described in Section \ref{sec:SL3}
for arbitrary intertwining differential operators $R(u)$
for $\wSL(3,\R)$.
We continue the notation and normalizations 
from Section \ref{sec:SL3}.

\subsection{Heisenberg ultrahyperbolic operator $R(D_s)$}
\label{subsec:HUO}

We start with the definition of 
the Heisenberg ultrahyperbolic operator $R(D_s)$ for $\wSL(3,\R)$.
As $R(D_s)$ being an intertwining differential operator
(see Proposition \ref{prop:square} below),
it is defined for the complex Lie algebra $\fg=\f{sl}(3,\C)$.

Let $X$ and $Y$ be the elements of $\bar{\fn}$ defined in \eqref{eqn:XY}.
For $s \in \C$, we define $D_s$ as
\begin{equation*}
D_s := (XY+YX)+ s[X,Y] \in \Cal{U}(\fg).
\end{equation*}
Then the second order differential operator
\begin{equation*}
R(D_s)=R\big((XY+YX) + s[X,Y] \big)
\end{equation*}
is called the \emph{Heisenberg ultrahyperbolic operator} 
for $\fg = \f{sl}(3,\C)$.
For the general definition 
for $\f{sl}(m,\C)$ with arbitrary rank $m \geq 3$,
see \cite[Sect.\ 3]{Kable12C}.

Recall from Section \ref{subsec:SL3} that the simple roots 
$\ga, \gb \in \Pi$ are realized as $\ga=\eps_1-\eps_2$ and $\gb=\eps_2-\eps_3$.
Then, for $s \in \C$, we set
\begin{equation}\label{eqn:wrho(s)}
\wrho(s):=\wrho - s \wrho^{\perp} 
\end{equation}
with
\begin{equation*}
\wrho:=\frac{1}{2}(\eps_1-\eps_3)
\quad
\text{and}
\quad
\wrho^{\perp}:=\frac{1}{2}(\eps_1+\eps_3).
\end{equation*}

Remark that as $\wrho^{\perp} \notin \fa^*$,
we have $\wrho(s) \notin \fa^*$ unless $s=0$.
For $\sigma \in \Irr(M)$, we write
\begin{equation}\label{eqn:IndB}
I(\sigma,\wrho(s))=I_B(\sigma,\wrho(s)\vert_{\fa^*})
\end{equation}
for the parabolically induced representation
$I_B(\sigma,\wrho(s)\vert_{\fa^*})$
defined as in \eqref{eqn:Ind}.
Proposition \ref{prop:square} below shows that
the operator $R(D_s)$ is indeed an intertwining differential operator
between parabolically induced representations.

\begin{prop}\label{prop:square}
We have
\begin{equation*}
R(D_s)
\in 
\Diff_G(I(\pp, -\wrho(s)), 
I(\mm, \wrho(-s)).
\end{equation*}
Consequently, for $\sigma \in \Irr(M)$, we have
\begin{equation*}
R(D_s) \otimes \id_\sigma 
\in 
\Diff_G(I(\sigma, -\wrho(s)), 
I(\mm \otimes \sigma, \wrho(-s)).
\end{equation*}
\end{prop}

\begin{proof}
The first assertion readily follows from \cite[Lem.\ 3.2]{Kable12C} and \cite[Lem.\ 6.4]{KuOr19}.
Lemma \ref{lem:tensor} then concludes the second.
\end{proof}

Our aim is to compute the branching law 
of the space $\CSol_{(D_{s};-\wrho(s))}(\sigma)_K$
of $K$-finite solutions to $(R(D_s) \otimes \id_\sigma)f=0$.
The $K$-type formula in \eqref{eqn:split2s} 
with $(u;\lambda) = (D_{s};-\wrho(s))$ gives
\begin{equation}\label{eqn:split2s}
\CSol_{(D_{s};-\wrho(s))}(\sigma)_K
\simeq 
\bigoplus_{n\geq 0}
\Pol_n[t] \otimes 
\mathrm{Hom}_{M}\left(
\Sol_{(D_s)}(n), \sigma\right)
\end{equation}
with
\begin{equation*}
\Sol_{(D_s)}(n)
=\{ p(t) \in \Pol_n[t]:\dpin^\Omega\big(\tau(D_{s}^\flat)\big)p(t)=0\}.
\end{equation*}
\vskip 0.05in

Observe that, for $X$ and $Y$ in \eqref{eqn:XY}, 
we have $X, Y \in \fg_0=\f{sl}(3,\R)$.
Thus $\tau(D_s^\flat)$ for $D_s = (XY+YX) + s[X,Y]$ is simply given as
\begin{equation}\label{eqn:tds}
\tau(D^\flat_s) =\tau(D_s)^\flat = D^\flat_{\bar{s}}\,,
\end{equation}
where $\bar{s}$ denotes the complex conjugate of $s\in \C$.
Then we write 
\begin{equation*}
\Sol(s;n):=\Sol_{(D_{\bar{s}})}(n)
\end{equation*}
so that
\begin{align}\label{eqn:SolK3}
\Sol(s;n)
&=\{ p(t) \in \Pol_n[t]:\dpin^\Omega\big(\tau(D_{\bar{s}}^\flat)\big)p(t)=0\} \nonumber\\
&=\{ p(t) \in \Pol_n[t]:\dpin^\Omega\big(D_s^\flat\big)p(t)=0\}.
\end{align}
Similarly, we put
\begin{equation*}
\CSol(s;\sigma)_K := \CSol_{(D_{\bar{s}};-\wrho(\bar{s}))}(\sigma)_K.
\end{equation*}
Then 
the $K$-type formula \eqref{eqn:split2s} yields
\begin{equation}\label{eqn:PWSL}
\CSol(s;\sigma)_K
\simeq 
\bigoplus_{n\geq 0}
\Pol_n[t] \otimes 
\mathrm{Hom}_{M}\left(
\Sol(s;n), \sigma\right).
\end{equation}
Hereafter we consider the branching law of $\CSol(s;\sigma)_K$
instead of $\CSol_{(D_{s};-\wrho(s))}(\sigma)_K$.

In order to solve the equation 
$\dpin^\Omega\big(D_s^\flat\big)p(t)=0$
in \eqref{eqn:SolK3},
we next find the compact model $D_s^\flat$ of $D_s$, 
which is an element
of  $\Cal{U}(\fk)$ such that
\begin{equation*}
D_s^\flat \otimes (\mathbb{1}_{{\tiny \pp}} \otimes \mathbb{1}_{\wrho(s)-\rho})
=D_s \otimes (\mathbb{1}_{{\tiny \pp}} \otimes \mathbb{1}_{\wrho(s)-\rho})
\end{equation*} 
in $\Cal{U}(\fg) \otimes_{\Cal{U}(\fb)}\C_{{\tiny \pp},\, \wrho(s)-\rho}$. 
For $X, Y \in \bar{\fn}$ in \eqref{eqn:XY}, 
define $X^\flat, Y^\flat \in \fk$ as
\begin{equation*}\label{eqn:XYb}
X^\flat:=X + \gt(X)
\quad
\text{and}
\quad
Y^\flat:=Y+\gt(Y),
\end{equation*}
where $\gt$ is the Cartan involution defined as $\gt(U)=-U^t$.

\begin{lem}\label{lem:ub}
We have
\begin{equation}\label{eqn:Db}
D^\flat_s = 
X^\flat Y^\flat+Y^\flat X^\flat+s[X^\flat, Y^\flat]  \in \Cal{U}(\fk).
\end{equation}
\end{lem}

\begin{proof}
One can easily verify that
\begin{align*}\label{eqn:ub}
(X^\flat Y^\flat+Y^\flat X^\flat+s[X^\flat, Y^\flat] )\otimes 
(\mathbb{1}_{{\tiny \pp}} \otimes \mathbb{1}_{\wrho(s)-\rho})
=D_s \otimes (\mathbb{1}_{{\tiny \pp}} \otimes \mathbb{1}_{\wrho(s)-\rho})
\end{align*} 
in $\Cal{U}(\fg) \otimes_{\Cal{U}(\fb)}\C_{{\tiny \pp}, \, \wrho(s)-\rho}$.
\end{proof}

\subsection{Relationship between
$\dpin^{\textnormal{I}}(D^\flat_s)$ and $\dpin^{\textnormal{II}}(D^\flat_s)$}
\label{subsec:Omega12}

In Section \ref{subsec:so3}, we discussed
two identifications of $\fk$ with $\f{sl}(2,\C)$, namely,
\begin{equation}\label{eqn:KOTO}
\Omega^{\textnormal{I}} \colon \fk \stackrel{\sim}{\to} \f{sl}(2,\C)
\quad
\text{and}
\quad
\Omega^{\textnormal{II}} \colon \fk \stackrel{\sim}{\to} \f{sl}(2,\C).
\end{equation}
Thus compact model $D^\flat_s$ 
may act on $\Pol_n[t]$ via $\dpin$ as
\begin{equation}\label{eqn:OI}
\dpin(\Omega^{\textnormal{I}}(D^\flat_s))
\quad
\text{and}
\quad
\dpin(\Omega^{\textnormal{II}}(D^\flat_s)).
\end{equation}
As for the notation 
$\dpin^\Omega(F) 
=\dpin(\Omega(F))$ 
in \eqref{eqn:OF}, 
we abbreviate \eqref{eqn:OI} as
\begin{equation*}
d\pi_n^J(D^\flat_s) = d\pi_n(\Omega^J(D^\flat_s))
\quad 
\text{for $J\in\{\textnormal{I}, \textnormal{II}\}$}.
\end{equation*}
\noindent
We next discuss a relationship between 
$\dpin^{\textnormal{I}}(D^\flat_s)$ and $\dpin^{\textnormal{II}}(D^\flat_s)$.
\vskip 0.2in

We begin with the expression of $\Omega^J(D^\flat_s)$ in terms of 
the $\f{sl}(2)$-triple $\{E_+, \, E_-,\, E_0\}$ in \eqref{eqn:sl2E}.
First, recall from \eqref{eqn:kZ} and \eqref{eqn:kW} that 
$\fk$ can be given as
\begin{equation*}
\fk = \text{span} \{ Z_+,\, Z_-,\, Z_0\} 
= \text{span} \{ W_+,\, W_-,\, W_0\},
\end{equation*}
where $\{ Z_+,\, Z_-,\, Z_0\}$ and 
$\{ W_+,\, W_-,\, W_0\}$ are the $\f{sl}(2)$-triples of $\fk$ 
defined in \eqref{eqn:Z} and \eqref{eqn:W}, respectively.

\begin{lem}\label{lem:DbZW}
The compact model 
$D^\flat_s = X^\flat Y^\flat+Y^\flat X^\flat+s[X^\flat, Y^\flat]$ is expressed as
\begin{align}
D^\flat_s 
&=\frac{\sqrt{-1}}{2}\left( (Z_++Z_-)(Z_+-Z_-)-(s-1)Z_0\right) \label{eqn:DbZ}\\
&=\frac{1}{2}\left((W_++W_-)W_0-(s-1)(W_+-W_-)\right). \label{eqn:DbW}
\end{align}
\end{lem}

\begin{proof}

A direct computation shows that $X^\flat$ and $Y^\flat$ are expressed 
in terms of $\{ Z_+,\, Z_-,\, Z_0\}$ and $\{ W_+,\, W_-,\, W_0\}$ as
\begin{alignat*}{2}
X^\flat
&=\frac{\sqrt{-1}}{2}(Z_++Z_-)
&&=-\frac{\sqrt{-1}}{2}(W_++W_-),\\
Y^\flat
&=\frac{1}{2}(Z_+-Z_-)
&&=\frac{\sqrt{-1}}{2}W_0.
\end{alignat*}
By substituting these expressions into $D^\flat_s$,
one obtains the lemma. 
\end{proof}

\begin{lem}\label{lem:DbO}
The elements 
$\Omega^{\textnormal{I}}(D^\flat_s)$ and $\Omega^{\textnormal{II}}(D^\flat_s)$ 
in $\Cal{U}(\f{sl}(2,\C))$ are given as
\begin{align}
\Omega^{\textnormal{I}}(D^\flat_s)
&=\frac{\sqrt{-1}}{2}\left( (E_++E_-)(E_+-E_-)-(s-1)E_0\right),
\label{eqn:oKODb}\\[1pt]
\Omega^{\textnormal{II}}(D^\flat_s)
&=\frac{1}{2}\left((E_++E_-)E_0-(s-1)(E_+-E_-)\right). \label{eqn:oTDb}
\end{align}
\end{lem}

\begin{proof}
This follows from \eqref{eqn:omegaKO}, \eqref{eqn:omegaT}, and Lemma \ref{lem:DbZW}.
\end{proof}

We set 
\begin{equation}\label{eqn:k0}
k_0:=\frac{1}{\sqrt{2}}
\begin{pmatrix}
\sqrt{-1} & 1\\
-1 & -\sqrt{-1}
\end{pmatrix} \in SU(2).
\end{equation}

\begin{prop}\label{prop:KOT}
We have
\begin{equation}\label{eqn:KOT}
\Omega^{\textnormal{I}}(D^\flat_s) = \Ad(k_0)\Omega^{\textnormal{II}}(D^\flat_s).
\end{equation}
Consequently, for $p(t) \in \Pol_n[t]$, the following are equivalent:
\begin{enumerate}[{\normalfont (i)}]
\item $d\pi_n^{\textnormal{II}}(D^\flat_s)p(t)=0$;
\item $d\pi_n^{\textnormal{I}}(D^\flat_s)\pi_n(k_0)p(t)=0$.
\end{enumerate}
\end{prop}

\begin{proof}
One can easily check that 
\begin{align*}
E_++E_-&=-\Ad(k_0)(E_++E_-),\\
E_+-E_-&=\sqrt{-1}\Ad(k_0)E_0.
\end{align*}
Now the identity \eqref{eqn:KOT} follows from Lemma \ref{lem:DbO}.
Since $d\pi_n^J(D^\flat_s) = d\pi_n(\Omega^J(D^\flat_s))$
for $J \in \{\textnormal{I}, \textnormal{II}\}$,
the equivalence between (i) and (ii) readily follows from \eqref{eqn:KOT}.
\end{proof}

For $J\in\{\textnormal{I}, \textnormal{II}\}$, we write
\begin{equation}\label{eqn:TKO}
\Sol_J(s;n)=\{p(t) \in \Pol_n[t] : d\pi^J_n(D^\flat_s)p(t)=0\}.
\end{equation}
In place of $\Sol(s;n)$ with $\Sol_J(s;n)$,
the $K$-type decomposition \eqref{eqn:PWSL} becomes
\begin{equation}\label{eqn:PWSL2}
\CSol(s;\sigma)_K
\simeq 
\bigoplus_{n\geq 0}
\Pol_n[t] \otimes 
\mathrm{Hom}_{M}\left(
\Sol_J(s;n), \sigma\right).
\end{equation}
\noindent

It follows from Proposition \ref{prop:KOT} that 
$\pi_n(k_0)$ yields a linear isomorphism
\begin{equation}\label{eqn:pik0}
\pi_n(k_0)\colon \Sol_{\textnormal{II}}(s;n) 
\stackrel{\sim}{\To} \Sol_{\textnormal{I}}(s;n).
\end{equation}
For $m^{\textnormal{I}}_j \in M$ for $j=0,1,2,3$ defined in \eqref{eqn:mKO}, 
we define 
$m^{\textnormal{II}}_j \in M$ as
\begin{equation}\label{eqn:mTKO}
m^{\textnormal{II}}_j:=k_0^{-1}m^{\textnormal{I}}_jk_0
\end{equation}
so that the following diagram commutes.
\begin{equation*}
\xymatrix{
\Sol_{\textnormal{II}}(s;n) 
\ar@{->}^{\pi_n(k_0)}_{\sim}[r]
\ar@{.>}_{\pi_n(m^{\textnormal{II}}_j)}[d]
&
\Sol_{\textnormal{I}}(s;n)
\ar@{->}^{\pi_n(m^{\textnormal{I}}_j)}[d]
\\
\Sol_{\textnormal{II}}(s;n) 
\ar@{->}_{\pi_n(k_0)}^{\sim}[r]
& 
\Sol_{\textnormal{I}}(s;n)
}
\end{equation*}
A direct computation shows that
$m^{\textnormal{II}}_j$ are given as follows.
\begin{equation}\label{eqn:mT}
m^{\textnormal{II}}_0=
\begin{pmatrix}
1 & 0\\
0 & 1\\
\end{pmatrix},\;
m^{\textnormal{II}}_1=
\begin{pmatrix}
0 & 1 \\
-1 & 0
\end{pmatrix},\;
m^{\textnormal{II}}_2=
\begin{pmatrix}
\sqrt{-1} & 0\\
0 & -\sqrt{-1}\\
\end{pmatrix},\;
m^{\textnormal{II}}_3=
-
\begin{pmatrix}
0 & \sqrt{-1} \\
\sqrt{-1} & 0
\end{pmatrix}.
\end{equation}
By \eqref{eqn:mKO} and \eqref{eqn:mT}, we have
\begin{equation}\label{eqn:mTKO2}
m_0^{\textnormal{II}}=m_0^{\textnormal{I}},\quad
m_1^{\textnormal{II}}=m_2^{\textnormal{I}},\quad
m_2^{\textnormal{II}}=m_1^{\textnormal{I}},\quad
m_3^{\textnormal{II}}=-m_3^{\textnormal{I}}.
\end{equation}
It then follows from \eqref{eqn:pi-wM} that
$m_j^{\textnormal{II}}$ act on $\Pol_n[t]$ via $\dpin$ as
\begin{equation}\label{eqn:pi-wMT}
m_1^{\textnormal{II}}\colon p(t) \mapsto t^np\left(-\frac{1}{t}\right); \;\;
m_2^{\textnormal{II}}\colon p(t) \mapsto  (\sqrt{-1})^n p(-t); \;\;
m_3^{\textnormal{II}}\colon p(t) \mapsto (\sqrt{-1}t)^n p\left(\frac{1}{t}\right).
\end{equation}
Table \ref{table:char2} 
is the character table for $(\eps, \eps')$ with $m_j^{\textnormal{II}}$.

\begin{table}[t]
\caption{Character table for $(\eps, \eps')$ for $m_j^{\textnormal{II}}$}
\begin{center}
\renewcommand{\arraystretch}{1.2} 
{
\begin{tabular}{|c|c|c|c|c|}
\hline
& $\pm m_0^{\textnormal{II}}$ 
& $\pm m_1^{\textnormal{II}}$ 
& $\pm m_2^{\textnormal{II}}$ 
& $\pm m_3^{\textnormal{II}}$\\
\hline
$\pp$ & $1$ & $1$ & $1$ & $1$ \\
\hline
$\pmi$ & $1$  & $-1$  & $-1$ & $1$\\
\hline
$\mip$ & $1$ & $-1$ & $1$ & $-1$\\
\hline
$\mm$ & $1$ & $1$ & $-1$ & $-1$\\
\hline
\end{tabular}
}
\end{center}
\label{table:char2}
\end{table}%

\subsection{Differential equation $\dpin^{J}(D^\flat_s)f(t)=0$ for $J \in \{\textnormal{I}, \textnormal{II}\}$}\label{subsec:HGEandHeun}

As indicated in the recipe in Section \ref{subsec:recipe},
to determine the $K$-type decomposition of $\CSol(s;n)_K$,
it is crucial to determine the space $\Sol_J(s;n)$ of $K$-type solutions,
for which one needs to find polynomial solutions $p(t) \in \Pol_n[t]$
to differential equations $\dpin^{J}(D^\flat_s)f(t)=0$
for $J \in \{\textnormal{I}, \textnormal{II}\}$. 
For this purpose we next investigate
these differential equations.

\subsubsection{Differential equation $\dpin^{\textnormal{I}}(D^\flat_s)f(t)=0$}

We start with the differential equation $\dpin^{\textnormal{I}}(D^\flat_s)f(t)=0$.
The explicit formula of $\dpin^{\textnormal{I}}(D^\flat_s)$ is given as follows.

\begin{lem}\label{lem:dpinKO}
We have
\begin{equation}\label{eqn:dpinKO}
-2\sqrt{-1}\, d\pi_n^{\textnormal{I}}(D^\flat_s)
=
(1-t^4)\frac{d^2}{dt^2}+2\left((n-1)t^2+s\right)t\frac{d}{dt} -n\left((n-1)t^2+s\right).
\end{equation}
\end{lem}

\begin{proof}
Recall from \eqref{eqn:oKODb} that we have 
\begin{equation*}
-2\sqrt{-1}\, \Omega^{\textnormal{I}}(D^\flat_s)
=(E_++E_-)(E_+-E_-)-(s-1)E_0.
\end{equation*}
Now the proposed identity follows from a direct computation with \eqref{eqn:Epm}.
\end{proof}

In order to study the differential equation $\dpin^{\textnormal{I}}(D^\flat_s)f(t)=0$,
we set 
\begin{equation}\label{eqn:HeunOp}
\D_H(a,q;\ga, \gb, \gamma,\delta;z)
:=\frac{d^2}{dz^2}+\left(\frac{\gamma}{z} + \frac{\delta}{z-1}
+\frac{\eps}{z-a}\right)\frac{d}{dz}
+\frac{\ga\gb z-q}{z(z-1)(z-a)},
\end{equation}
where $a, q, \ga, \gb, \gamma, \delta, \eps$ are complex parameters
with $a \neq 0, 1$ and $\gamma + \delta+\eps = \ga + \gb +1$.
Then the differential equation
\begin{equation}\label{eqn:Heun0}
\D_H(a,q;\ga,\gb,\gamma,\delta;z)f(z)=0
\end{equation}
is called \emph{Heun's differential equation}. 
For a brief account of the equation \eqref{eqn:Heun0}, 
see Section \ref{subsec:Heun}.
The differential equation 
$\dpin^{\textnormal{I}}(D^\flat_s)f(t)=0$ can be identified 
with Heun's differential equation as in Lemma \ref{lem:Heun} below.
(We are grateful to Hiroyuki Ochiai for pointing it out.)

\begin{lem}\label{lem:Heun}
We have
\begin{equation*}
-\frac{d\pi_n^{\textnormal{I}}(D^\flat_s)}{4}
=
\D_H(-1, -\frac{ns}{4}; -\frac{n}{2}, -\frac{n-1}{2}, \frac{1}{2}, \frac{1-n-s}{2};t^2).
\end{equation*}
\end{lem}

\begin{proof}
This follows from a direct computation with a change of variables $z=t^2$ 
on \eqref{eqn:dpinKO}.
\end{proof}

By Lemma \ref{lem:Heun},
to determine the space 
$\Sol_{\textnormal{I}}(s;n)$ of $K$-type solutions to
$\dpin^{\textnormal{I}}(D^\flat_s)f(t)=0$, 
it suffices to find polynomial solutions to the Heun equation
\begin{equation}\label{eqn:Heun1}
\D_H(-1, -\frac{ns}{4}; -\frac{n}{2}, -\frac{n-1}{2}, \frac{1}{2}, \frac{1-n-s}{2};t^2)
f(t)=0.
\end{equation}

Let $Hl(a,q;\ga, \gb,\gamma,\delta;z)$ denote the power series solution 
$Hl(a,q;\ga, \gb,\gamma,\delta;z)=\sum_{r=0}^\infty c_rz^r$ with $c_0=1$
to \eqref{eqn:Heun0} at $z=0$ (\cite{Ron95}).
We set
\begin{align}
&u_{[s;n]}(t):=
Hl(-1, -\frac{ns}{4}; -\frac{n}{2}, -\frac{n-1}{2}, \frac{1}{2}, \frac{1-n-s}{2};t^2),
\label{eqn:u}
\\[3pt]
&v_{[s;n]}(t):=
tHl(-1, -\frac{(n-2)s}{4}; -\frac{n-1}{2}, -\frac{n-2}{2}, \frac{3}{2}, \frac{1-n-s}{2};t^2).
\label{eqn:v}
\end{align}
\noindent
It follows from \eqref{eqn:HeunSol} in Section \ref{sec:appendix} that
$u_{[s;n]}(t)$ and $v_{[s;n]}(t)$ are linearly independent solutions at $t=0$
to the Heun equation \eqref{eqn:Heun1}. 
Therefore we have
\begin{equation}\label{eqn:KOuv}
\Sol_{\textnormal{I}}(s;n) \subset \C u_{[s;n]}(t) \oplus \C v_{[s;n]}(t).
\end{equation}
We shall classify the parameters $(s,n) \in \C \times \Z_{\geq 0}$ such that
$u_{[s;n]}(t), v_{[s;n]}(t) \in \Sol_{\textnormal{I}}(s;n)$
in Section \ref{sec:Heun} (see Propositions \ref{prop:upoly} and \ref{prop:vpoly}).

\begin{rem}\label{rem:S0F}
Let $F(a,b,c;z)\equiv {}_2F_1(a,b,c;z)$ denote the Gauss hypergeometric function.
In \cite[(3.8)]{Maier05}, Maier found the following Heun-to-Gauss reduction formula:
\begin{equation}\label{eqn:Maier}
Hl(-1,0;\ga,\gb, \gamma, \frac{\ga+\gb-\gamma+1}{2};t)
=F(\frac{\ga}{2},\frac{\gb}{2},\frac{\gamma+1}{2};t^2).
\end{equation}
We remark that the case of $s=0$ for $\dpin^{\textnormal{I}}(D^\flat_0)f(t)=0$ 
gives special cases of this identity.
Indeed, it follows from \cite[Lem.\ 7.3]{KuOr19} that
a change of variables $z=t^4$ yields the identity
\begin{equation*}
\frac{\sqrt{-1}}{32t^2}d\pi_n(D^\flat_0) 
=
\D_F(-\frac{n}{4}, -\frac{n-1}{4}, \frac{3}{4};t^4),
\end{equation*}
where $\D_F(a,b,c;z)$ denotes the differential operator
\begin{equation}\label{eqn:HGEop}
\D_F(a,b,c;z)=z(1-z)\frac{d^2}{dz^2}+(c-(a+b+1)z)\frac{d}{dz}-ab
\end{equation}
such that 
\begin{equation*}
\D_F(a,b,c;z)f(z)=0
\end{equation*}
is the hypergeometric differential equation.
Thus the equation $\dpin^{\textnormal{I}}(D^\flat_0)f(t)=0$ is also equivalent to
the hypergeometric equation
\begin{equation}\label{eqn:Gauss}
\D_F(-\frac{n}{4}, -\frac{n-1}{4}, \frac{3}{4};t^4)f(t)=0. 
\end{equation}
Then \eqref{eqn:Heun1} and \eqref{eqn:Gauss} yield
the following Heun-to-Gauss reductions:
\begin{align}
Hl(-1, 0; -\frac{n}{2}, -\frac{n-1}{2}, \frac{1}{2}, -\frac{n-1}{2};t)
&=F(-\frac{n}{4}, -\frac{n-1}{4}, \frac{3}{4};t^2);\label{eqn:HG1}\\[5pt]
Hl(-1, 0; -\frac{n-1}{2}, -\frac{n-2}{2}, \frac{3}{2}, -\frac{n-1}{2};t) \label{eqn:HG2}
&=F(-\frac{n-1}{4}, -\frac{n-2}{4}, \frac{5}{4};t^2).
\end{align}
\end{rem}

\subsubsection{Differential equation $\dpin^{\textnormal{II}}(D^\flat_s)f(t)=0$}
We next consider 
$\dpin^{\textnormal{II}}(D^\flat_s)f(t)=0$.
The equation $\dpin^{\textnormal{II}}(D^\flat_s)f(t)=0$ can be identified with
the hypergeometric equation $\D_F(a,b,c;z)f(z)=0$ as follows.

\begin{lem}\label{lem:HGE}
We have
\begin{equation*}
\frac{d\pi_n^{\textnormal{II}}(D^\flat_s)}{16t}
=
D_F(-\frac{n}{2},-\frac{n+s-1}{4},\frac{3-n+s}{4};t^2).
\end{equation*}
\end{lem}

\begin{proof}
First, it follows from \eqref{eqn:oTDb} and \eqref{eqn:Epm} that
$\dpin^{\textnormal{II}}(D^\flat_s)$ is given as
\begin{equation}\label{lem:dpiT}
2\,\dpin^{\textnormal{II}}(D^\flat_s)=2(1-t^2)t\frac{d^2}{dt^2}+((s+3n-3)t^2+(s-n+1))\frac{d}{dt}
-n(s+n-1)t.
\end{equation}
Then a direct computation with a change of variable $z=t^2$ concludes the lemma.
\end{proof}

Similar to the equation $\dpin^{\textnormal{I}}(D^\flat_s)f(t)=0$,
Lemma \ref{lem:HGE} shows that, 
to determine the space $\Sol_{\textnormal{II}}(s;n) $ 
of $K$-type solutions to $\dpin^{\textnormal{II}}(D^\flat_s)f(t)=0$, 
it suffices to find polynomial solutions to
the hypergeometric equation
\begin{equation}\label{eqn:Gauss2}
D_F(-\frac{n}{2},-\frac{n+s-1}{4},\frac{3-n+s}{4};t^2)f(t)=0.
\end{equation}
We set
\begin{align*}
&a_{[s;n]}(t):= F(-\frac{n}{2}, -\frac{n+s-1}{4}, \frac{3-n+s}{4};t^2),\\[3pt]
&b_{[s;n]}(t):=t^{\frac{1+n-s}{2}}F(-\frac{n+s-1}{4},-\frac{s-1}{2}, \frac{5+n-s}{4};t^2).
\end{align*}
Then $a_{[s;n]}(t)$ and $b_{[s;n]}(t)$ form fundamental set of solutions to
\eqref{eqn:Gauss2} for suitable parameters $(s,n) \in \C \times \Z_{\geq 0}$ 
for which
$a_{[s;n]}(t)$ and $b_{[s;n]}(t)$ are well-defined.
Therefore we have
\begin{equation}\label{eqn:Tab}
\Sol_{\textnormal{II}}(s;n) \subset \C a_{[s;n]}(t) \oplus \C b_{[s;n]}(t).
\end{equation}
The precise conditions of $(s,n)$ such that
$a_{[s;n]}(t), b_{[s;n]}(t) \in \Sol_{\textnormal{II}}(s;n)$
will be investigated in Section \ref{sec:SolT}.

\subsection{Recipe for the $K$-type decomposition of $\CSol(s;\sigma)_K$}
\label{subsec:recipe2}

In Section \ref{subsec:recipe},
we gave a general recipe to determine the $K$-type formula
 \eqref{eqn:split2} of $\CSol_{(u;\lambda)}(\sigma)_K$.
We modify the recipe in such a way that it will fit well
 for $\CSol(s;\sigma)_K$ with 
 \begin{equation*}
\CSol(s;\sigma)_K
\simeq 
\bigoplus_{n\geq 0}
\Pol_n[t] \otimes 
\mathrm{Hom}_{M}\left(
\Sol_J(s;n), \sigma\right).
\end{equation*}
\noindent

We first fix an identification 
\begin{equation*}
\Omega^J \colon \fk \stackrel{\sim}{\To} \f{sl}(2,\C)\;\;
\text{for $J \in \{\textnormal{I}, \textnormal{II}\}$}.
\end{equation*}

\noindent
Then the $K$-type formula of $\CSol(s;\sigma)_K$ may be determined 
in the following three steps.
\vskip 0.1in

\noindent
\textsf{Step A:}
Classify $(s,n) \in \C \times \Z_{\geq 0}$ such that 
$\Sol_J(s;n)\neq \{0\}$.
It follows from \eqref{eqn:KOuv} and \eqref{eqn:Tab} that 
this is indeed equivalent to classifying $(s,n) \in \C \times \Z_{\geq 0}$
such that 
\begin{itemize}
\item $u_{[s;n]}(t) \in \Pol_n[t]$ or $v_{[s;n]}(t)\in \Pol_n[t]$ for $J=\textnormal{I}$;
\vskip 0.1in
\item $a_{[s;n]}(t)\in \Pol_n[t]$ or $b_{[s;n]}(t)\in \Pol_n[t]$ for $J=\textnormal{II}$.
\end{itemize}

\vskip 0.1in

\noindent
\textsf{Step B:}
For  $(s,n) \in \C\times \Z_{\geq 0}$ with 
$\Sol_J(s;n)\neq \{0\}$,
classify the $M$-representations on $\Sol_J(s;n)$.
\vskip 0.1in

\noindent
\textsf{Step C:}
Given $\sigma \in \Irr(M)$,
classify $(s,n) \in \C\times \Z_{\geq 0}$ with 
$\Sol_J(s;n)\neq \{0\}$ such that
\begin{equation*}
\mathrm{Hom}_{M}\left(
\Sol_J(s;n),\, \sigma\right) \neq \{0\}.
\end{equation*}
\vskip 0.1in

In Sections \ref{sec:HGE} and \ref{sec:Heun},
we shall proceed Steps A, B, and C for
$J=\textnormal{II}$ and $J=\textnormal{I}$, respectively.


\section{Hypergeometric model 
$\dpin^{\textnormal{II}}(D^\flat_s)f(t)=0$}\label{sec:HGE}

The aim of this section is to 
classify the $K$-type formulas for $\CSol(s;\sigma)_K$ 
by using the hypergeometric model $d\pi^{\textnormal{II}}_n(D^\flat_s)f(t)=0$.
The decomposition formulas 
are achieved in Theorem \ref{thm:K-type}.

\subsection{The classification of $\Sol_{\textnormal{II}}(s;n)$}
\label{sec:SolT}

As Step A of the recipe in Section \ref{subsec:recipe2},
we first wish to classify $(s,n) \in \C \times \Z_{\geq 0}$ such that
$\Sol_{\textnormal{II}}(s;n)\neq \{0\}$, where
\begin{equation*}
\Sol_{\textnormal{II}}(s;n)=\{p(t) \in \Pol_n[t] : 
d\pi^\textnormal{II}_n(D^\flat_s)p(t)=0\}.
\end{equation*}
Recall from \eqref{eqn:Tab} that we have
\begin{equation*}
\Sol_{\textnormal{II}}(s;n) \subset \C a_{[s;n]}(t) \oplus \C b_{[s;n]}(t)
\end{equation*}
with
\begin{align}
&a_{[s;n]}(t) = F(-\frac{n}{2}, -\frac{n+s-1}{4}, \frac{3-n+s}{4};t^2),
\label{eqn:an}\\[3pt]
&b_{[s;n]}(t)=t^{\frac{1+n-s}{2}}
F(-\frac{n+s-1}{4},-\frac{s-1}{2}, \frac{5+n-s}{4};t^2). \label{eqn:bn}
\end{align}
\vskip 0.05in
\noindent
It thus suffices to classify $(s,n) \in \C \times \Z_{\geq 0}$ such that
$a_{[s;n]}(t)\in \Pol_n[t]$ or $b_{[s;n]}(t)\in \Pol_n[t]$.

\vskip 0.1in
Given $n\in \Z_{\geq 0}$, we define 
$I_0^{\pm}$, $J_0$, $I_1$, $I_2^{\pm}$, $J_2$, $I_3 \subset \Z$ as follows:
\begin{alignat}{3}\label{eqn:IJ}
\begin{aligned}
I_0^{\pm}&:=\{\pm(3+4j) : j=0,1,\ldots,\frac{n}{4}-1\}\quad (n \in 4\Z_{\geq0});\\
J_0&:=\{1+4j: j=0,1,\ldots, \frac{n}{4}-1\} \quad (n \in 4\Z_{\geq0}); \\
I_1&:=\{\pm 4j : j=0,1, \dots, \left[\frac{n}{4}\right] \};\\
I_2^{\pm}
&:=\{\pm(1+4j) : j=0,1,\ldots,
\left[\frac{n}{4}\right]\};\\
J_2&:=\{3+4j: j=0,1,\ldots, 
\left[\frac{n}{4}\right]-1\};\\
I_3&:=\{\pm (2+4j) : j=0,1, \dots, \left[\frac{n}{4}\right] \}.
\end{aligned}
\end{alignat}
\vskip 0.1in

We start by observing when $a_{[s;n]}(t) \in \Pol_n[t]$.
A key observation is that,
for $(a,b,c) \in \C$ for which the hypergeometric function 
$F(a,b,c;z)$ is well-defined, we have
\begin{equation*}
F(a,b,c;z) \in \Pol_n[z] 
\quad
\Longleftrightarrow
\quad
\text{$a \in \{0, -1, \ldots, -n\}$ or $b \in \{0, -1, \ldots, -n\}$}.
\end{equation*}


\begin{prop}\label{prop:an}
The following conditions on $(s,n) \in \C \times \Z_{\geq 0}$ are equivalent.
\begin{enumerate}[{\normalfont (i)}]
\item $a_{[s;n]}(t) \in \Pol_n[t]$.
\item One of the following conditions hold:
\begin{enumerate}[{\normalfont (a)}]
\item $n\equiv 0 \Mod 4:$ $s \in \C \backslash J_0;$

\item $n \equiv 1 \Mod 4:$ $s \in I_1;$

\item $n\equiv 2 \Mod 4:$ $s \in \C \backslash J_2;$

\item $n\equiv 3 \Mod 4:$ $s \in I_3$.

\end{enumerate}
\end{enumerate}
\end{prop}

\begin{proof}
The proposition follows from a careful observation for
the parameters of $a_{[s;n]}(t)$ in \eqref{eqn:an}.
Indeed, suppose that $n$ is even. 
Then the parameters of $a_{[s;n]}(t)$
imply that $a_{[s;n]}(t) \notin \Pol_n[t]$ if and only if 
the following conditions are satisfied
\begin{equation*}\label{eqn:nEven}
\frac{3-n+s}{4} \in -\Z_{\geq 0},
\quad
-\frac{n}{2}<\frac{3-n+s}{4},
\quad
\text{and}
\quad
-\frac{n+s-1}{4}<\frac{3-n+s}{4},
\end{equation*}
which is equivalent to $s \in J_0$ for $n\equiv 0 \Mod 4$ and 
$s \in J_2$ for $n\equiv 2 \Mod 4$.
Now the assertions for (a) and (c) follow from the contrapositive
of the arguments.

Next suppose that $n$ is odd.
In this case it follows from \eqref{eqn:an} that
$a_{[s;n]}(t) \in \Pol_n[t]$ if and only if
\begin{equation*}
0 \leq \frac{n+s-1}{4} \leq \left[\frac{n}{2}\right]
\quad 
\text{and}
\quad
\frac{n+s-1}{4} \in \Z_{\geq 0},
\end{equation*}
which is equivalent to 
$s \in I_1$ for $n\equiv 1 \Mod 4$
and $s \in I_3$ for $n\equiv 3 \Mod 4$.
Here we remark that, for $s \in I_1 \cup I_3$, 
we have $\frac{3-n-s}{4} \notin \Z$.
This concludes the proposition.
\end{proof}

Suppose that $a_{[s;n]}(t) \in \Pol_n[t]$.
Then generically we have $\deg a_{[s;n]}(t) = n$. 
Lemma \ref{lem:an1} below classifies 
the singular parameters of $s \in \C$ for $a_{[s;n]}(t)$ 
in a sense that $\deg a_{[s;n]}(t) <n$
(see Section \ref{subsub:Case2C}).

\begin{lem}\label{lem:an1}
Suppose that $n \equiv k \Mod 4$ 
and $s \in \C \backslash J_k$ for $k = 0, 2$.
Then the following conditions on $(s, n) \in \C \times \Z_{\geq 0}$ are equivalent.
\begin{enumerate}[{\normalfont (i)}]
\item $\deg a_{[s;n]}(t) < n$.
\item $\deg a_{[s;n]}(t) = \frac{n+s-1}{2}$.
\item One of the following conditions holds:
\begin{enumerate}[{\normalfont (a)}]\item $n\equiv 0 \Mod 4:$ $s \in I_0^-;$
\item $n \equiv 2 \Mod 4:$ $s\in I_2^-$.
\end{enumerate}
\end{enumerate}
\end{lem}

\begin{proof}
It follows from the parameters of $a_{[s;n]}(t)$ that
$\deg a_{[s;n]}(t) < n$ if and only if 
\begin{equation}\label{eqn:ans1}
2\cdot \frac{n+s-1}{4} < n 
\quad
\text{and}
\quad
\frac{n+s-1}{4} \in \Z_{\geq 0},
\end{equation}
which shows the equivalence between (i) and (ii).
Moreover, \eqref{eqn:ans1} is equivalent to
\begin{equation}\label{eqn:ans2}
s \in \{-n+1+4j : j=0,1,2, \ldots, \frac{n}{2}-1\}.
\end{equation}
One can readily verifty that, under the condition $s \notin J_k$,
 \eqref{eqn:ans2}
is indeed equivalent to $s \in I^-_k$ for $k=0,2$.
\end{proof}


We next consider $b_{[s;n]}(t)$ in \eqref{eqn:bn}.

\begin{prop}\label{prop:bn}
The following conditions on $(s,n) \in \C\times \Z_{\geq 0}$ are equivalent.
\begin{enumerate}[{\normalfont (i)}]
\item $b_{[s;n]}(t) \in \Pol_n[t]$ with $b_{[s;n]}(t) \neq a_{[s;n]}(t)$.
\item One of the following conditions holds.
\begin{enumerate}[{\normalfont (a)}]\item $n \equiv 0 \Mod 4:$
$s \in I^+_0 \cup I^-_0 \cup J_0;$
\item $n \equiv 1 \Mod 4:$
$s \in I_1;$
\item $n\equiv 2 \Mod 4:$
$s \in I^+_2 \cup I^-_2 \cup J_2;$
\item $n \equiv 3 \Mod 4:$
$s \in I_3;$
\end{enumerate}

\end{enumerate}
\end{prop}

\begin{proof}
 Observe that 
if $b_{[s;n]}(t) \in \Pol[t]$, then the exponent $\frac{1+n-s}{2}$ for 
$t^\frac{1+n-s}{2}$ in \eqref{eqn:bn} must satisfy $\frac{1+n-s}{2} \in \Z_{\geq 0}$,
which in particular forces $\frac{5+n-s}{4} \notin -\Z_{\geq 0}$. Moreover,
if $\frac{1+n-s}{2}=0$, then 
\begin{equation*}
b_{[n+1;n]}(t)=F(-\frac{n}{2},-\frac{n}{2},1;t^2)=a_{[n+1;n]}(t).
\end{equation*}
\vskip 0.05in
\noindent
Consequently, we have 
$b_{[s;n]}(t) \in \Pol_n[t]$ with $b_{[s;n]}(t) \neq a_{[s;n]}(t)$
if and only if either
\begin{alignat}{4}
& 
&&\frac{1+n-s}{2} \in 1+\Z_{\geq 0}, 
\quad
&&\frac{n+s-1}{4} \in \Z_{\geq 0},
\quad
&&\text{and}
\quad
 \frac{1+n-2}{2}+2\cdot \frac{n+s-1}{4} \leq n,\label{eqn:bn1}\\
&\text{or}\qquad  && && && \nonumber \\
&
&&\frac{1+n-s}{2} \in 1+\Z_{\geq 0}, 
\quad
&&\frac{s-1}{2} \in \Z_{\geq 0},
\quad
&&\text{and}
\quad 
\frac{1+n-2}{2}+2\cdot \frac{s-1}{2} \leq n.\label{eqn:bn2}
\end{alignat}
\vskip 0.05in
\noindent
A direct observation shows that the conditions \eqref{eqn:bn1} and \eqref{eqn:bn2}
are equivalent to
\begin{equation}\label{eqn:bn3}
s \in (-n+1+4\Z_{\geq 0}) \cap (-\infty, n+1)
\end{equation}
and
\begin{equation}\label{eqn:bn4}
\text{$n$ is even and $s\in (1+2\Z_{\geq 0}) \cap (-\infty, n+1)$},
\end{equation}
respectively. 
One can directly verify that
\eqref{eqn:bn3} and \eqref{eqn:bn4}  
are equivalent to the conditions on $s \in \C$ 
stated in Proposition \ref{prop:bn} for 
$n \equiv k \Mod 4$ for $k=0,1,2,3$.
Indeed, if $n \equiv 0 \Mod 4$,
then 
\eqref{eqn:bn3} and \eqref{eqn:bn4}  are equivalent to
$s \in I^-_0 \cup J_0$ and $s \in I^+_0 \cup J_0$, respectively.
Since the other three cases can be shown similarly, we omit the proof.
\end{proof}

It follows from Propositions \ref{prop:an} and \ref{prop:bn} that
if $n \equiv k \Mod 4$ for $k=0,2$, then 
\begin{equation}\label{eqn:abn}
\Sol_{\textnormal{II}}(s;n)=
\C a_{[s;n]}(t) \oplus \C b_{[s;n]}(t) 
\quad
\textnormal{for $s \in I_k^+\cup I_k^-$}.
\end{equation}
For a later purpose for determining the $M$-representations
on $\Sol_{\textnormal{II}}(s;n)$,
for such $n\equiv k \Mod 4$, we define
\begin{equation*}
c^{\pm}_{[s;n]}(t):=a_{[s;n]}(t)\pm C(s;n)b_{[s;n]}(t)
\quad
\text{for $s \in I_k^-$},
\end{equation*}
where
\begin{equation}\label{eqn:Csn}
C(s;n)
:=
\begin{cases}
\frac{\left(-\frac{n}{2}, \left[\frac{n}{4}\right]\right)}
{\left[\frac{n}{4}\right]!} &
\textnormal{if $n\equiv 2 \Mod 4$ and $s=-1$},\vspace{5pt}\\
\frac{\left(-\frac{n}{2}, \frac{n+s-1}{4}\right)
\left(-\frac{n+s-1}{4}, \frac{n+s-1}{4}\right)}
{\left(\frac{3-n+s}{4},\frac{n+s-1}{4}\right)}
\cdot
\frac{1}{\left(\frac{n+s-1}{4}\right)!}&
\textnormal{otherwise}.
\end{cases}
\end{equation}
\vskip 0.05in
\noindent
Here $(\ell,m)$ stands for the shifted factorial, namely,
$(\ell,m) = \frac{\Gamma(\ell+m)}{\Gamma(\ell)}$.
Then, 
for $n \equiv k \Mod 4$ for $k=0,2$,
the space $\Sol_{\textnormal{II}}(s;n)$ may be described as
\begin{equation}\label{eqn:Solw}
\Sol_{\textnormal{II}}(s;n)=
\begin{cases}
\C a_{[s;n]}(t) \oplus \C b_{[s;n]}(t) &
\textnormal{if $s \in I_k^+$},\\
\C c^+_{[s;n]}(t) \oplus \C c^-_{[s;n]}(t) &
\textnormal{if $s \in I_k^-$}.
\end{cases}
\end{equation}
\vskip 0.05in
\noindent
We then summarize the parameters $(s,n)$ such that $\Sol_{\textnormal{II}}(s;n)\neq \{0\}$
as follows.

\begin{thm}\label{thm:Sol1}
The following conditions on $(s,n) \in \C \times \Z_{\geq 0}$ are 
equivalent.
\begin{enumerate}[{\normalfont (i)}]
\item $\Sol_{\textnormal{II}}(s;n) \neq \{0\}$. 
\item One of the following conditions is satisfied.
\begin{itemize}
\item $n \equiv 0 \Mod 4:$ $s \in \C$.
\item $n \equiv 1 \Mod 4:$ $s \in I_1$.
\item $n\equiv 2 \Mod 4:$ $s \in \C$.
\item $n \equiv 3 \Mod 4:$ $s \in I_3$.
\end{itemize}
\end{enumerate}
Further, 
for such $(s,n)$,
the space $\Sol_{\textnormal{II}}(s;n)$ may be given as follows.
\begin{enumerate}[{\quad \normalfont (1)}]
\item 
$n \equiv 0 \Mod 4:$ 
\begin{equation*}
\Sol_{\textnormal{II}}(s;n)
=
\begin{cases}
\C a_{[s;n]}(t) & \textnormal{if $s \in \C\backslash (I_0^+ \cup I_0^- \cup J_0)$},\\
\C a_{[s;n]}(t) \oplus \C b_{[s;n]}(t) &\textnormal{if $s \in I_0^+$},\\
\C b_{[s;n]}(t) & \textnormal{if $s \in J_0$},\\
\C c^+_{[s;n]}(t) \oplus \C c^-_{[s;n]}(t) &\textnormal{if $s\in I_0^-$}.
\end{cases}
\end{equation*}
\vskip 0.1in

\item $n \equiv 1 \Mod 4:$ 
\begin{equation*}
\Sol_{\textnormal{II}}(s;n)=\C a_{[s;n]}(t) \oplus \C b_{[s;n]}(t) \quad \textnormal{for $s\in I_1$}.
\end{equation*}

\item 
$n \equiv 2 \Mod 4:$ 
\begin{equation*}
\Sol_{\textnormal{II}}(s;n)
=
\begin{cases}
\C a_{[s;n]}(t) & \textnormal{if $s \in \C\backslash (I_2^+ \cup I_2^- \cup J_2)$},\\
\C a_{[s;n]}(t) \oplus \C b_{[s;n]}(t) &\textnormal{if $s \in I_2^+$},\\
\C b_{[s;n]}(t) & \textnormal{if $s \in J_2$},\\
\C c^+_{[s;n]}(t) \oplus \C c^-_{[s;n]}(t) & \textnormal{if $s\in I_2^-$}.
\end{cases}
\end{equation*}
\vskip 0.1in

\item $n \equiv 3 \Mod 4:$ 
\begin{equation*}
\Sol_{\textnormal{II}}(s;n)=
\C a_{[s;n]}(t) \oplus \C b_{[s;n]}(t) \quad \textnormal{for $s\in I_3$}.
\end{equation*}

\end{enumerate}
\end{thm}

\begin{proof}
This is a summary of the results in
Propositions \ref{prop:an} and \ref{prop:bn} and \eqref{eqn:Solw}.
\end{proof}

\begin{rem}\label{rem:SolT1}
Theorem \ref{thm:Sol1} shows that the structure of 
$\Sol_{\textnormal{II}}(s;n)$ (hypergeometric model)
is somewhat complicated. It will be shown in Theorem \ref{thm:SolKO1}
that that of $\Sol_{\textnormal{I}}(s;n)$ (Heun model)
is more straightforward.
\end{rem}

\begin{rem}\label{rem:SolT2}
Theorem \ref{thm:Sol1} also completely classifies the dimension
$\dim_\C\Sol_{\textnormal{II}}(s;n)$.
When $n$ is even, the dimension $\dim_\C\Sol_{\textnormal{II}}(s;n)$
($=\dim_\C\Sol_{\textnormal{I}}(s;n)$) was determined
in \cite[Thm.\ 5.13]{Kable12A} by factorization formulas
of certain tridiagonal determinants 
(see Theorem \ref{thm:factorPQ} and Remark \ref{rem:PQ1}).
We shall study such determinants
in Section \ref{sec:PQ} from a different perspective from \cite{Kable12A}.
\end{rem}

\subsection{The $M$-representations on $\Sol_{\textnormal{II}}(s;n)$}

As Step B of the recipe in Section \ref{subsec:recipe2}, 
we next classify the $M$-representations on $\Sol_{\textnormal{II}}(s;n)$.
Here is the classification.

\begin{thm}\label{thm:Sol2}
For each $(s;n) \in \C \times \Z_{\geq 0}$ determined in Theorem \ref{thm:Sol1},
the $M$-representations on $\Sol_{\textnormal{II}}(s;n)$ are classified as follows.

\begin{enumerate}[{\normalfont (1)}]
\item 
$n \equiv 0 \Mod 4:$ 
\begin{equation*}
\Sol_{\textnormal{II}}(s;n)
\simeq
\begin{cases}
\pp & \textnormal{if $s \in \C\backslash (I_0^+ \cup I_0^- \cup J_0)$},\\
\pp \oplus \pmi & \textnormal{if $s \in I_0^+$},\\
\pp & \textnormal{if $s \in J_0$},\\
\pp \oplus \mip & \textnormal{if $s\in I_0^-$}.
\end{cases}
\end{equation*}
\vskip 0.1in

\item $n \equiv 1 \Mod 4:$ 
\begin{equation*}
\Sol_{\textnormal{II}}(s;n) \simeq \mathbb{H} \quad \textnormal{for $s\in I_1$}.
\end{equation*}

\item 
$n \equiv 2 \Mod 4:$ 
\begin{equation*}
\Sol_{\textnormal{II}}(s;n)
\simeq
\begin{cases}
\mm & \textnormal{if $s \in \C\backslash (I_2^+ \cup I_2^- \cup J_2)$},\\
\mm \oplus \mip &\textnormal{if $s \in I_2^+$},\\
\mm & \textnormal{if $s \in J_2$},\\
\mm \oplus \pmi & \textnormal{if $s\in I_2^-$}.
\end{cases}
\end{equation*}
\vskip 0.1in

\item $n \equiv 3 \Mod 4:$ 
\begin{equation*}
\Sol_{\textnormal{II}}(s;n)\simeq \mathbb{H} \quad \textnormal{for $s\in I_3$}.
\end{equation*}

\end{enumerate}
Here the characters $(\eps, \eps')$ stand for 
the ones on $\C a_{[s;n]}(t)$, $\C b_{[s;n]}(t)$, 
and $\C c^{\pm}_{[s;n]}(t)$ at the same places in
Theorem \ref{thm:Sol1}.
\end{thm}

We prove Theorem \ref{thm:Sol2} by considering the following cases
separately.

\begin{itemize}

\item Case 1: $n$ is odd.

\item Case 2: $n$ is even. Let $k\in \{0,2\}$.

\begin{itemize}
\item Case 2a: $n\equiv k \Mod 4$ and 
$s \in \C\backslash (I_k^-\cup J_k)$ 

\item Case 2b: $n\equiv k \Mod 4$ and
$s \in I_k^+\cup J_k$

\item Case 2c: $n\equiv k \Mod 4$ and
$s \in I_k^-$

\end{itemize}

\end{itemize}

\subsubsection{Case 1}
We start with the case that $\Sol_{\textnormal{II}}(s;n)$ is a two-dimensional 
representation of $M$.

\begin{lem}\label{lem:A2}
Suppose that $n \equiv k \Mod 4$ and $s \in I_k$ for $k = 1, 3$.
Then, as an $M$-representation, we have
\begin{equation*}
\C a_{[s;n]}(t) \oplus \C b_{[s;n]}(t) \simeq \mathbb{H}.
\end{equation*}
\end{lem}

\begin{proof}
In this case $a_{[s;n]}(t)$ and $b_{[s;n]}(t)$ are even and odd functions,
respectively. Then the assertion can be shown by essentially the same 
argument as the one for \cite[Prop.\ 6.11]{KuOr19} by replacing
$u_n(t)$ and $v_n(t)$ with
$a_{[s;n]}(t)$ and $b_{[s;n]}(t)$ , respectively.
Hence we omit the proof.
\end{proof}

\subsubsection{Cases 2a}
We next consider the characters on $\C a_{[s;n]}(t)$.

\begin{lem}\label{lem:A1}
Suppose that $n \equiv k \Mod 4$ and 
$s \in \C\backslash (I_k^-\cup J_k)$ for $k = 0, 2$,
Then $M$ acts on $\C a_{[s;n]}(t)$ as a character.
\end{lem}

\begin{proof}
We only give a proof for the case $k=0$, namely,
$n\equiv 0 \Mod 4$ and $s \in \C \backslash (I_0^-\cup J_0)$;
the other case can be shown similarly.
As  
\begin{equation*}
\C \backslash (I_0^-\cup J_0)
=( \C \backslash (I_0^+\cup I_0^- \cup J_0)) \cup I_0^+,
\end{equation*}
we consider the cases 
$s\in \C \backslash (I_0^+\cup I_0^- \cup J_0)$ and 
$s \in I_0^+$, separately.

First suppose that $s\in \C \backslash (I_0^+\cup I_0^- \cup J_0)$.
By Theorem \ref{thm:Sol1}, we have $\Sol_{\textnormal{II}}(s;n) = \C a_{[s;n]}(t)$.
Since 
$\Sol_{\textnormal{II}}(s;n)$ is an $M$-representation by
Lemma \ref{lem:SolKO},
the assertion clearly holds for this case.
We next suppose that $s \in I_0^+$. 
In this case we have
$\Sol_{\textnormal{II}}(s;n) = \C a_{[s;n]}(t) \oplus \C b_{[s;n]}(t)$ by Theorem \ref{thm:Sol1}.
As the exponent $\frac{1+n-s}{2}$ for $t^{\frac{1+n-s}{2}}$
of $b_{[s;n]}(t)$ is odd, 
$a_{[s;n]}(t)$ and $b_{[s;n]}(t)$ are 
an even and odd function, respectively. 
The transformation laws
\eqref{eqn:pi-wMT}
imply that the action of $M$ on 
$\Pol_n[t]$ preserves the parities of the polynomials
for $n$ even. Hence $M$ acts both
on $\C a_{[s;n]}(t)$ and $\C b_{[s;n]}(t)$ as a character.
\end{proof}

We next determine the characters on $a_{[s;n]}(t)$ explicitly.
It follows from Lemma \ref{lem:an1} that
$a_{[s;n]}(t)$ has degree $\deg a_{[s;n]}(t) =n$
for $(s,n)$
with 
$n \equiv k \Mod 4$ and 
$s \in \C \backslash (I_k^- \cup J_k)$ for $k=0,2$.
Thus, in this case, the hypergeometric polynomial 
$a_{[s;n]}(t)$ is given as
\begin{equation}\label{eqn:ansum}
a_{[s;n]}(t) = 
\sum_{j=0}^{n/2}A_j(s;n)t^{2j}
\end{equation}
with
\begin{equation*}
A_j(s;n)
=\frac{\left(-\frac{n}{2}, j\right)
\left(-\frac{n+s-1}{4}, j\right)}
{\left(\frac{3-n+s}{4},j\right)}
\cdot
\frac{1}{j!},
\end{equation*}
where 
$(\ell,m)$ stands for the shifted factorial.
Recall from \eqref{eqn:mT} that 
$M = \{ \pm m_j^{\textnormal{II}} : j = 0, 1, 2, 3\}$ with
\begin{equation*}
m^{\textnormal{II}}_0=
\begin{pmatrix}
1 & 0\\
0 & 1\\
\end{pmatrix},\;
m^{\textnormal{II}}_1=
\begin{pmatrix}
0 & 1 \\
-1 & 0
\end{pmatrix},\;
m^{\textnormal{II}}_2=
\begin{pmatrix}
\sqrt{-1} & 0\\
0 & -\sqrt{-1}\\
\end{pmatrix},\;
m^{\textnormal{II}}_3=
-
\begin{pmatrix}
0 & \sqrt{-1} \\
\sqrt{-1} & 0
\end{pmatrix}.
\end{equation*}

\noindent
As $m_3^{\textnormal{II}} = m_1^{\textnormal{II}}m_2^{\textnormal{II}}$, it suffices to check the actions 
of $m_1^{\textnormal{II}}$ and $m_2^{\textnormal{II}}$
on $a_{[s;n]}(t) = \sum_{j=0}^{n/2}A_j(s;n)t^{2j}$.

\begin{prop}\label{prop:A2}
Under the same hypothesis in Lemma \ref{lem:A1},
the character $(\eps,\eps')$ on $\C a_{[s;n]}(t)$ is given as follows.
\begin{enumerate}[{\normalfont (1)}]
\item $n \equiv 0 \Mod 4$ and $s \in \C \backslash (I_0^- \cup J_0):$
\begin{equation*}
\C a_{[s;n]}(t) \simeq \pp.
\end{equation*}

\item $n \equiv 2 \Mod 4$ and $s \in \C \backslash (I_2^- \cup J_2):$
\begin{equation*}
\C a_{[s;n]}(t) \simeq \mm.
\end{equation*}
\end{enumerate}
\end{prop}

\begin{proof}
Since the second assertion can be shown similarly,
we only give a proof for the first assertion.
Let $n \equiv 0 \Mod 4$ and 
$s \in \C \backslash (I_0^- \cup J_0)$.
We wish to show that both 
$m_1^{\textnormal{II}}$ and $m_2^{\textnormal{II}}$ act trivially.
First, it is easy to see that the action of $m_2^{\textnormal{II}}$ is trivial.
Indeed, since $n \equiv 0 \Mod 4$
and $a_{[s;n]}(t)$ is an even function,
the transformation law of \eqref{eqn:pi-wMT} shows that
\begin{equation*}
m_2^{\textnormal{II}}\colon a_{[s;n]}(t)
\To
(\sqrt{-1})^{n}a_{[s;n]}(-t)
=a_{[s;n]}(t).
\end{equation*}
In order to show that 
$m_1^{\textnormal{II}}$ also acts trivially, observe that, 
by \eqref{eqn:pi-wMT} and \eqref{eqn:ansum}, 
we have
\begin{equation*}
m_1^{\textnormal{II}} \colon a_{[s;n]}(t)
\To
t^na_{[s;n]}\left(-\frac{1}{t}\right)
= 
\sum_{j=0}^{n/2}A_{\frac{n}{2}-j}(s;n)t^{2j}.
\end{equation*}
On the other hand, 
by Lemma \ref{lem:A1} and 
Table \ref{table:char2}, $m_1^{\textnormal{II}}$ acts on $a_{[s;n]}(t)$ by $\pm 1$.
Therefore,
\begin{equation*}
\sum_{j=0}^{n/2}A_{\frac{n}{2}-j}(s;n)t^{2j}=
\pm \sum_{j=0}^{n/2}A_j(s;n)t^{2j}.
\end{equation*}
An easy computation shows that we have
$A_{\frac{n}{2}}(s;n)=1=A_0(s;n)$, which forces
$t^na_{[s;n]}\left(-\frac{1}{t}\right) = a_{[s;n]}(t)$.
Hence $m_1^{\textnormal{II}}$ also acts on $\C a_{[s;n]}(t)$ trivially.
\end{proof}

\subsubsection{Cases 2b}
Next we consider the characters on $\C b_{[s;n]}(t)$.

\begin{prop}\label{prop:B2}
Suppose that $n \equiv k \Mod 4$ and 
$s \in I_k^+ \cup J_k$ for $k = 0, 2$.
Then $M$ acts on $\C b_{[s;n]}(t)$ as a character as follows.
\begin{enumerate}[{\normalfont (1)}]
\item $n \equiv 0 \Mod 4$ and $s \in I_0^+ \cup J_0:$
\begin{equation*}
\C b_{[s;n]}(t) \simeq 
\begin{cases}
\pmi & \textnormal{if $s \in I^+_0$},\\
\pp & \textnormal{if $s \in J_0$}.
\end{cases}
\end{equation*}

\item $n \equiv 2 \Mod 4$ and $s \in I_2^+ \cup J_2:$
\begin{equation*}
\C b_{[s;n]}(t) \simeq
\begin{cases}
\mip & \textnormal{if $s \in I^+_2$},\\
\mm & \textnormal{if $s \in J_2$}.\\
\end{cases}
\end{equation*}
\end{enumerate}
\end{prop}

\begin{proof}
Since the assertions can be shown similarly to 
Lemma \ref{lem:A1} and Proposition \ref{prop:A2}, 
we omit the proof. We remark that it is already shown in 
the proof of Lemma \ref{lem:A1} that $M$ acts on
$\C b_{[s;n]}(t)$ as a character for 
$(s,n)$ with $n\equiv 0 \Mod 4$ and $s \in I_0^+$. 
\end{proof}

\subsubsection{Case 2c}\label{subsub:Case2C}

Now we consider $\C c^{\pm}_{[s;n]}(t)$ 
for $n \equiv k \Mod 4$ and $s \in I_k^-$ 
for $k = 0, 2$, where $\C c^{\pm}_{[s;n]}(t)$ is given as
\begin{equation*}
c^{\pm}_{[s;n]}(t)=a_{[s;n]}(t)\pm C(s;n)b_{[s;n]}(t)
\end{equation*}
with $C(s;n)$ in \eqref{eqn:Csn}.

Observe that, in this case,
$a_{[s;n]}(t)$ has degree $\deg a_{[s;n]}(t)= \frac{n+s-1}{2}<n$
by Lemma \ref{lem:an1}. Therefore, $a_{[s;n]}(t)$ is of the form
\begin{equation}\label{eqn:ansum2}
a_{[s;n]}(t) = 
\sum_{j=0}^{(n+s-1)/4}\wA_j(s;n)t^{2j}.
\end{equation}
To give an explicit formula of $\wA_j(s;n)$,
first remark that when 
$n\equiv 2 \Mod 4$ and $s=-1$, we have
$-\frac{n+s-1}{4}=\frac{3-n+s}{4} \in \Z_{\geq 0}$.
Namely, for $n=4\ell+2$ and $s=-1$, 
the hypergeometric polynomial $a_{[-1;4\ell+2]}(t)$ is 
\begin{equation}\label{eqn:an2C}
a_{[-1;4\ell+2]}(t) = F(-2\ell-1,-\ell,-\ell;t^2).
\end{equation}
In this paper we regard \eqref{eqn:an2C} as
\begin{equation*}
a_{[-1;4\ell+2]}(t) = \sum_{j=0}^k \frac{(-2\ell-1,j)}{j!} t^{2j}
\end{equation*}
so that $a_{[-1;4\ell+2]}(t)$ and $b_{[-1;4\ell+2]}(t)$ still
form a fundamental solutions to 
the hypergeometric equation \eqref{eqn:Gauss2}.
Therefore $\wA_j(s;n)$ in \eqref{eqn:ansum2} is given as
\begin{equation*}
\wA_j(s;n)
=
\begin{cases}
\frac{\left(-\frac{n}{2},j\right)}{j!} & 
\textnormal{if $n\equiv 2 \Mod 4$ and $s=-1$}\vspace{5pt} \\
\frac{\left(-\frac{n}{2}, j\right)
\left(-\frac{n+s-1}{4}, j\right)}
{\left(\frac{3-n+s}{4},j\right)}
\cdot
\frac{1}{j!} 
& \textnormal{otherwise}.
\end{cases}
\end{equation*}
In particular, by \eqref{eqn:Csn}, we have
\begin{equation}\label{eqn:ACsn}
C(s;n) = \wA_{\frac{n+s-1}{4}}(s;n).
\end{equation}

It follows from \eqref{eqn:bn} that $b_{[s;n]}(t)$ is given as
\begin{equation}\label{eqn:bnsum1}
b_{[s;n]}(t) = t^{\frac{1+n-s}{2}}\sum_{j=0}^{(n+s-1)/4}B_j(s;n)t^{2j}
\end{equation}
with
\begin{equation*}
B_j(s;n)
=\frac{\left(-\frac{n+s-1}{4}, j\right)
\left(-\frac{s-1}{2}, j\right)}
{\left(\frac{5+n-s}{4},j\right)}
\cdot
\frac{1}{j!}.
\end{equation*}

Lemma \ref{lem:C2} below plays a key role to 
determine the $M$-representation
on $\C c^{\pm}_{[s;n]}(t)$.

\begin{lem}\label{lem:C2}
Suppose that $n \equiv k \Mod 4$ and $s \in I_k^-$
for $k=0,2$. Then, 
\begin{equation*}
\pi_n(m_1^{\textnormal{II}})
a_{[s;n]}(t)=
C(s;n)b_{[s;n]}(t).
\end{equation*}
\end{lem}

\begin{proof}
We only show the case of $k=0$; 
the other case can be shown similarly,
including the exceptional case for $n\equiv 2 \Mod 4$ and $s=-1$.
Let $n\equiv 0 \Mod 4$ and $s \in I_0^-$. 
It follows from \eqref{eqn:abn} that
we have 
\begin{equation*}
\Sol_{\textnormal{II}}(s;n)=
\C a_{[s;n]}(t) \oplus \C b_{[s;n]}(t).
\end{equation*}
Therefore there exist some constants $c_1, c_2 \in \C$ 
such that 
\begin{equation}\label{eqn:abc12}
\pi_n(m_1^{\textnormal{II}})a_{[s;n]}(t)
=c_1 a_{[s;n]}(t) + c_2 b_{[s;n]}(t).
\end{equation}
On the other hand, by \eqref{eqn:pi-wMT} and \eqref{eqn:ansum2}, 
we have
\begin{equation}\label{eqn:maA}
\pi_n(m_1^{\textnormal{II}})a_{[s;n]}(t)
=t^n a_{[s;n]}\left(-\frac{1}{t}\right)
=\sum_{j=0}^{(n+s-1)/4}\wA_j(s;n)t^{n-2j}.
\end{equation}
Since
\begin{equation*}
n-\frac{n+s-1}{2} > \frac{n+s-1}{2}=\deg a_{[s;n]}(t)
\end{equation*}
for $n \equiv 0 \Mod 4$ and $s \in I_0^-$, 
all exponents $n-2j$ for $t^{n-2j}$ in \eqref{eqn:maA} is 
$n-2j > \deg a_{[s;n]}(t)$.
Therefore,
\begin{equation*}
\pi_n(m_1^{\textnormal{II}})a_{[s;n]}(t)=c_2 b_{[s;n]}(t).
\end{equation*}
Further, by \eqref{eqn:maA} and \eqref{eqn:bnsum1}, we then have
\begin{equation*}
\wA_{\frac{n+s-1}{4}}(s;n)=c_2B_0(s;n)=c_2.
\end{equation*}
Now \eqref{eqn:ACsn} concludes the assertion.
\end{proof}

\begin{prop}\label{prop:C2}
Suppose that 
$n \equiv k \Mod 4$ and $s \in I_k^-$ 
for $k = 0, 2$.
Then $M$ acts on $\C c^{\pm}_{[s;n]}(t)$ as the following characters.
\begin{enumerate}[{\quad \normalfont (1)}]\item $n \equiv 0 \Mod 4$ and $s \in I_0^-$.
\begin{equation*}
\C c^+_{[s;n]}(t) \simeq \pp
\quad
\textnormal{and}
\quad
\C c^-_{[s;n]}(t) \simeq \mip.
\end{equation*}
\vskip 0.05in

\item $n \equiv 2 \Mod 4$ and $s \in I_2^-$.
\begin{equation*}
\C c^+_{[s;n]}(t) \simeq \mm
\quad
\textnormal{and}
\quad
\C c^-_{[s;n]}(t) \simeq \pmi.
\end{equation*}

\end{enumerate}
\end{prop}

\begin{proof}
As in Lemma \ref{lem:C2}, we only show the case of $k=0$;
the other case can be treadted similarly 
(including the exceptional case for $n\equiv 2 \Mod 4$ and $s=-1$). 
Let $n\equiv 0 \Mod 4$ and $s \in I_0^-$. 
To show the assertion it suffices to consider only $m_1^{\textnormal{II}}$.
Indeed, since the exponent $\frac{1+n-s}{2}$ for $t^{\frac{1+n-s}{2}}$
in \eqref{eqn:bnsum1} is even, 
$a_{[s;n]}(t)$ and $b_{[s;n]}(t)$ are both even functions;
thus, $m_2^{\textnormal{II}}$ acts on $\C c^{\pm}_{[s;n]}(t)$
trivially by \eqref{eqn:pi-wMT}.
To consider the action of $m_1^{\textnormal{II}}$ on $\C c^{\pm}_{[s;n]}(t)$,
observe that, by Lemma \ref{lem:C2}, we have
\begin{equation*}
\pi_n(m_1^{\textnormal{II}})b_{[s;n]}(t) = C(s;n)^{-1}a_{[s;n]}(t).
\end{equation*}
Therefore, for $\eps \in \{+, -\}$,
$m_1^{\textnormal{II}}$ transforms $c^{\eps}_{[s;n]}(t)$ as
\begin{equation*}
m_1^{\textnormal{II}}\colon c^{\eps}_{[s;n]}(t) \To \eps c^{\eps}_{[s;n]}(t).
\end{equation*}
Now Table \ref{table:char2} concludes the assertion. 
\end{proof}

\subsection{The classification of  $\Hom_M(\Sol_{\textnormal{II}}(s;n),\sigma)$}
As Step C in the recipe in Section \ref{subsec:recipe2},
we now classify $(\sigma, s, n) \in \Irr(M) \times \C \times \Z_{\geq 0}$
such that $\Hom_M(\Sol_{\textnormal{II}}(s;n), \sigma) \neq \{0\}$.
Given $\sigma \in \Irr(M)$, we define 
$I^{\pm}(\pmi)$, $I^{\pm}(\mip)$, $I(\mathbb{H}) \subset \Z \times \Z_{\geq 0}$
as follows.
\begin{alignat}{3}
\begin{aligned}\label{eqn:Ipm}
I^+(\pmi)&:=\{(s,n) \in (3+4\Z_{\geq 0}) \times (4\Z_{\geq 0}): n > s \},\\
I^-(\pmi)&:=\{(s,n) \in -(1+4\Z_{\geq 0}) \times (2+4\Z_{\geq 0}) : n > |s| \},\\[5pt]
I^+(\mip)&:=\{(s,n) \in (1+4\Z_{\geq 0}) \times (2+4\Z_{\geq 0}): n > s \},\\
I^-(\mip)&:=\{(s,n) \in -(3+4\Z_{\geq 0}) \times (4\Z_{\geq 0}) : n > |s| \},\\[5pt]
I(\mathbb{H})&:=\{(s,n) \in (2\Z) \times (1+2\Z_{\geq 0}) : n > |s| \}.
\end{aligned}
\end{alignat}

\begin{thm}\label{thm:HomHGE}
The following conditions on 
$(\sigma, s, n) \in \Irr(M) \times \C \times \Z_{\geq 0}$ are equivalent.

\begin{enumerate}[{\normalfont (i)}]

\item $\Hom_M(\Sol_{\textnormal{II}}(s;n), \sigma) \neq \{0\}$.

\item $\dim_\C\Hom_M(\Sol_{\textnormal{II}}(s;n), \sigma) = 1$.

\item One of the following conditions holds.

\begin{itemize}

\item $\sigma = \pp:$
$(s, n) \in \C \times 4\Z_{\geq 0}$.

\item $\sigma = \mm:$
$(s, n) \in \C \times (2+4\Z_{\geq 0})$.

\item $\sigma=\pmi:$
$(s,n) \in I^+(\pmi)\cup I^-(\pmi)$.

\item $\sigma=\mip:$
$(s,n) \in I^+(\mip)\cup I^-(\mip)$.

\item $\sigma=\mathbb{H}:$
\hspace{13pt}
$(s,n) \in I(\mathbb{H})$.
\end{itemize}
\end{enumerate}

Further, for such $(\sigma, s, n)$, 
the space $\Hom_M(\Sol_{\textnormal{II}}(s;n), \sigma)$ is given as follows.

\begin{enumerate}[{\qquad \normalfont (1)}]

\item $\sigma = \pp:$  For $n \in 4\Z_{\geq 0}$, we have
\begin{equation*}
\Hom_M(\Sol_{\textnormal{II}}(s;n), \pp)
=
\begin{cases}
\C a_{[s;n]}(t) & \textnormal{if $s \in \C \backslash (I^-_0 \cup J_0)$},\\
\C b_{[s;n]}(t) & \textnormal{if $s \in J_0$},\\
\C c^+_{[s;n]}(t) & \textnormal{if $s \in I^-_0$}.
\end{cases}
\end{equation*}

\item $\sigma = \mm:$ For $n \in 2+4\Z_{\geq 0}$, we have
\begin{equation*}
\Hom_M(\Sol_{\textnormal{II}}(s;n), \mm)
=
\begin{cases}
\C a_{[s;n]}(t) & \textnormal{if $s \in \C \backslash (I^-_2 \cup J_2)$},\\
\C b_{[s;n]}(t) & \textnormal{if $s \in J_2$},\\
\C c^+_{[s;n]}(t) & \textnormal{if $s \in I^-_2$.}
\end{cases}
\end{equation*}

\item $\sigma=\pmi:$ We have
\begin{equation*}
\Hom_M(\Sol_{\textnormal{II}}(s;n), \pmi) 
=
\begin{cases}
\C b_{[s;n]}(t) & \textnormal{if $(s,n) \in I^+(\pmi)$},\\
\C c^-_{[s;n]}(t) & \textnormal{if $(s,n) \in I^-(\pmi)$}.
\end{cases}
\end{equation*}

\item $\sigma=\mip:$ We have
\begin{equation*}
\Hom_M(\Sol_{\textnormal{II}}(s;n), \mip) 
=
\begin{cases}
\C b_{[s;n]}(t) & \textnormal{if $(s,n) \in I^+(\mip)$},\\
\C c^-_{[s;n]}(t) & \textnormal{if $(s,n) \in I^-(\mip)$}.
\end{cases}
\end{equation*}

\item $\sigma=\mathbb{H}:$ We have 
\begin{equation*}
\Hom_M(\Sol_{\textnormal{II}}(s;n), \mathbb{H}) 
=\C \varphi^{(s;n)}_{\textnormal{II}}
\quad \textnormal{for $(s,n) \in I(\mathbb{H})$},
\end{equation*}
where $\varphi^{(s;n)}_{\textnormal{II}}$ is a non-zero $M$-isomorphism
\begin{equation*}
\varphi^{(s;n)}_{\textnormal{II}}\colon \Sol_{\textnormal{II}}(s;n) 
\stackrel{\sim}{\To} \mathbb{H}.
\end{equation*}

\end{enumerate}
\end{thm}

\begin{proof}
The theorem simply follows from Theorems \ref{thm:Sol1} and \ref{thm:Sol2}.
Indeed, for instance, suppose that $\sigma =\pp$. It then follows from 
Theorem \ref{thm:Sol2} that 
$\Hom_M(\Sol_{\textnormal{II}}(s;n),\pp) \neq \{0\}$ if and only if 
$n\equiv 0 \Mod 4$. Moreover, 
Theorems \ref{thm:Sol1} and \ref{thm:Sol2} show that
\begin{equation*}
\Hom_M(\Sol_{\textnormal{II}}(s;n), \pp)
=
\begin{cases}
\C a_{[s;n]}(t) & \textnormal{if $s \in \C \backslash (I^-_0 \cup J_0)$},\\
\C b_{[s;n]}(t) & \textnormal{if $s \in J_0$},\\
\C c^+_{[s;n]}(t) & \textnormal{if $s \in I^-_0$}.
\end{cases}
\end{equation*} 
This concludes the assertion for $\sigma=\pp$. 
Similarly, suppose that $\sigma = \mathbb{H}$. 
Then, by Theorem \ref{thm:Sol2}, we have
$\Hom_M(\Sol_{\textnormal{II}}(s;n),\mathbb{H}) \neq \{0\}$ if and only if 
$(s,n) \in I(\mathbb{H})$. Furthermore, in this case,
$\Sol_{\textnormal{II}}(s;n) \simeq \mathbb{H}$. Therefore,
$\Hom_M(\Sol_{\textnormal{II}}(s;n),\mathbb{H})$ is spanned by 
a non-zero $M$-isomorphism 
$\varphi^{(s;n)}_{\textnormal{II}}\colon \Sol_{\textnormal{II}}(s;n) 
\stackrel{\sim}{\to} \mathbb{H}$.
Since the other cases can be handled similarly, we omit the proof. 
\end{proof}

\begin{rem}\label{rem:HomSolT}
As for Theorem \ref{thm:Sol1},
it will be shown in Theorem \ref{thm:HomHeun} that
$\Hom_M(\Sol_{\textnormal{I}}(s;n),\sigma)$ 
is simpler than $\Hom_M(\Sol_{\textnormal{II}}(s;n), \sigma)$.
\end{rem}

Now we give the $K$-type formulas for $\CSol(s;n)_K$ 
for each $\sigma \in \Irr(M)$ in the polynomial realization
\eqref{eqn:IrrK} of $\Irr(K)$.

\begin{thm}\label{thm:K-type}
The following conditions on $(\sigma, s) \in \Irr(M) \times \C$ 
are equivalent.

\begin{enumerate}[{\normalfont (i)}]
\item $\CSol(s;\sigma)_K\neq \{0\}$.
\item One of the following conditions holds.

\begin{itemize}
\item $\sigma = \pp:$ $s \in \C$.
\item $\sigma=\mm:$ $s \in \C$.
\item $\sigma=\pmi:$ $s \in 3+4\Z$.
\item $\sigma=\mip:$ $s \in 1+4\Z$.
\item $\sigma=\mathbb{H}:$ \hspace{13pt} $s \in 2\Z$
\end{itemize}
\end{enumerate}

Further, the $K$-type formulas for $\CSol(s;\sigma)_K$ 
may be given as follows.

\begin{enumerate}[{\qquad \normalfont (1)}]
\item $\sigma = \pp:$ 
\begin{equation*}
\CSol(s;\pp)_K \simeq 
\bigoplus_{n\geq 0} \Pol_{4n}[t] 
\quad\; \;\;\textnormal{for all $s \in \C$}.
\end{equation*}

\item $\sigma=\mm:$
\begin{equation*}
\CSol(s;\mm)_K \simeq 
\bigoplus_{n\geq 0} \Pol_{2+4n}[t]
\quad \textnormal{for all $s \in \C$}.
\end{equation*}

\item $\sigma=\pmi:$
\begin{equation*}
\qquad 
\CSol(s;\pmi)_K
\simeq
\bigoplus_{n\geq 0}
\Pol_{|s|+1+4n}[t] \quad \textnormal{for $s \in 3+4\Z$}.
\end{equation*}

\item $\sigma=\mip:$
\begin{equation*}
\qquad 
\CSol(s;\mip)_K
\simeq
\bigoplus_{n\geq 0}
\Pol_{|s|+1+4n}[t] \quad \textnormal{for $s \in 1+4\Z$}.
\end{equation*}

\item $\sigma=\mathbb{H}:$
\begin{equation*}
\qquad 
\CSol(s;\mathbb{H})_K
\simeq
\bigoplus_{n\geq 0}
\Pol_{|s|+1+4n}[t] \quad \textnormal{for $s \in 2\Z$}.
\end{equation*}

\end{enumerate}
\end{thm}

\begin{proof}
We only give a proof for $\sigma = \pmi$; the other cases may be handled
similarly. Suppose that $\sigma =\pmi$. By \eqref{eqn:PWSL2} with $J=\textnormal{II}$,
we have
\begin{equation*}
\CSol(s;\pmi)_K
\simeq 
\bigoplus_{n\geq 0}
\Pol_n[t] \otimes 
\mathrm{Hom}_{M}\left(
\Sol_{\textnormal{II}}(s;n), \pmi\right).
\end{equation*}
Thus, $\CSol(s;\pmi)_K \neq \{0\}$ if and only if 
$\mathrm{Hom}_{M}\left(
\Sol_{\textnormal{II}}(s;n), \pmi\right)\neq \{0\}$ for some $n \in \Z_{\geq 0}$,
which is further equivalent to 
\begin{equation*}
s \in (3+4\Z_{\geq 0}) \cup (-(1+4\Z_{\geq 0})) = 3+4\Z
\end{equation*}
by Theorem \ref{thm:HomHGE}. This shows the 
assertion
between (i) and (ii) for $\sigma = \pmi$.

To determine the explicit branching law, observe that, by the 
Frobenius reciprocity, we have
\begin{equation*}
\Hom_K(\Pol_n[t], \CSol(s;\pmi)_K) \neq \{0\}
\quad
\text{if and only if}
\quad
\Hom_M(\Sol_{\textnormal{II}}(s;n), \pmi) \neq \{0\},
\end{equation*}
which is equivalent to
\begin{equation*}
\textnormal{$n \in 4\Z_{\geq 0}$ with $n > s$ for $s \in 3 +4\Z_{\geq 0}$}
\end{equation*}
or
\begin{equation*}
\textnormal{$n \in 2+4\Z_{\geq 0}$ with $n > |s|$ for $s \in -(1 +4\Z_{\geq 0})$}
\end{equation*}
by Theorem \ref{thm:HomHGE}. 
It also follows from Theorem \ref{thm:HomHGE} that
$\CSol(s;\pmi)_K$ is multiplicity-free. 
Therefore, we have
\begin{equation*}
\CSol(s;\pmi)_K
\simeq 
\begin{cases}
\begin{aligned}
\bigoplus\limits_{\substack{n\equiv 0 \Mod 4\\ n> s}}
\Pol_n[t] \quad & \textnormal{if $s \in 3+4\Z_{\geq 0}$},\\[1ex]
\bigoplus\limits_{\substack{n\equiv 2 \Mod 4\\ n> |s|}}
\Pol_n[t] \quad & \textnormal{if $s \in -(1+4\Z_{\geq 0})$},\\
\end{aligned}
\end{cases}\\[2ex]
\end{equation*}
which is equivalent to the proposed formula.
\end{proof}


\begin{rem}\label{rem:Tamori}
The $K$-type formula for the case of $\sigma =\mathbb{H}$ is obtained 
also in \cite[Prop.\ 5.4.6]{Tamori20}. 
\end{rem}

We close this section by giving a proof of Theorem \ref{thm:K-type0}
in the introduction.

\begin{proof}[Proof of Theorem \ref{thm:K-type0}]
The assertions simply follow from Theorem \ref{thm:K-type}.
Indeed, as $\CSol(s;n)_K$ is dense in $\CSol(s;n)$, we have 
$\CSol(s;n)_K \neq \{0\}$ if and only if $\CSol(s;n) \neq \{0\}$.
Then the equivalence $(\pi_n,\Pol_n[t]) \simeq (n/2)$
concludes the theorem.
\end{proof}


\section{Heun model $\dpin^{\textnormal{I}}(D^\flat_s)f(t)=0$}\label{sec:Heun}

The purpose of this short section is to interpret the results in Section \ref{sec:HGE}
to the Heun model $\dpin^{\textnormal{I}}(D^\flat_s)f(t)=0$. 
In particular we give a variant of Theorem \ref{thm:HomHGE}
for $\Hom_M(\Sol_{\textnormal{I}}(s;n),\sigma)$
in Theorem \ref{thm:HomHeun}.
We remark that the results in this section 
will play a key role to compute certain
tridiagonal determinants in Section \ref{sec:PQ}.

\subsection{Relationships between  
$a_{[s;n]}(t),b_{[s;n]}(t),c^{\pm}_{[s;n]}(t)$ 
and 
$u_{[s;n]}(t),v_{[s;n]}(t)$}\label{subsec:SolTKO}
As for the hypergeometric model 
$\dpin^{\textnormal{II}}(D^\flat_s)f(t)=0$, 
we start with Step A
of the recipe in Section \ref{subsec:recipe2}, that is, the classification 
of $(s,n) \in \C \times \Z_{\geq 0}$ such that $\Sol_{\textnormal{I}}(s;n)\neq \{0\}$.
Recall from \eqref{eqn:KOuv} that we have
\begin{equation}\label{eqn:KOuv2}
\Sol_{\textnormal{I}}(s;n) \subset \C u_{[s;n]}(t) \oplus \C v_{[s;n]}(t)
\end{equation}
with
\begin{align*}
&u_{[s;n]}(t)=
Hl(-1, -\frac{ns}{4}; -\frac{n}{2}, -\frac{n-1}{2}, \frac{1}{2}, \frac{1-n-s}{2};t^2),\\[3pt]
&v_{[s;n]}(t)=
tHl(-1, -\frac{n-2}{4}s; -\frac{n-1}{2}, -\frac{n-2}{2}, \frac{3}{2}, \frac{1-n-s}{2};t^2).
\end{align*}
Thus, equivalently, we wish to classify $(s,n)$ such that 
$u_{[s;n]}(t)\in \Pol_n[t]$ or $v_{[s;n]}(t) \in \Pol_n[t]$.

To make full use of the results for $\Sol_{\textnormal{II}}(s;n)$, we first consider
a transfer of elements in $\Sol_{\textnormal{II}}(s;n)$ 
to $\Sol_{\textnormal{I}}(s;n)$.
Recall from \eqref{eqn:k0} and \eqref{eqn:pik0} that 
\begin{equation*}
k_0 = \frac{1}{\sqrt{2}}
\begin{pmatrix}
\sqrt{-1} & 1\\
-1 & -\sqrt{-1}
\end{pmatrix}\in SU(2)
\end{equation*}
gives an $M$-isomorphism 
\begin{equation}
\pi_n(k_0)\colon \Sol_{\textnormal{II}}(s;n) \stackrel{\sim}{\To} \Sol_{\textnormal{I}}(s;n).
\end{equation}
For $p(t), q(t) \in \Pol_n[t]$ with $\pi_n(k_0)p(t) \in \C q(t)$,
we write
\begin{equation*}
p(t) \stackrel{k_0}{\sim} q(t).
\end{equation*}
Then, for $n \in 2\Z_{\geq 0}$,  the polynomials
$a_{[s;n]}(t), b_{[s;n]}(t), c^{\pm}_{[s;n]}(t) 
\in \Sol_{\textnormal{II}}(s;n)$ are 
transferred to 
$u_{[s;n]}(t), v_{[s;n]}(t) \in \Sol_{\textnormal{I}}(s;n)$ 
via $\pi_n(k_0)$ as follows.

\begin{lem}\label{lem:n0}
Let $n \in 2\Z_{\geq 0}$. We have the following.

\begin{enumerate}[{\normalfont (1)}]
\item $n \equiv 0 \Mod 4:$
\vskip 0.1in

\begin{enumerate}[{\normalfont (a)}]\item $a_{[s;n]}(t) \stackrel{k_0}{\sim} u_{[s;n]}(t)$
\quad \textnormal{for $s \in \C\backslash (I^-_0 \cup J_0)$}.
\vskip 0.1in

\item $b_{[s;n]}(t) \stackrel{k_0}{\sim} u_{[s;n]}(t)$
\quad
\textnormal{for $s \in J_0$}.
\vskip 0.1in

\item $c^+_{[s;n]}(t) \stackrel{k_0}{\sim} u_{[s;n]}(t)$
\quad 
\textnormal{for $s\in I_0^-$}.
\vskip 0.1in

\item $b_{[s;n]}(t) \stackrel{k_0}{\sim} v_{[s;n]}(t)$
\quad 
\textnormal{for $s \in I^+_0$}.
\vskip 0.1in

\item $c^-_{[s;n]}(t) \stackrel{k_0}{\sim} v_{[s;n]}(t)$
\quad
\textnormal{for $s \in I^-_0$}.
\vskip 0.1in

\end{enumerate}
\vskip 0.1in

\item $n \equiv 2 \Mod 4:$
\vskip 0.1in

\begin{enumerate}[{\normalfont (a)}]
\item $a_{[s;n]}(t) \stackrel{k_0}{\sim} v_{[s;n]}(t)$
\quad
\textnormal{for $s \in \C\backslash (I^-_2 \cup J_2)$}.
\vskip 0.1in

\item $b_{[s;n]}(t) \stackrel{k_0}{\sim} v_{[s;n]}(t)$
\quad
\textnormal{for $s \in J_2$}.
\vskip 0.1in

\item $c^+_{[s;n]}(t) \stackrel{k_0}{\sim} v_{[s;n]}(t)$
\quad
\textnormal{for $s\in I_2^-$}.
\vskip 0.1in

\item $b_{[s;n]}(t) \stackrel{k_0}{\sim} u_{[s;n]}(t)$
\quad 
\textnormal{for $s \in I^+_2$}.
\vskip 0.1in

\item $c^-_{[s;n]}(t) \stackrel{k_0}{\sim} u_{[s;n]}(t)$
\quad
\textnormal{for $s \in I^-_2$}.
\vskip 0.1in

\end{enumerate}

\end{enumerate}

\end{lem}

\begin{proof}
We only give a proof of (1)(a); the other cases can be shown similarly.
Let $n\equiv 0 \Mod 4$ and $s \in \C \backslash (I_0^-\cap J_0)$.
It follows from Theorem \ref{thm:Sol1} that in this case
$a_{[s;n]}(t) \in \Sol_{\textnormal{II}}(s;n)$
and thus $\pi_n(k_0)a_{[s;n]}(t) \in \Sol_{\textnormal{I}}(s;n)$. By \eqref{eqn:KOuv2},
this implies that $\pi_n(k_0)a_{[s;n]}(t) = c_1 u_{[s;n]}(t) + c_2 v_{[s;n]}(t)$
for some constants $c_1, c_2 \in \C$.
We wish to show $c_2 =0$. To do so, it suffices to show that 
$\pi_n(k_0)a_{[s;n]}(t)$ is an even function, as
$u_{[s;n]}(t)$ and $v_{[s;n]}(t)$ are even and odd functions, respectively.
It follows from Theorem \ref{thm:Sol2} that $m_1^\textnormal{II}$ acts on 
$a_{[s;n]}(t)$ trivially. Since the linear map 
$\pi_n(k_0)\colon \Sol_{\textnormal{II}}(s;n) \to \Sol_{\textnormal{I}}(s;n)$
is an $M$-isomorphism, by \eqref{eqn:mTKO}, this implies that 
$m_1^{\textnormal{I}}$ acts on $\pi_n(k_0)a_{[s;n]}(t)$ trivially. 
Then, by \eqref{eqn:pi-wM} and the assumption $n \equiv 0 \Mod 4$, we have
\begin{align*}
\pi_n(k_0)a_{[s;n]}(-t)
&=(\sqrt{-1})^n \pi_n(k_0)a_{[s;n]}(-t)\\
&=\pi_n(m_1^{\textnormal{I}})\pi_n(k_0)a_{[s;n]}(t)\\
&=\pi_n(k_0)a_{[s;n]}(t).
\end{align*} 
Now the assertion follows.
\end{proof}

\begin{rem}\label{rem:S0F2}
Lemma \ref{lem:n0} in particular 
gives Gauss-to-Heun transformations.
Moreover, by Remark \ref{rem:S0F}, 
the transformations (1)(a) and (2)(a) reduce to 
Gauss-to-Gauss transformations for $s=0$.
\end{rem}

\begin{prop}\label{prop:upoly}
Given $n \in \Z_{\geq 0}$, 
the following conditions on $s\in\C$ are equivalent.

\begin{enumerate}[{\normalfont (i)}]
\item $u_{[s;n]}(t) \in \Pol_n[t]$.
\item One of the following holds.
\begin{enumerate}[{\normalfont (a)}]\item $n \equiv 0 \Mod 4:$ $s\in \C$.
\item $n \equiv 1 \Mod 4:$ $s \in I_1$.
\item $n \equiv 2 \Mod 4:$ $s \in I^+_2 \cup I^-_2$.
\item $n \equiv 3 \Mod 4:$ $s \in I_3$.
\end{enumerate}
\end{enumerate}
\end{prop}

\begin{proof}
For the case of $n$ even, the assertions simply follow from 
Theorem \ref{thm:Sol1} and Lemma \ref{lem:n0}. 
For the case of $n$ odd, suppose that $n \equiv 1 \Mod 4$.
By Theorem \ref{thm:Sol1}, we have
$\Sol_{\textnormal{II}}(s;n) \neq \{0\}$ if and only if $s \in I_1$.
Moreover, for such $s \in I_1$, the dimension $\dim_\C \Sol_{\textnormal{II}}(s;n)$ is
$\dim_\C \Sol_{\textnormal{II}}(s;n) =2$.
Thus $\dim_\C \Sol_{\textnormal{I}}(s;n) =2$ for $s \in I_1$.
Now the assertion follows from \eqref{eqn:KOuv2}.
Since the case of $n\equiv 3\Mod 4$ can be handled similarly, we omit the proof.
\end{proof}

\begin{prop}\label{prop:vpoly}
Given $n \in \Z_{\geq 0}$, 
the following conditions on $s\in\C$ are equivalent.

\begin{enumerate}[{\normalfont (i)}]

\item $v_{[s;n]}(t) \in \Pol_n[t]$.
\item One of the following holds.
\begin{enumerate}[{\normalfont (a)}]\item $n \equiv 0 \Mod 4:$ $s \in I^+_0 \cup I^-_0$.
\item $n \equiv 1 \Mod 4:$ $s \in I_1$.
\item $n \equiv 2 \Mod 4:$ $s\in \C$.
\item $n \equiv 3 \Mod 4:$ $s \in I_3$.
\end{enumerate}
\end{enumerate}
\end{prop}

\begin{proof}
Since the proof is similar to the one for Proposition \ref{prop:upoly},
we omit the proof.
\end{proof}

\begin{thm}\label{thm:SolKO1}
The following conditions on $(s,n) \in \C \times \Z_{\geq 0}$ are 
equivalent.
\begin{enumerate}[{\normalfont (i)}]
\item $\Sol_{\textnormal{I}}(s;n) \neq \{0\}$. 
\item One of the following conditions is satisfied.
\begin{itemize}
\item $n \equiv 0 \Mod 4:$ $s \in \C$.
\item $n \equiv 1 \Mod 4:$ $s \in I_1$.
\item $n\equiv 2 \Mod 4:$ $s \in \C$.
\item $n \equiv 3 \Mod 4:$ $s \in I_3$.
\end{itemize}
\end{enumerate}
Moreover,
for such $(s,n)$,
the space $\Sol_{\textnormal{I}}(s;n)$ may be described as follows.
\begin{enumerate}[{\quad \normalfont (1)}]
\item 
$n \equiv 0 \Mod 4:$ 
\begin{equation*}
\Sol_{\textnormal{I}}(s;n)
=
\begin{cases}
\C u_{[s;n]}(t) & \textnormal{if $s \in \C\backslash (I^+_0 \cup I^-_0)$},\\
\C u_{[s;n]}(t) \oplus \C v_{[s;n]}(t) & \textnormal{if $s \in I^+_0$},\\
\C u_{[s;n]}(t) \oplus \C v_{[s;n]}(t) & \textnormal{if $s \in I^-_0$}.
\end{cases}
\end{equation*}

\item $n \equiv 1 \Mod 4:$ 
\begin{equation*}
\Sol_{\textnormal{I}}(s;n) = \C u_{[s;n]}(t) \oplus \C v_{[s;n]}(t) \quad 
\textnormal{for $s \in I_1$}.
\end{equation*}

\item 
$n \equiv 2 \Mod 4:$ 
\begin{equation*}
\Sol_{\textnormal{I}}(s;n)
=
\begin{cases}
\qquad \qquad \quad
\C v_{[s;n]}(t) & \textnormal{if $s \in \C\backslash (I^+_2 \cup I^-_2)$},\\
\C u_{[s;n]}(t) \oplus \C v_{[s;n]}(t) & \textnormal{if $s \in I^+_2$},\\
\C u_{[s;n]}(t) \oplus \C v_{[s;n]}(t) & \textnormal{if $s \in I^-_2$}.
\end{cases}
\end{equation*}
\vskip 0.1in

\item $n \equiv 3 \Mod 4:$ 
\begin{equation*}
\Sol_{\textnormal{I}}(s;n) = \C u_{[s;n]}(t) \oplus \C v_{[s;n]}(t), \quad 
\textnormal{for $s \in I_3$}.
\end{equation*}

\end{enumerate}
\end{thm}

\begin{proof}
This is a summary of the results in
Propositions \ref{prop:upoly} and \ref{prop:vpoly}.
\end{proof}

\begin{thm}\label{thm:SolKO2}
For each $(s,n) \in \C \times \Z_{\geq 0}$ determined in Theorem \ref{thm:SolKO1},
the $M$-representations on $\Sol_{\textnormal{II}}(s;n)$ are classified as follows.

\begin{enumerate}[{\normalfont (a)}]
\item 
$n \equiv 0 \Mod 4:$ 
\begin{equation*}
\Sol_{\textnormal{I}}(s;n)
\simeq
\begin{cases}
\pp & \textnormal{if $s \in \C\backslash (I_0^+ \cup I_0^-)$},\\
\pp \oplus \pmi &\textnormal{if $s \in I_0^+$},\\
\pp \oplus \mip & \textnormal{if $s\in I_0^-$}.
\end{cases}
\end{equation*}
\vskip 0.1in

\item $n \equiv 1 \Mod 4:$ 
\begin{equation*}
\Sol_{\textnormal{I}}(s;n) \simeq \mathbb{H} \quad \textnormal{for $s\in I_1$}.
\end{equation*}

\item 
$n \equiv 2 \Mod 4:$ 
\begin{equation*}
\Sol_{\textnormal{I}}(s;n)
\simeq
\begin{cases}
\qquad \quad \;\; \mm & \textnormal{if $s \in \C\backslash (I_2^+ \cup I_2^-)$},\\
\mip \oplus \mm  &\textnormal{if $s \in I_2^+$},\\
\pmi \oplus \mm  & \textnormal{if $s\in I_2^-$}.
\end{cases}
\end{equation*}
\vskip 0.1in

\item $n \equiv 3 \Mod 4:$ 
\begin{equation*}
\Sol_{\textnormal{I}}(s;n)\simeq \mathbb{H} \quad \textnormal{for $s\in I_3$}.
\end{equation*}

\end{enumerate}
Here the characters $(\eps, \eps')$ stand for 
the ones on $\C u_{[s;n]}(t)$ and $\C v_{[s;n]}(t)$
at the same places in
Theorem \ref{thm:SolKO1}.

\end{thm}

\begin{proof}
The case of $n$ odd, the assertions simply follow from
Theorem \ref{thm:Sol2} via the $M$-isomorphism
$\pi_n(k_0)\colon \Sol_{\textnormal{II}}(s;n)\stackrel{\sim}{\to} \Sol_{\textnormal{I}}(s;n)$.
Similarly, the assertions for $n$ even are drawn by
Lemma \ref{lem:n0}, Theorem \ref{thm:SolKO1}, and
Propositions \ref{prop:A2}, \ref{prop:B2}, and \ref{prop:C2}.
\end{proof}

Recall from \eqref{eqn:Ipm} the subsets 
$I^{\pm}(\pmi)$, $I^{\pm}(\mip)$, $I(\mathbb{H}) \subset \Z \times \Z_{\geq 0}$.
We now give the explicit description of $\Hom_M(\Sol_{\textnormal{I}}(s;n), \sigma) \neq \{0\}$.

\begin{thm}\label{thm:HomHeun}
The following conditions on 
$(\sigma, s, n) \in \Irr(M) \times \C \times \Z_{\geq 0}$ are equivalent.

\begin{enumerate}[{\normalfont (i)}]

\item $\Hom_M(\Sol_{\textnormal{I}}(s;n), \sigma) \neq \{0\}$.

\item $\dim_\C\Hom_M(\Sol_{\textnormal{I}}(s;n), \sigma) = 1$.

\item One of the following conditions holds.

\begin{itemize}

\item $\sigma = \pp:$
$(s, n) \in \C \times 4\Z_{\geq 0}$.

\item $\sigma = \mm:$
$(s, n) \in \C \times (2+4\Z_{\geq 0})$.

\item $\sigma=\pmi:$
$(s,n) \in I^+(\pmi)\cup I^-(\pmi)$.

\item $\sigma=\mip:$
$(s,n) \in I^+(\mip)\cup I^-(\mip)$.

\item $\sigma=\mathbb{H}:$
\hspace{13pt}
$(s,n) \in I(\mathbb{H})$.
\end{itemize}
\end{enumerate}

Moreover, for such $(\sigma, s, n)$, 
the space $\Hom_M(\Sol_{\textnormal{I}}(s;n), \sigma)$ is given as follows.

\begin{enumerate}[{\qquad \normalfont (1)}]

\item $\sigma = \pp:$ For $n \in 4\Z_{\geq 0}$, we have
\begin{equation*}
\Hom_M(\Sol_{\textnormal{I}}(s;n), \pp)= \C u_{[s;n]}(t)
\quad \textnormal{for all $s \in \C$}.
\end{equation*}

\item $\sigma = \mm:$ For $n \in 2+4\Z_{\geq 0}$, we have
\begin{equation*}
\Hom_M(\Sol_{\textnormal{I}}(s;n), \mm)= \C v_{[s;n]}(t)
\quad \textnormal{for all $s \in \C$}.
\end{equation*}

\item $\sigma=\pmi:$ We have
\begin{equation*}
\Hom_M(\Sol_{\textnormal{I}}(s;n), \pmi) 
=
\begin{cases}
\C v_{[s;n]}(t) & \textnormal{if $(s,n) \in I^+(\pmi)$},\\
\C u_{[s;n]}(t) & \textnormal{if $(s,n) \in I^-(\pmi)$}.
\end{cases}
\end{equation*}

\item $\sigma=\mip:$ We have
\begin{equation*}
\Hom_M(\Sol_{\textnormal{I}}(s;n), \mip) 
=
\begin{cases}
\C u_{[s;n]}(t) & \textnormal{if $(s,n) \in I^+(\mip)$},\\
\C v_{[s;n]}(t) & \textnormal{if $(s,n) \in I^-(\mip)$}.
\end{cases}
\end{equation*}

\item $\sigma=\mathbb{H}:$ We have
\begin{equation*}
\Hom_M(\Sol_{\textnormal{I}}(s;n), \mathbb{H}) 
=\C \varphi^{(s;n)}_{\textnormal{I}}
\quad \textnormal{for $(s,n) \in I(\mathbb{H})$},
\end{equation*}
where $\varphi^{(s;n)}_{\textnormal{I}}$ is a non-zero $M$-isomorphism
\begin{equation*}
\varphi^{(s;n)}_{\textnormal{I}}\colon \Sol_{\textnormal{I}}(s;n) 
\stackrel{\sim}{\To} \mathbb{H}.
\end{equation*}

\end{enumerate}
\end{thm}

\begin{proof}
Since the theorem can be shown similarly to Theorem \ref{thm:HomHGE},
we omit the proof.
\end{proof}


\section{Sequences $\{P_k(x;y)\}_{k=0}^{\infty}$ and $\{Q_k(x;y)\}_{k=0}^\infty$ 
of tridiagonal determinants}
\label{sec:PQ}

The aim of this section is to discuss about two sequences 
$\{P_k(x;y)\}_{k=0}^{\infty}$ and $\{Q_k(x;y)\}_{k=0}^\infty$ 
of $k \times k$ tridiagonal determinants associated to
polynomial solutions to the Heun model
$d\pi_n^{\textnormal{I}}(D^\flat_s)f(t)=0$.
We give factorization formulas for 
$P_{\left[\frac{n+2}{2}\right]}(x;n)$ and $Q_{\left[\frac{n+1}{2}\right]}(x;n)$
as well as palindromic property of 
$\{P_k(x;y)\}_{k=0}^{\infty}$ and $\{Q_k(x;y)\}_{k=0}^{\infty}$ for $n$ even.
(See Definition \ref{def:palindromic} for the definition of palindromic property.)
These are achieved in Theorems 
\ref{thm:factorPQ} and
\ref{thm:factorP}
(factorization formulas),
and \ref{thm:PalinP} and
\ref{thm:PalinQ} (palindromic property).

\subsection{Sequences 
$\{P_k(x;y)\}_{k=0}^\infty$ and $\{Q_k(x;y)\}_{k=0}^\infty$
of tridiagonal determinants}
\label{subsec:PQ}

We start with the definitions of 
$\{P_k(x;y)\}_{k=0}^\infty$ and $\{Q_k(x;y)\}_{k=0}^\infty$.
Let $a(x)$ and $b(x)$ be two polynomials such that
\begin{equation}
a(x)=2x(2x-1) 
\quad \text{and} \quad
b(x)=2x(2x+1). 
\end{equation}
For instance, for $x=1, 2, 3,\dots$, we have
\begin{alignat*}{3}
&a(1)=1 \cdot 2, \quad
&&a(2)=3 \cdot 4, \quad
&&a(3)=5 \cdot 6, \dots, \\
&b(1)=2 \cdot 3, \quad 
&&b(2)=4 \cdot 5, \quad
&&b(3)=6 \cdot 7, \dots.
\end{alignat*}
Similarly, for $x=\frac{1}{2}, \frac{3}{2}, \frac{5}{2},\ldots$, we have
\begin{alignat*}{3}
&a\left(\frac{1}{2}\right)=0 \cdot 1, \quad
&&a\left(\frac{3}{2}\right)=2 \cdot 3, \quad
&&a\left(\frac{5}{2}\right)=4 \cdot 5, \dots, \\[5pt]
&b\left(\frac{1}{2}\right)=1 \cdot 2, \quad 
&&b\left(\frac{3}{2}\right)=3 \cdot 4, \quad
&&b\left(\frac{5}{2}\right)=5 \cdot 6, \dots.
\end{alignat*}
\vskip 0.1in
\noindent
Clearly, the polynomials $a(x)$ and $b(x)$ satisfy

\begin{equation}\label{eqn:ab1}
a(x)= b\left(\frac{2x-1}{2}\right)
\quad \text{and} \quad
b(x)=a\left(\frac{2x+1}{2}\right).
\end{equation}
\vskip 0.1in
We define $k \times k$ tridiagonal determinants
$P_k(x;y)$ and $Q_k(x;y)$ 
in terms of $a(x)$ and $b(x)$ as follows.

\begin{enumerate}[(1)]

\item $P_k(x;y):$
\vskip 0.1in

\begin{itemize}

\item $k =0:$ $P_0(x;y)=1$,
\vskip 0.1in

\item $k=1:$ $P_1(x;y) = yx$,
\vskip 0.2in

\item $k \geq 2:$ $P_k(x;y) = 
\begin{small}
\begin{vmatrix}
yx& a(1)    &              &            &             &       \\
-a\left(\frac{y}{2}\right)& (y-4)x    &     a(2)       &            &             &       \\
  & -a\left(\frac{y-2}{2}\right) &     (y-8)x       &  a(3)        &             &       \\
  &       &  \dots  &  \dots & \dots  &       \\
  &       &              &-a\left(\frac{y-2k+6}{2}\right)   &(y-4k+8)x            & a(k-1) \\
  &       &              &            &-a\left(\frac{y-2k+4}{2}\right)     & (y-4k+4)x    \\
\end{vmatrix}
\end{small}$.
\end{itemize}
\vskip 0.2in

\item $Q_k(x;y):$
\vskip 0.1in
 
\begin{itemize}

\item $k =0:$ $Q_0(x;y)=1$,
\vskip 0.1in

\item $k=1:$ $Q_1(x;y) = (y-2)x$,
\vskip 0.2in

\item $k \geq 2:$ $Q_k(x;y) = 
\begin{small}
\begin{vmatrix}
(y-2)x& b(1)    &              &            &             &       \\
-b\left(\frac{y-2}{2}\right)& (y-6)x    &     b(2)       &            &             &       \\
  & -b\left(\frac{y-4}{2}\right) &     (y-10)x       &  b(3)        &             &       \\
  &       &  \dots  &  \dots & \dots  &       \\
  &       &              &-b\left(\frac{y-2k+4}{2}\right)   &(y-4k+6)x            & b(k-1) \\
  &       &              &            &-b\left(\frac{y-2k+2}{2}\right)     & (y-4k+2)x    \\
\end{vmatrix}
\end{small}
$.
\end{itemize}
\vskip 0.2in

\end{enumerate}

\begin{example}
Here are a few examples of $P_k(x;4)$ and $Q_k(x;6)$ for $k=2, 3, 4$.

\begin{enumerate}[(1)]

\item $P_k(x;4):$
\begin{equation*}
P_2(x;4)=
\begin{vmatrix}
4x & 1 \cdot 2\\
- 3 \cdot 4 & 0
\end{vmatrix},
\quad
P_3(x;4)=
\begin{small}
\begin{vmatrix}
4x & 1 \cdot 2 & \\
- 3\cdot 4 & 0 & 3 \cdot 4\\
 &  -1 \cdot 2&-4x 
\end{vmatrix},
\end{small}
\quad
P_4(x;4)=
\begin{small}
\begin{vmatrix}
4x            & 1 \cdot 2  &                   &\\
- 3\cdot 4 & 0              & 3 \cdot 4    &\\
                &  -1 \cdot 2&-4x               & 5 \cdot 6\\
                &                 &0                  &-12x 
\end{vmatrix}
\end{small}
\end{equation*}

\item  $Q_k(x;6):$ 
\begin{equation*}
Q_2(x;6)=
\begin{vmatrix}
4x & 2 \cdot 3\\
- 4 \cdot 5 & 0
\end{vmatrix},
\quad
Q_3(x;6)=
\begin{small}
\begin{vmatrix}
4x            & 2 \cdot 3       & \\
- 4\cdot 5 & 0                   & 4 \cdot 5\\
                 &  -2 \cdot 3     &-4x 
\end{vmatrix},
\end{small}
\quad
Q_4(x;6)=
\begin{small}
\begin{vmatrix}
4x            & 2 \cdot 3  &                   &\\
- 4\cdot 5 & 0              & 4 \cdot 5    &\\
                &  -2 \cdot 3&-4x               & 6 \cdot 7\\
                &                 &0                  &-12x 
\end{vmatrix}
\end{small}
\end{equation*}

\end{enumerate}

\noindent
Moreover, for instance, for $y=5$ and $k=3$, we have

\begin{equation*}
P_3(x;5)=
\begin{small}
\begin{vmatrix}
5x & 1 \cdot 2 & \\
- 4\cdot 5 & x & 3 \cdot 4\\
 &  -2 \cdot 3&-3x 
\end{vmatrix}
\end{small}
\qquad
\text{and}
\qquad
Q_3(x;5)=
\begin{small}
\begin{vmatrix}
3x            & 2 \cdot 3       & \\
- 3\cdot 4 & -x                   & 4 \cdot 5\\
                 &  -1 \cdot 2     &-5x 
\end{vmatrix}.
\end{small}
\end{equation*}
\end{example}

\vskip 0.1in

\begin{rem}
\label{rem:PQ}
In general $P_k(x;y)$ and $Q_k(x;y)$ satisfy the following properties
for specific $y$ and $k$.

\begin{enumerate}

\item
If $y=n \in 2+2\Z_{\geq 0}$, then 
$P_{\frac{n+2}{2}}(x;n)$ is anti-centrosymmetric:
\begin{equation}\label{eqn:P1}
\begin{small}
P_{\frac{n+2}{2}}(x;n) 
= 
\begin{vmatrix}
nx& a(1)    &              &            &             &       \\
-a\left(\frac{n}{2}\right)& (n-4)x    &     a(2)       &            &             &       \\
  & -a\left(\frac{n-2}{2}\right) &     (n-8)x       &  a(3)        &             &       \\
  &       &  \dots  &  \dots & \dots  &       \\
  &       &              &-a(2)   &-(n-4)x            & a\left(\frac{n}{2}\right) \\
  &       &              &            &-a(1)     & -nx    \\
\end{vmatrix}.\\[5pt]
\end{small}
\end{equation}
\vskip 0.1in

\item
If $y=n \in 4+2\Z_{\geq 0}$, then 
$Q_{\frac{n}{2}}(x;n)$ is also anti-centrosymmetric:
\begin{equation}\label{eqn:Q1}
\begin{small}
Q_{\frac{n}{2}}(x;n) = 
\begin{vmatrix}
(n-2)x& b(1)    &              &            &             &       \\
-b\left(\frac{n-2}{2}\right)& (n-6)x    &     b(2)       &            &             &       \\
  & -b\left(\frac{n-4}{2}\right) &     (n-10)x       &  b(3)        &             &       \\
  &       &  \dots  &  \dots & \dots  &       \\
  &       &              &-b(2)   &-(n-6)x            & b\left(\frac{n-2}{2}\right) \\
  &       &              &            &-b(1)     & -(n-2)x    \\
\end{vmatrix}.\\[5pt]
\end{small}
\end{equation}
\vskip 0.1in

\item 
It follows from \eqref{eqn:ab1} that
for $y=n \in 3+2\Z_{\geq 0}$, we have
\begin{align}
P_{\frac{n+1}{2}}(x;n)
&= 
\small
\begin{vmatrix}
nx& a(1)    &              &            &             &       \\
-b\left(\frac{n-1}{2}\right)& (n-4)x    &     a(2)       &            &             &       \\
  & -b\left(\frac{n-3}{2}\right) &     (n-8)x       &  a(3)        &             &       \\
  &       &  \dots  &  \dots & \dots  &       \\
  &       &              &-b\left(2\right)   &-(n-6)x            & a\left(\frac{n-1}{2}\right) \\
  &       &              &            &-b\left(1\right)     & -(n-2)x    \\
\end{vmatrix}
\normalsize
 \label{eqn:P2}\\[10pt]
Q_{\frac{n+1}{2}}(x;n) 
&= 
\small
\begin{vmatrix}
(n-2)x& b(1)    &              &            &             &       \\
-a\left(\frac{n-1}{2}\right)& (n-6)x    &     b(2)       &            &             &       \\
  & -a\left(\frac{n-3}{2}\right) &     (n-10)x       &  b(3)        &             &       \\
  &       &  \dots  &  \dots & \dots  &       \\
  &       &              &-a(2)   &-(n-4)x            & b\left(\frac{n-1}{2}\right) \\
  &       &              &            &-a(1)     & -nx    \\
\end{vmatrix}.\label{eqn:Q2}
\normalsize
\end{align}
We also have $P_1(x;1)=x$ and $Q_1(x;1)=-x$.
Therefore, 
\begin{equation}\label{eqn:QP}
Q_{\frac{n+1}{2}}(x;n) 
=
(-1)^{\frac{n+1}{2}}P_{\frac{n+1}{2}}(x;n)
\quad 
\text{for $n \in 1+2\Z_{\geq 0}$.}
\end{equation}

\end{enumerate}
\end{rem}

The tridiagonal determinants
$P_k(x;y)$
and 
$Q_k(x;y)$
enjoy the following property.

\begin{lem}\label{lem:PQs}
For $k \in \Z_{\geq 0}$, we have
\begin{equation*}
P_k(-x;y)=(-1)^kP_k(x;y) 
\quad \text{and} \quad 
Q_k(-x;y)=(-1)^kQ_k(x;y).
\end{equation*}
\end{lem}

\begin{proof}
The identities for $k=0,1$ clearly hold by definition.
The assertions for $k\geq 2$ follow from the following 
general property of $k \times k$ tridiagonal determinants:

\begin{equation}\label{eqn:tridiag}
\begin{small}
\begin{vmatrix}
-a_1& b_1    &              &            &             &      \\
c_1& -a_2    &     b_2       &            &             &    \\
  & c_2 &     -a_3       &  b_3        &             &       &\\
  &       &  \dots  &  \dots & \dots  &       &\\
  &       &              &c_{k-2}          &-a_{k-1}            & b_{k-1} \\
  &       &              &            &c_{k-1}            & -a_{k} \\
\end{vmatrix}
\end{small}
=(-1)^k
\begin{small}
\begin{vmatrix}
a_1& b_1    &              &            &             &      \\
c_1& a_2    &     b_2       &            &             &    \\
  & c_2 &   a_3       &  b_3        &             &       &\\
  &       &  \dots  &  \dots & \dots  &       &\\
  &       &              &c_{k-2}          & a_{k-1}            & b_{k-1} \\
  &       &              &            &c_{k-1}            & a_{k} \\
\end{vmatrix}.
\end{small}
\qedhere
\end{equation}
\end{proof}
\vskip 0.1in

From the next subsections $y$ is taken to be $y=n \in \Z_{\geq 0}$
and we shall discuss several properties of $\{P_k(x;n)\}_{k=0}^\infty$ 
and $\{Q_k(x;n)\}_{k=0}^\infty$.

\subsection{Generating functions of $\{P_k(x;n)\}_{k=0}^\infty$ and 
$\{Q_k(x;n)\}_{k=0}^\infty$}

We first give the generating functions of $\{P_k(x;n)\}_{k=0}^\infty$ and 
$\{Q_k(x;n)\}_{k=0}^\infty$.
Let $u_{[s;n]}(t)$ be 
the local Heun function defined in \eqref{eqn:u}.
Write
$u_{[s;n]}(t)=\sum_{k=0}^\infty U_k(s;n)t^{2k}$
for the power series expansion at $t=0$.
It follows from \eqref{eqn:uk} with \eqref{eqn:uk0} in Section \ref{sec:appendix}
that each coefficient $U_k(s;n)$ can be given as
\begin{equation}
U_k(s;n) = 
\begin{small}
\begin{vmatrix}
E^u_0& -1    &              &            &             &      \\
F^u_1& E^u_1    &     -1       &            &             &    \\
  & F^u_2 &     E^u_2       &  -1        &             &       &\\
  &       &  \dots  &  \dots & \dots  &       &\\
  &       &              &F^u_{k-2}          &E^u_{k-2}            & -1 \\
  &       &              &            &F^u_{k-1}            & E^u_{k-1} \\
\end{vmatrix} \label{eqn:uk1}\\[5pt]
\end{small}
\end{equation}
with
\begin{equation}
E^u_0=\frac{ns}{a(1)},\quad
E^u_k=\frac{(n-4k)s}{a(k+1)},
\quad \text{and} \quad
F^u_k=\frac{a\left(\frac{n-2k+2}{2}\right)}{a(k+1)}.\label{eqn:uk01}
\end{equation}
Similarly,
equation \eqref{eqn:vk} with \eqref{eqn:vk0} in Section \ref{sec:appendix}
shows that the coefficients $V_k(s;n)$ of  the power series expansion
$v_{[s;n]}(t) = \sum_{k=0}^\infty V_k(s;n)t^{2k+1}$ of 
the second solution $v_{[s;n]}(t)$ to the Heun equation \eqref{eqn:Heun1}
at $t=0$ are given as
\begin{equation}
V_k(s;n) = 
\begin{small}
\begin{vmatrix}
E^v_0& -1    &              &            &             &      \\
F^v_1& E^v_1    &     -1       &            &             &    \\
  & F^v_2 &     E^v_2       &  -1        &             &       &\\
  &       &  \dots  &  \dots & \dots  &       &\\
  &       &              &F^v_{k-2}          &E^v_{k-2}            & -1 \\
  &       &              &            &F^v_{k-1}            & E^v_{k-1} \\
\end{vmatrix}\label{eqn:vk1}
\end{small}
\end{equation}
with
\begin{equation}
E^v_0=\frac{(n-2)s}{b(1)}, \quad
E^v_k=\frac{(n-4k-2)s}{b(k+1)},
\quad \text{and} \quad
F^v_k=\frac{b\left(\frac{n-2k}{2}\right)}{b(k+1)}. \label{eqn:vk01}
\end{equation}
Proposition \ref{prop:PQ} below then shows that
$u_{[s;n]}(t)$ and $v_{[s;n]}(t)$ are in fact
the ``hyperbolic cosine'' generating function of $\{P_k(s;n)\}_{k=0}^\infty$
and ``hyperbolic sine'' generating function of $\{Q_k(s;n)\}_{k=0}^\infty$,
respectively.

\begin{prop}\label{prop:PQ}
We have
\begin{align*}
u_{[s;n]}(t) = \sum_{k=0}^\infty P_k(s;n) \frac{t^{2k}}{(2k)!}
\quad \mathrm{and} \quad 
v_{[s;n]}(t) = \sum_{k=0}^\infty Q_k(s;n) \frac{t^{2k+1}}{(2k+1)!}.
\end{align*}
\end{prop}

\begin{proof}
We only show the identity for $u_{[s;n]}(t)$; 
the assertion for $v_{[s;n]}(t)$ can be shown similarly.
We wish to show that 
$U_k(s;n)=P_k(s;n)/(2k)!$
for all $k\in \Z_{\geq 0}$.
By definition, it is clear that
$U_0(s;n)=1=P_0(s;n)$ and 
$U_1(s;n)=ns/2=P_1(s;n)/2!$.
We then assume that $k\geq 2$. It follows from \eqref{eqn:uk01} that
$U_k(s;n)$ can be given as
\begin{equation*}
U_k(s;n)
=\frac{(-1)^k}{\prod^k_{j=1}a(j)}P_k(-s;n)
=\frac{P_k(s;n)}{\prod^k_{j=1}a(j)}.
\end{equation*}
Here, Lemma \ref{lem:PQs} is applied from the second identity to the third.
Now the assertion follows from the identity
$\prod^k_{j=1}a(j) = (2k)!$.
\end{proof}

For $R_k(s;n) \in \{P_k(s;n), Q_k(s;n)\}$, we define
\begin{equation*}
\ESol_k\left(R;n\right)
:=\{s \in \C : R_k(s;n)=0\}.
\end{equation*}
\noindent
We recall from \eqref{eqn:IJ} the subsets 
$I^\pm_0$, $I_1$, $I^\pm_2$, $I_3 \subset \Z$.

\begin{cor}\label{cor:PQ}
Let $n \in 2\Z_{\geq 0}$. Then 
$\ESol_{\frac{n+2}{2}}(P;n)$ and $\ESol_{\frac{n}{2}}(Q;n)$ are given as
\begin{equation*}
\ESol_{\frac{n+2}{2}}(P;n)=
\begin{cases}
\C & \emph{if $n \equiv 0 \Mod 4$},\\[3pt]
I^+_2\cup I^-_2 & \emph{if $n\equiv 2 \Mod 4$}
\end{cases}
\end{equation*}
and 
\begin{equation*}
\ESol_{\frac{n}{2}}(Q;n)=
\begin{cases}
I^+_0\cup I^-_0 & \emph{if $n\equiv 0 \Mod 4$},\\[3pt]
\C & \emph{if $n \equiv 2 \Mod 4$}.
\end{cases}
\end{equation*}
\vskip 0.05in
\noindent
Further, for $n \in 1+2\Z_{\geq 0}$, the sets 
$\ESol_{\frac{n+1}{2}}(P;n)$ and $\ESol_{\frac{n+1}{2}}(Q;n)$ are given as
\begin{equation*}
\ESol_{\frac{n+1}{2}}(P;n)=\ESol_{\frac{n+1}{2}}(Q;n)
=
\begin{cases}
I_1 &\emph{if $n \equiv 1 \Mod 4$},\\[3pt]
I_3 &\emph{if $n \equiv 3 \Mod 4$}.
\end{cases}
\end{equation*}
\end{cor} 

\begin{proof}
It follows from Propositions \ref{prop:upoly} and \ref{prop:PQ} that
the following conditions on $s \in \C$ are equivalent for $n \in 2\Z_{\geq 0}$:
\begin{enumerate}[{\normalfont (i)}]
\item $s \in \ESol_{\frac{n+2}{2}}(P;n)$;
\vskip 5pt
\item $u_{[s;n]}(t) \in \Pol_n[t]$;
\vskip 5pt
\item $s \in 
\begin{cases}
\C & \text{if $n\equiv 0 \Mod 4$},\\
I^+_2 \cup I^-_2 & \text{if $n \equiv 2 \Mod 4$}.
\end{cases}$
\vskip 5pt
\end{enumerate} 
This concludes the assertion for $\ESol_{\frac{n+2}{2}}(P;n)$.
The other cases can be shown similarly. 
\end{proof}

\subsection{Factorization formulas of 
$P_{\left[\frac{n+2}{2}\right]}(x;n)$ and $Q_{\left[\frac{n+1}{2}\right]}(x;n)$}

We next show the factorization formulas 
for $P_{\frac{n+2}{2}}(x;n)$ 
and $Q_{\frac{n}{2}}(x;n)$   
for $n$ even,
and for $P_{\frac{n+1}{2}}(x;n)$ 
and $Q_{\frac{n+1}{2}}(x;n)$  
for $n$ odd. 
We remark that any product of the form 
$\prod^{j-1}_{\ell=j} c_\ell$  with 
$c_\ell \in \Pol[t]$ (in particular, $c_\ell \in \C$)
is regarded as $\prod^{j-1}_{\ell=j} c_\ell=1$.

\subsubsection{Factorization formulas of
$P_{\frac{n+2}{2}}(x;n)$ and $Q_{\frac{n}{2}}(x;n)$}

We start with $P_{\frac{n+2}{2}}(x;n)$ 
and $Q_{\frac{n}{2}}(x;n)$ for $n$ even (see \eqref{eqn:P1} and \eqref{eqn:Q1}).
For $n \in 2\Z_{\geq 0}$, let  $\ga_n$ and $\gb_n$ be the products of the coefficients
of $x$ on the main diagonal of $P_{\frac{n+2}{2}}(x;n)$ and $Q_{\frac{n}{2}}(x;n)$, respectively, 
Namely, we have 
\begin{equation*}
\ga_n=\prod^{\frac{n}{2}}_{\ell=0}(n-4\ell)
\quad
\text{and}
\quad
\gb_n=\prod^{\frac{n-2}{2}}_{\ell=0}(n-2-4\ell).
\end{equation*}
\noindent

It is remarked that $\ga_n$ and $\gb_n$ may be given as follows.

\begin{equation*}
\ga_n
=
\begin{cases}
0 & \text{if $n \equiv 0 \Mod 4$},\\[3pt]
(-4)^{\frac{n+2}{4}}\prod\limits^{\frac{n-2}{4}}_{\ell=0} (1+2\ell)^2
& \text{if $n \equiv 2\Mod 4$}
\end{cases}
\end{equation*}
and
\begin{equation*}
\beta_n=
\begin{cases}
(-4)^{\frac{n}{4}}\prod\limits^{\frac{n-4}{4}}_{\ell=0}(1+2\ell)^2
& \text{if $n \equiv 0 \Mod 4$}, \\[3pt]
0 & \text{if $n \equiv 2 \Mod4$}.
\end{cases}
\end{equation*}

\begin{thm}
[Factorization formulas 
of $P_{\frac{n+2}{2}}(x;n)$ and $Q_{\frac{n}{2}}(x;n)$]
\label{thm:factorPQ}
For $n \in 2\Z_{\geq 0}$, 
the polynomials $P_{\frac{n+2}{2}}(x;n)$ and $Q_{\frac{n}{2}}(x;n)$ 
are either $0$ or factored as follows.

\begin{equation}\label{eqn:factor1}
P_{\frac{n+2}{2}}(x;n) =
\begin{cases}
0 &\emph{if $n \equiv 0 \Mod 4$},\\[3pt]
\ga_n \prod\limits^{\frac{n-2}{4}}_{\ell=0}(x^2-(4\ell+1)^2) & \emph{if $n \equiv 2 \Mod 4$}
\end{cases}
\end{equation}
and
\begin{equation}\label{eqn:factor2}
Q_{\frac{n}{2}}(x;n)=
\begin{cases}
\beta_n\prod\limits^{\frac{n-4}{4}}_{\ell=0}(x^2-(4\ell+3)^2) &\emph{if $n \equiv 0 \Mod 4$},\\[3pt]
0 &\emph{if $n \equiv 2 \Mod 4$}.
\end{cases}
\end{equation}
\end{thm}

\begin{proof}
We only demonstrate the proof for $P_{\frac{n+2}{2}}(x;n)$; the assertion for 
$Q_{\frac{n}{2}}(x;n)$ can be shown similarly.
It follows from Corollary \ref{cor:PQ} that 
\begin{equation*}
\ESol_{\frac{n+2}{2}}(P;n)=
\begin{cases}
\C & \text{if $n \equiv 0 \Mod 4$},\\[3pt]
I^+_2\cup I^-_2 & \text{if $n\equiv 2 \Mod 4$}.
\end{cases}
\end{equation*}
In particular, 
as $\ga_n$ is the product of the coefficients of $x$ 
in $P_{\frac{n+2}{2}}(x;n)$, we have 
\begin{equation*}
P_{\frac{n+2}{2}}(x;n)
=\begin{cases}
0 & \text{if $n \equiv 0 \Mod 4$},\\[2pt]
\ga_n \prod\limits_{s \in I^+_2 \cup I^-_2} (x-s) & \text{if $n \equiv 2 \Mod 4$}.
\end{cases}
\end{equation*}
Now the assertion follows from the identity 
$\prod_{s \in I^+_2 \cup I^-_2} (x-s) = 
\prod^{\frac{n-2}{4}}_{\ell=0}(x^2-(4\ell+1)^2)$.
\end{proof}

\begin{rem}\label{rem:PQ1}
The factorization formulas \eqref{eqn:factor1} and \eqref{eqn:factor2}
can also be obtained from \cite[Prop.~5.11]{Kable12C}.
In fact, in \cite{Kable12C}, the dimension 
$\dim_\C \Sol_{\textnormal{I}}(s;n) (=\dim_\C\Sol_{\textnormal{II}}(s;n))$
was determined by using \eqref{eqn:factor1} and \eqref{eqn:factor2}
(see Remark \ref{rem:SolT2}).
\end{rem}

\subsubsection{Factorization formula of $P_{\frac{n+1}{2}}(x;n)$}
We next consider 
$P_{\frac{n+1}{2}}(x;n)$ and $Q_{\frac{n+1}{2}}(x;n)$
for $n$ odd (see \eqref{eqn:P2} and \eqref{eqn:Q2}). 
As shown in \eqref{eqn:QP}, the determinant
$Q_{\frac{n+1}{2}}(x;n)$ is given as
\begin{equation*}
Q_{\frac{n+1}{2}}(x;n) = (-1)^{\frac{n+1}{2}}P_{\frac{n+1}{2}}(x;n)
\quad
\text{for $n \in 1+2\Z_{\geq 0}$}.
\end{equation*}
It thus suffices to only consider $P_{\frac{n+1}{2}}(x;n)$.
For $n$ odd, let 
$\gamma_n$ be the product of the coefficients of $x$ on the main diagonal of 
$P_{\frac{n+1}{2}}(x;n)$, namely, 
\begin{equation*}
\gamma_n = \prod^{\frac{n-1}{2}}_{\ell=0}(n-4\ell).
\end{equation*}
We remark that $\gamma_n$ may be given as
\begin{equation*}
\gamma_n
=
\begin{cases}
(-1)^{\frac{n-1}{4}}\prod\limits^{\frac{n-3}{2}}_{\ell=0}(3+2\ell) & \text{if $n \equiv 1 \Mod 4$},\\[2ex]
(-1)^{\frac{n+1}{4}}\prod\limits^{\frac{n-3}{2}}_{\ell=0}(3+2\ell) & \text{if $n \equiv 3 \Mod4$}.
\end{cases}
\end{equation*}

\begin{thm}
[Factorization formula of $P_{\frac{n+1}{2}}(x;n)$]
\label{thm:factorP}
For $n \in 1+2\Z_{\geq 0}$, the polynomial $P_{\frac{n+1}{2}}(x;n)$ 
is factored as
\begin{equation}\label{eqn:factor3}
P_{\frac{n+1}{2}}(x;n) 
= 
\gamma_n 
\prod\limits^{\frac{n-1}{2}}_{\ell=0}(x-(n-1)+4\ell).
\end{equation}
Equivalently, we have
\begin{equation*}
P_{\frac{n+1}{2}}(x;n)=
\begin{cases}
\gamma_n\, x\prod\limits^{\frac{n-1}{4}}_{\ell=1}(x^2-(4\ell)^2)
& \textnormal{if $n\equiv 1 \Mod 4$},\\[5pt]
\gamma_n \prod\limits^{\frac{n-3}{4}}_{\ell=0}(x^2-(4\ell+2)^2)
& \textnormal{if $n\equiv 3 \Mod 4$}.
\end{cases}
\end{equation*}
\end{thm}

\begin{proof}
Since the argument is similar to that for Theorem \ref{thm:factorPQ}, 
we omit the proof.
\end{proof}

The factorization formula of $P_{\frac{n+1}{2}}(x;n)$ in \eqref{eqn:factor3}
suggests that one may express $P_{\frac{n+1}{2}}(x;n)$ in terms of 
the so-called \emph{Sylvester determinant} $\Sylv(x;n)$
(\eqref{eqn:Sylv_intro}).
The factorization formula \eqref{eqn:factorSylv} below for $\Sylv(x;n)$
was first observed by Sylvester in 1854 (\cite{Sylv54}).
Later, a number of people such as 
Askey (\cite{Askey05}),
Edelman--Kostlan (\cite{EK94}),
Holtz (\cite{Holtz05}),
Kac (\cite{Kac47}),
Mazza (\cite[p.\ 442]{Muir06}),
Taussky-Todd (\cite{TT91}),
among others,
gave several proofs.
For more details about the Sylvester determinant $\Sylv(x;n)$
as well as some related topics, 
see, for instance, 
\cite{FK20, FMMW13, MT05, NT12} 
and references therein.
We remark that the factorization formula \eqref{eqn:factorSylv} 
also readily follows from a general theory of $\f{sl}(2, \C)$ representations
(see Proposition \ref{prop:genCay8} and Section \ref{sec:appendixB}).

\begin{fact}[Sylvester's factorization formula]\label{fact:Sylv}
Let $n \in 1+\Z_{\geq 0}$.
The Sylvester determinant $\Sylv(x;n)$ in \eqref{eqn:Sylv_intro}
is factored as
\begin{equation}\label{eqn:factorSylv}
\Sylv(x;n)
=\prod_{\ell=0}^n(x-n+2\ell).
\end{equation}
Equivalently, we have
\begin{equation}\label{eqn:factorSylv2}
\Sylv(x;n)=
\begin{cases}
\prod\limits^{\frac{n-1}{2}}_{\ell=0}(x^2-(2\ell+1)^2) 
&\text{if $n$ is odd},\\
x\prod\limits^{\frac{n}{2}}_{\ell=1}(x^2-(2\ell)^2) 
&\text{if $n$ is even}.\\
\end{cases}
\end{equation}
\end{fact}

\begin{cor}\label{cor:PS}
For $n$ odd, we have 
\begin{equation*}
P_{\frac{n+1}{2}}(x;n) 
= 2^{n+1}
\Sylv(\frac{1}{2}; \frac{n-1}{2})
\Sylv(\frac{x}{2};\frac{n-1}{2}).
\end{equation*}
\end{cor}

\begin{proof}
By a direct computation one finds
\begin{equation*}
\prod^{\frac{n-1}{2}}_{\ell=0}(n-4\ell)
=2^{\frac{n+1}{2}} \Sylv(\frac{1}{2};\frac{n-1}{2})
\end{equation*}
and
\begin{equation*}
\prod^{\frac{n-1}{2}}_{\ell=0}(x-(n-1)+4\ell)
=
2^{\frac{n+1}{2}} \Sylv(\frac{x}{2};\frac{n-1}{2}).
\end{equation*}
Now the proposed identity follows from Theorem \ref{thm:factorP} as 
$\gamma_n = \prod^{\frac{n-1}{2}}_{\ell=0}(n-4\ell)$.
\end{proof}

\subsection{Palindromic properties of 
$\{P_k(x;y)\}_{k=0}^{\infty}$ and $\{Q_k(x;y)\}_{k=0}^{\infty}$}
\label{subsec:PalinPQ}

We finish this section by showing 
the palindromic properties of
$\{P_k(x;y)\}_{k=0}^{\infty}$ and $\{Q_k(x;y)\}_{k=0}^{\infty}$
(see Definition \ref{def:palindromic}).
For $s \in \R\backslash \{0\}$, we set
\begin{equation*}
\sgn(s) := 
\begin{cases}
+1 & \text{if $s > 0$},\\
-1 & \text{if $s < 0$}.
\end{cases}
\end{equation*}

\subsubsection{Palindromic property of $\{P_k(x;y)\}_{k=0}^{\infty}$}
\label{subsubsec:PalinP}

We start with $\{P_k(x;y)\}_{k=0}^{\infty}$.
Recall from Corollary \ref{cor:PQ} that
$\ESol_{\frac{n+2}{2}}(P;n)$ for $n \in 2\Z_{\geq 0}$ is given as
\begin{equation*}
\ESol_{\frac{n+2}{2}}(P;n)=
\begin{cases}
\C & \textnormal{if $n \equiv 0 \Mod 4$},\\[3pt]
I^+_2\cup I^-_2 & \textnormal{if $n\equiv 2 \Mod 4$}.
\end{cases}
\end{equation*}
We then define a map
$\gt_{(P;n)}\colon \ESol_{\frac{n+2}{2}}(P;n) \to \{\pm 1\}$ as 
\begin{equation}\label{eqn:thetaP}
\gt_{(P;n)}(s)=
\begin{cases}
1 & \textnormal{if $n \equiv 0 \Mod 4$},\\
\sgn(s) & \textnormal{if $n \equiv 2 \Mod 2$}.
\end{cases}
\end{equation}

\begin{thm}[{Palindromic property of $\{P_k(x;y)\}_{k=0}^{\infty}$}]
\label{thm:PalinP}
The pair $(\{P_k(x;y)\}_{k=0}^\infty, \{(2k)!\}_{k=0}^\infty)$ 
is a palindromic pair 
with degree $\frac{n}{2}$.
Namely,
let $n \in 2\Z_{\geq 0}$ and 
$s \in \ESol_{\frac{n+2}{2}}(P;n)$. 
Then we have
$P_k(s;n)=0$ for $k \geq \frac{n+2}{2}$ and
\begin{equation}\label{eqn:PPu}
\frac{P_k(s;n)}{(2k)!}
=
\gt_{(P;n)}(s)\frac{P_{\frac{n}{2}-k}(s;n)}{(n-2k)!}
\quad
\emph{for $k \leq \frac{n}{2}$}.
\end{equation}
\vskip 0.05in
\noindent
Equivalently,
for $k \leq \frac{n}{2}$,
the palindromic identity \eqref{eqn:PPu}
is given as follows.
\begin{enumerate}[{\quad \normalfont (1)}]

\item $n \equiv 0 \Mod 4:$ We have
\begin{equation}\label{eqn:PPu0}
\frac{P_k(s;n)}{(2k)!}
=
\frac{P_{\frac{n}{2}-k}(s;n)}{(n-2k)!}
\quad
\emph{for all $s \in \C$}.
\end{equation}

\item $n \equiv 2 \Mod 4:$ We have
\begin{equation}\label{eqn:PPu2}
\frac{P_k(s;n)}{(2k)!}
=
\sgn(s)
\frac{P_{\frac{n}{2}-k}(s;n)}{(n-2k)!}
\quad
\emph{for $s \in I^+_2 \cup I^-_2$}.
\end{equation}

\end{enumerate}
\end{thm}

\begin{proof}
Take $s \in \ESol_{\frac{n+2}{2}}(P;n)$.
It follows from the equivalence given in the proof of 
Corollary \ref{cor:PQ} that $P_k(s;n)=0$ for $k \geq \frac{n+2}{2}$.
Thus, to prove the theorem,
it suffices to show \eqref{eqn:PPu0} and \eqref{eqn:PPu2}
for $k \leq \frac{n}{2}$.
We first show that 
\begin{equation}\label{eqn:PalP1}
\frac{P_k(s;n)}{(2k)!}
=
\pm \frac{P_{\frac{n}{2}-k}(s;n)}{(n-2k)!}.
\end{equation}
It follows from Theorem \ref{thm:SolKO2} and Corollary \ref{cor:PQ} that
$M$ acts on $u_{[s;n]}(t)$ as a character 
$\chi_{(\eps, \eps')}$ in \eqref{eqn:charM}; 
in particular, we have
\begin{equation}\label{eqn:usn0}
u_{[s;n]}(t)=\chi_{(\eps, \eps')}(m_2^{\textnormal{I}}) \pi_n(m_2^{\textnormal{I}})u_{[s;n]}(t).
\end{equation}
By \eqref{eqn:pi-wM}, the identity \eqref{eqn:usn0} is equivalent to
\begin{equation}\label{eqn:usn}
u_{[s;n]}(t)=\chi_{(\eps, \eps')}(m_2^{\textnormal{I}})  t^nu_{[s;n]}\left(-\frac{1}{t}\right).
\end{equation}
As $s\in\ESol_{\frac{n+2}{2}}(P;n)$,
by Proposition \ref{prop:PQ}, we have
\begin{equation*}
u_{[s;n]}(t) 
= 
\sum_{k=0}^\frac{n}{2} P_k(s;n) \frac{t^{2k}}{(2k)!}.
\end{equation*}
The identity \eqref{eqn:usn} thus yields the identity
\begin{equation}\label{eqn:usn1}
\sum_{k=0}^\frac{n}{2} P_k(s;n) \frac{t^{2k}}{(2k)!}
=
\chi_{(\eps, \eps')}(m_2^{\textnormal{I}})
\sum_{k=0}^\frac{n}{2}P_{\frac{n}{2}-k}(s;n)\frac{t^{2k}}{(n-2k)!}.
\end{equation}
Since $\chi_{(\eps, \eps')}(m_2^{\textnormal{I}}) \in \{\pm 1\}$,
the identity \eqref{eqn:PalP1} follows from \eqref{eqn:usn1}.

In order to show \eqref{eqn:PPu0} and \eqref{eqn:PPu2}, 
observe that Theorem \ref{thm:SolKO2} shows that
the character $\chi_{(\eps, \eps')}$ is given as
\begin{itemize}
\item $\chi_{(\eps, \eps')} = \chi_{\pp}$\;
for $n \equiv 0 \Mod 4$ and $s \in \C$,\vspace{3pt}
\item $\chi_{(\eps, \eps')}=\chi_{\mip}$\;
for $n \equiv 2 \Mod 4$ and $s \in I^+_2$, \vspace{3pt}
\item $\chi_{(\eps, \eps')}=\chi_{\pmi}$\;
for $n \equiv 2 \Mod 4$ and $s \in I^-_2$.
\end{itemize}
\vskip 0.1in
\noindent
By Table \ref{table:char}, we have 
$\chi_{\pp}(m_2^{\textnormal{I}})=\chi_{\mip}(m_2^{\textnormal{I}})=1$ and $\chi_{\pmi}(m_2^{\textnormal{I}})=-1$.
Now \eqref{eqn:PPu0} and \eqref{eqn:PPu2} follow from \eqref{eqn:usn1}.
\end{proof}

Theorem \ref{thm:PalinP} in particular implies
the factorial identity of $P_{\frac{n}{2}}(s;n)$ (see \eqref{eqn:P_intro})
as follows.


\begin{cor}
[Factorial identity of $P_{\frac{n}{2}}(x;n)$]
\label{cor:PalinP}
Let $n \in 2\Z_{\geq 0}$. Then the following hold.

\begin{enumerate}[{\quad \normalfont (1)}]
\item $n \equiv 0 \Mod 4:$ We have
\begin{equation*}
P_{\frac{n}{2}}(s;n) = n! 
\quad
\emph{for all $s\in \C$}.
\end{equation*}

\item $n\equiv 2 \Mod 4:$ We have
\begin{equation}\label{eqn:Pn2}
P_{\frac{n}{2}}(s;n) = \sgn(s) n!
\quad
\emph{for $s \in I^+_2 \cup I^-_2$}.
\end{equation}
\end{enumerate}
\end{cor}

\begin{proof}
This is simply the case of $k=\frac{n}{2}$ in \eqref{eqn:PPu0} and \eqref{eqn:PPu2}.
\end{proof}

We now give a proof for Corollary \ref{cor:P0} in the introduction.

\begin{proof}[Proof of Corollary \ref{cor:P0}]
It follows from Corollary \ref{cor:PalinP} that 
it suffices to show \eqref{eqn:PalinPintro}.
Suppose $n \equiv 2 \Mod 4$.
Since 
\begin{equation*}
I^+_2\cup I^-_2=\left\{\pm (4j+1) : j=0, 1, 2, \ldots, \frac{n-2}{4}\right\},
\end{equation*}
we have 
\begin{equation*}
I^+_2\cup I^-_2
=\{\pm 1, \pm 5, \pm 9 \dots, \pm (n-2)\}.
\end{equation*}
The identity \eqref{eqn:Pn2} then concludes the assertion.
\end{proof}

Theorem \ref{thm:PalinP} shows that 
$(\{P_k(x;y)\}_{k=0}^\infty, \{(2k)!\}_{k=0}^\infty)$
is a palindromic pair for $n$ even. 
Proposition \ref{prop:PalinPodd} below shows that 
it is not the case for $n$ odd.

\begin{prop}\label{prop:PalinPodd}
The pair $(\{P_k(x;y)\}_{k=0}^\infty, \{(2k)!\}_{k=0}^\infty)$
is not a palindromic pair for $n$ odd.
\end{prop}

\begin{proof}
Proof by contradiction. Suppose the contrary, that is,
there exist
$n_0 \in 1+2\Z_{\geq 0}$
and 
a map $d\colon \Z_{\geq 0} \to \R$ with
$d(n_0) \in \Z_{\geq 0}$
such that,
for all $s \in \ESol_{d(n_0)+1}(P;n_0)$,
we have
$P_k(s;n_0)=0$ for $k \geq d(n_0)+1$ and
\begin{equation*}
\frac{P_k(s;n_0)}{(2k)!}
=
\pm\frac{P_{d(n_0)-k}(s;n_0)}{(2d(n_0)-2k)!}
\quad
\text{for $k \leq d(n_0)$}.
\end{equation*}
Take $s \in \ESol_{d(n_0)+1}(P;n_0)$.
It follows from 
Proposition \ref{prop:PQ} that
the palindromity of the pair
$(\{P_k(x;y)\}_{k=0}^\infty, \{(2k)!\}_{k=0}^\infty)$ implies
\begin{equation*}
u_{[s;n_0]}(t) = \sum_{k=0}^{d(n_0)}P_k(s;n_0)\frac{t^{2k}}{(2k)!}
\in \Pol_{d(n_0)}[t]
\end{equation*}
and
\begin{equation*}
t^{d(n_0)} u_{[s;n_0]}\left(\frac{1}{t}\right) = \pm u_{[s;n_0]}(t),
\end{equation*}
which is, by \eqref{eqn:pi-wM}, equivalent to
\begin{equation*}
\pi_{d(n_0)}(m_2^{\textnormal{I}}) u_{[s;n_0]}(t) = \pm u_{[s;n_0]}(t).
\end{equation*}
Since $u_{[s;n_0]}(-t) = u_{[s;n_0]}(t)$, this further implies that 
$\C u_{[s;n_0]}(t)$ is a one-dimensional representation of $M$ for 
such $(s,n_0)$. On the other hand, as $n_0 \in 1+2\Z_{\geq 0}$,
it follows from Proposition \ref{prop:upoly} that, 
for $n_0 \equiv \ell \Mod 4$ for 
$\ell = 1, 3$, we have
\begin{equation*}
s\in I_\ell \cap \ESol_{d(n_0)+1}(P;n_0).
\end{equation*}
In particular, by Theorem \ref{thm:SolKO2},
the space
$\C u_{[s;n_0]}(t) \oplus \C v_{[s;n_0]}(t)$ forms the unique two-dimensional
irreducible representation of $M$, which is a contradiction. 
Now the proposed assertion follows.
\end{proof}

\subsubsection{Palindromic property of $\{Q_k(x;y)\}_{k=0}^{\infty}$}
\label{subsubsec:PalinQ}

We next consider $\{Q_k(x;y)\}_{k=0}^{\infty}$.
Corollary \ref{cor:PQ} shows that
\begin{equation*}
\ESol_{\frac{n}{2}}(Q;n)=
\begin{cases}
I^+_0\cup I^-_0 & \textnormal{if $n\equiv 0 \Mod 4$},\\[3pt]
\C & \textnormal{if $n \equiv 2 \Mod 4$}.
\end{cases}
\end{equation*}
We then define a map
$\gt_{(Q;n)}\colon \ESol_{\frac{n}{2}}(Q;n) \to \{\pm 1\}$ as 
\begin{equation}\label{eqn:thetaQ}
\gt_{(Q;n)}(s)=
\begin{cases}
\sgn(s) & \textnormal{if $n \equiv 0 \Mod 4$},\\
1 & \textnormal{if $n \equiv 2 \Mod 4$}.
\end{cases}
\end{equation}

\begin{thm}[{Palindromic property of $\{Q_k(x;y)\}_{k=0}^{\infty}$}]
\label{thm:PalinQ}
The pair $(\{Q_k(x;y)\}_{k=0}^\infty, \{(2k+1)!\}_{k=0}^\infty)$ 
is a palindromic pair
with degree $\frac{n-2}{2}$.
Namely, let $n \in 2(1+\Z_{\geq 0})$ and 
$s \in \ESol_{\frac{n}{2}}(Q;n)$.
Then we have $Q_k(s;n)=0$ for $k \geq \frac{n}{2}$ and
\begin{equation}\label{eqn:PPv}
\frac{Q_k(s;n)}{(2k+1)!}
=
\gt_{(Q;n)}(s)
\frac{Q_{\frac{n-2}{2}-k}(s;n)}{(n-2k-1)!}
\quad
\emph{for $k \leq \frac{n-2}{2}$}.
\end{equation}
\vskip 0.05in
\noindent
Equivalently, for $k \leq \frac{n}{2}-1$, 
the palindromic identity \eqref{eqn:PPv} is given as follows.
\begin{enumerate}[{\quad \normalfont (1)}]

\item $n \equiv 0 \Mod 4:$ We have
\begin{equation}\label{eqn:PPv0}
\frac{Q_k(s;n)}{(2k+1)!}
=
\sgn(s)
\frac{Q_{\frac{n-2}{2}-k}(s;n)}{(n-2k-1)!}
\quad
\emph{for $s \in I^+_0 \cup I^-_0$}.
\end{equation}
\vskip 0.05in

\item $n \equiv 2 \Mod 4:$ We have
\begin{equation}\label{eqn:PPv2}
\frac{Q_k(s;n)}{(2k+1)!}
=
\frac{Q_{\frac{n-2}{2}-k}(s;n)}{(n-2k-1)!}
\quad
\emph{for all $s \in \C$}.
\end{equation}

\end{enumerate}
\end{thm}

\begin{proof}
Since the argument goes similarly to the one for Theorem \ref{thm:PalinP},
we omit the proof.
\end{proof}

\begin{cor}
[Factorial identity of $Q_{\frac{n-2}{2}}(x;n)$]
\label{cor:PalinQ}
Let $n \in 2(1+\Z_{\geq 0})$. Then the following hold for 
$Q_{\frac{n-2}{2}}(s;n)$ (see \eqref{eqn:Q_intro}).

\begin{enumerate}[{\quad \normalfont (1)}]
\item $n \equiv 0 \Mod 4:$ 
We have
\begin{equation}\label{eqn:Qn0}
Q_{\frac{n-2}{2}}(s;n) = \sgn(s) (n-1)!
\quad
\emph{for $s \in I^+_0 \cup I^-_0$}.
\end{equation}

\item $n\equiv 2 \Mod 4:$ 
We have
\begin{equation*}
Q_{\frac{n-2}{2}}(s;n) = (n-1)! 
\quad
\emph{for all $s\in \C$}.
\end{equation*}
\end{enumerate}
\end{cor}

\begin{proof}
This is the case of $k=\frac{n-2}{2}$ in \eqref{eqn:PPv0} and \eqref{eqn:PPv2}.
\end{proof}

Here is a proof of Corollary \ref{cor:Q0} in the introduction.

\begin{proof}[Proof of Corollary \ref{cor:Q0}]
By Corollary \ref{cor:PalinQ}, it suffices to show \eqref{eqn:PalinQintro}.
Suppose $n\equiv 0 \Mod 4$.
Since 
\begin{equation*}
I^+_0 \cup I^-_0 =\left\{\pm (4j+3) : j=0,1,2, \ldots, \frac{n-4}{4}\right\},
\end{equation*}
we have 
\begin{equation*}
I^+_0 \cup I^-_0 
=\left\{\pm 3, \pm 7, \pm 11, \ldots, \pm(n-1)\right\}.
\end{equation*}
Then \eqref{eqn:Qn0} concludes the proposed identity.
\end{proof}

We close this section by stating about the
palindromity of the pair
$(\{Q_k(x;y)\}_{k=0}^\infty, \{(2k+1)!\}_{k=0}^\infty)$ 
for $n$ odd.

\begin{prop}\label{prop:PalinQodd}
The pair $(\{Q_k(x;y)\}_{k=0}^\infty, \{(2k+1)!\}_{k=0}^\infty)$
is not a palindromic pair for $n$ odd.
\end{prop}

\begin{proof}
As the proof is similar to the one for Proposition \ref{prop:PalinPodd},
we omit the proof.
\end{proof}


\section{Cayley continuants 
$\{\Cay_k(x;y)\}_{k=0}^\infty$ and
Krawtchouk polynomials
$\{\Cal{K}_k(x;y)\}_{k=0}^\infty$}
\label{sec:U}

In Section \ref{sec:PQ}, it was shown that
$\{P_k(x;y)\}_{k=0}^\infty$ and $\{Q_k(x;y)\}_{k=0}^\infty$
admit palindromic properties.
In this short section we show that 
the sequences of Cayley continuants $\{\Cay_k(x;y)\}_{k=0}^\infty$ 
and Krawtchouk polynomials $\{\Cal{K}_k(x;y)\}_{k=0}^\infty$
also 
admit palindromic properties. These are  
achieved in Theorem \ref{thm:recip} 
for Cayley continuants 
and in Theorem \ref{thm:PalinK} for  
Krawtchouk polynomials.

\subsection{Cayley continuants $\{\Cay_k(x;y)\}_{k=0}^\infty$}
\label{subsec:Cayley1}
We start with the definition of $\{\Cay_k(x;y)\}_{k=0}^\infty$.
For each $k\in \Z_{\geq 0}$, 
the $k \times k$ tridiagonal determinants $\Cay_k(x;y)$ are defined as follows
(\cite{Cayley58, MT05} and \cite[p.\ 429]{Muir06}).

\begin{itemize}

\item $k =0:$ $\Cay_0(x:y)=1$,
\vskip 0.1in

\item $k=1:$ $\Cay_1(x;y) = x$,
\vskip 0.2in

\item $k \geq 2:$ $\Cay_k(x;y) = 
\begin{small}
\begin{vmatrix}
x& 1    &              &            &             &       \\
y& x    &     2       &            &             &       \\
  & y-1 &     x       &  3        &             &       \\
  &       &  \dots  &  \dots & \dots  &       \\
  &       &              &y-k+3   &x            & k-1 \\
  &       &              &            &y-k+2     & x    \\
\end{vmatrix}
\end{small}
$.\\
\vskip 0.1in
\end{itemize}

Following \cite{MT05}, we refer to
each $\Cay_k(x;y)$ as a \emph{Cayley continuant}.
There are some combinatorial interpretations of them. 
For instance, they can be thought of as
a generalization of the raising factorial (shifted factorial)
$x^{\overline{k}}:=x(x+1)\cdots(x+k-1)$,
the falling factorial
$x^{\underline{k}}:=x(x-1)\cdots(x-k+1)$,
and
the factorial $k!$.
Indeed, we have
\begin{equation*}
\Cay_k(x;x)=x^{\underline{k}},
\quad
\Cay_k(x;-x)=x^{\overline{k}},
\quad
\text{and}
\quad
\Cay_k(1;-1)=k!.
\end{equation*}
\noindent
For more details on a relationship with combinatorics, 
see \cite{MT05}.

\begin{rem}
When $y = n \in 1+\Z_{\geq 0}$,
the Cayley continuant
$\Cay_{n+1}(x;n)$ is the Sylvester determinant $\Sylv(x;n)$ 
(see \eqref{eqn:Sylv_intro}).
Thus, by \eqref{eqn:factorSylv}, we have
\begin{equation}\label{eqn:USylv}
\Cay_{n+1}(x;n)=
\begin{small}
\begin{vmatrix}
x& 1    &              &            &             &       &\\
n& x    &     2       &            &             &       &\\
  & n-1 &     x       &  3        &             &       &\\
  &       &  \dots  &  \dots & \dots  &       &\\
  &       &              &3          &x            & n-1 &\\
  &       &              &            &2            & x    & n\\
  &       &              &            &              &  1   & x
\end{vmatrix}
\end{small}
=\Sylv(x;n)=\prod_{\ell=0}^n(x-n+2\ell).
\end{equation}
\end{rem}

The generating function of the
Cayley continuants $\{\Cay_k(x;y)\}_{k=0}^\infty$ is known
and several proofs are in the literature
(see, for instance, \cite{Kac47, MT05, TT91}).
In Proposition \ref{prop:genCay8} below,
by utilizing the idea for the proof of Proposition \ref{prop:PQ},
we shall provide another proof for the generating function 
as well as Sylvester's formula \eqref{eqn:USylv}.
We remark that 
our argument for the formula \eqref{eqn:USylv} can be more abstract.
(See Section \ref{sec:appendixB}.)

\begin{prop}
\label{prop:genCay8}
We have
\begin{equation}\label{eqn:generateU}
 (1+t)^{\frac{y+x}{2}}(1-t)^{\frac{y-x}{2}}
= 
\sum_{k=0}^\infty \Cay_k(x;y) \frac{t^k}{k!}.
\end{equation}
In particular,
\begin{equation*}
\Cay_{n+1}(x;n)=\prod_{\ell=0}^n(x-n+2\ell).
\end{equation*}
\end{prop}

\begin{proof}
Let $E_+$ and $E_-$ be the elements of $\f{sl}(2,\C)$ in \eqref{eqn:sl2E}.
It then follows from \eqref{eqn:Epm} that
the matrix $\dpin(E_++E_-)_{\Cal{B}}$ with respect to the ordered basis
$\Cal{B}:=\{1, t, t^2, \ldots, t^n\}$ is given as 
\begin{equation*}
\dpin(E_++E_-)_{\Cal{B}}=
\begin{small}
\begin{pmatrix}
0& -1    &              &            &             &       &\\
-n& 0    &     -2       &            &             &       &\\
  & -(n-1) &     0       &  -3        &             &       &\\
  &       &  \dots  &  \dots & \dots  &       &\\
  &       &              &-3          &0           & -(n-1) &\\
  &       &              &            &-2            & 0   & -n\\
  &       &              &            &              &  -1   & 0
\end{pmatrix}.\\[5pt]
\end{small}
\end{equation*}
Therefore the Sylvester determinant $\Cay_{n+1}(x;n) (=\Sylv(x;n)))$
is the characteristic polynomial of $\dpin(E_++E_-)$; thus,
to show \eqref{eqn:USylv}, it suffices to find the eigenvalues of 
$\dpin(E_++E_-)$.
We set 
\begin{equation*}
\Cal{D}_{[x;n]}:=(1-t^2)\frac{d}{dt}+nt-x.
\end{equation*}
By \eqref{eqn:Epm}, for $p(t) \in \Pol_n[t]$,
we have
\begin{equation*}
(\dpin(E_+)+\dpin(E_-)+x)p(t)=0
\quad
\Longleftrightarrow
\quad
\Cal{D}_{[x;n]}p(t)=0.
\end{equation*} 
Thus it further suffices to determine $s \in \C$ for which
$\Cal{D}_{[s;n]}p(t)=0$ has a non-zero solution.

By separation of variables, any solution to
$\Cal{D}_{[x;n]}f(t)=0$, where $f(t)$ is not necessarily a polynomial,
is of the from
\begin{equation*}
f(t)=c \cdot (1+t)^{\frac{n+x}{2}}(1-t)^{\frac{n-x}{2}}
\end{equation*}
for some constant $c$. 
It is clear that $(1+t)^{\frac{n+x}{2}}(1-t)^{\frac{n-x}{2}}$ becomes 
a polynomial if and only if $x=n-2\ell$ for $\ell = 0, 1, \ldots, n$.
Now Sylvester's formula \eqref{eqn:USylv} follows.

In order to show \eqref{eqn:generateU},
suppose that $\sum_{k=0}^\infty h_k(x;n) t^k$ 
is the power series solution of 
$\Cal{D}_{[x;n]}f(t)=0$ with $h_0(x;n)=1$. 
Then, by a direct observation, the coefficients 
$h_k(x;n)$ for $k\geq 1$
satisfy the following recurrence relations:
\begin{align*}
h_1(x;n)&=x,\\
h_{k+1}(x;n)&= \frac{x}{k+1}h_k(x;n) + \frac{k-n-1}{k+1}h_{k-1}(x;n).
\end{align*}
It then follows from Lemma \ref{lem:Kris} in Section \ref{sec:appendix} 
with \eqref{eqn:tridiag} that each $h_k(x;n)$ is given as
\begin{align*}
h_k(x;n)=
\begin{small}
\begin{vmatrix}
x& -1    &              &            &             &       \\
-\frac{n}{2}& \frac{x}{2}    &     -1       &            &             &       \\
  & -\frac{n-1}{3} &     \frac{x}{3}       &  -1        &             &       \\
  &       &  \ddots  &  \ddots & \ddots  &       \\
  &       &              &-\frac{n-k+3}{k-1}   & \frac{x}{k-1}            & -1 \\
  &       &              &            &-\frac{n-k+2}{k}     & \frac{x}{k}    \\
\end{vmatrix}
\end{small}
=\frac{\Cay_k(x;n)}{k!}.
\end{align*}
Thus 
$\sum_{k=0}^\infty\frac{\Cay_k(x;n)}{k!}t^k$ is 
a solution of $\Cal{D}_{[x;n]}f(t)=0$; consequently,
there exists a constant $c$ such that
\begin{equation*}
c \cdot (1+t)^{\frac{n+x}{2}}(1-t)^{\frac{n-x}{2}} = 
\sum_{k=0}^\infty\frac{\Cay_k(x;n)}{k!}t^k.
\end{equation*}
As both sides are equal to $1$ at $t=0$, we have $c=1$;
thus, the identity \eqref{eqn:generateU} holds for $y=n \in \Z_{\geq 0}$.
Since both sides of \eqref{eqn:generateU} are well-defined
for any $y\in \C$ for suitable values of $t$,
the desired identity now holds.
\end{proof}

\begin{rem}
The differential operator $\dpin(E_++E_-)$ is related to 
the study of $K$-type decomposition of the space of $K$-finite solutions
to intertwining differential operators.
In fact, let $X$ be an nilpotent element of $\f{sl}(2,\C)$ 
defined in \eqref{eqn:XY}.
Then $\dpin(E_++E_-)$ is the realization of
the differential operator $R(X)$ on the $K$-type $\Pol_n[t]$
via the identification 
$\Omega^{\textnormal{I}}\colon \fk \stackrel{\sim}{\to} \f{sl}(2,\C)$
in \eqref{eqn:omegaKO}
(see \cite[Sect.\ 5]{KuOr19}).
\end{rem}

We write 
\begin{equation*}
g^C_{[x;n]}(t) :=  (1+t)^{\frac{n+x}{2}}(1-t)^{\frac{n-x}{2}}
\end{equation*}
and set 
\begin{equation*}
\ESol_k(\Cay;n):=\{s \in \C : \Cay_k(s;n) = 0\}.
\end{equation*}
It follows from \eqref{eqn:USylv} 
(or, equivalently, from \eqref{eqn:generateU}) 
that 
\begin{equation}\label{eqn:ESolU}
\ESol_{n+1}(\Cay;n)
=\{n-2\ell : \ell=0.1, \dots, n\}.
\end{equation}

\begin{cor}\label{cor:g}
The following conditions of $s \in\C$ are 
equivalent.
\begin{enumerate}[{\normalfont (i)}]
\item $g^C_{[s;n]}(t) \in \Pol_n[t]$.
\item $s \in \ESol_{n+1}(\Cay;n)$.
\end{enumerate}
\end{cor}

\begin{proof}
A direct consequence of Proposition \ref{prop:genCay8}.
\end{proof}

\subsection{Palindromic property for $\{\Cay_k(x;y)\}_{k=0}^\infty$}
\label{subsec:PalinC}

Now we show the palindromic property 
and factorial identity for the Cayley continuants $\{\Cay_k(x;y)\}_{k=0}^\infty$.

\begin{thm}[Palindromic property 
of $\{\Cay_k(x;y)\}_{k=0}^\infty$
and factorial identity of $\Cay_n(x;n)$]
\label{thm:recip}
The pair $(\{\Cay_k(x;y)\}_{k=0}^\infty, \{k!\}_{k=0}^\infty)$ 
is a palindromic pair 
with degree $n$.
Namely,
let $n \in \Z_{\geq 0}$ and $s \in \ESol_{n+1}(\Cay;n)$. 
Then we have $\Cay_k(s;n) =0$ for $k \geq n+1$ and
\begin{equation*}
\frac{\Cay_k(s;n)}{k!}
=
(-1)^{\frac{n-s}{2}}
\frac{\Cay_{n-k}(s;n)}{(n-k)!}
\quad
\emph{for $k \leq n$}.
\end{equation*}
In particular we have 
\begin{equation}\label{eqn:Ufactorial}
\Cay_n(s;n) = (-1)^{\frac{n-s}{2}}n!
\quad
\emph{for $s \in \ESol_{n+1}(\Cay;n)$}.
\end{equation}
\end{thm}

\begin{proof}
Take $s\in \ESol_{n+1}(\Cay;n)$. 
It follows from Corollary \ref{cor:g} that, for such $s$, we have
\begin{equation}\label{eqn:GC}
g^C_{[s;n]}(t) 
= 
\sum_{k=0}^n \frac{\Cay_k(s;n)}{k!} t^k, 
\end{equation}
namely, $\Cay_k(s;n) =0$ for $k \geq n+1$.
Further, \eqref{eqn:GC} shows that
$t^ng^C_{[s;n]}\left(\frac{1}{t}\right)$ is given as
\begin{align*}
t^ng^C_{[s;n]}\left(\frac{1}{t}\right) 
=\sum_{k=0}^n \frac{\Cay_{n-k}(s;n)}{(n-k)!} t^k.
\end{align*}
On the other hand, as 
$g^C_{[x;n]}(t) =  (1+t)^{\frac{n+x}{2}}(1-t)^{\frac{n-x}{2}}$, we also have
\begin{equation*}
t^ng^C_{[s;n]}\left(\frac{1}{t}\right) = (-1)^{\frac{n-s}{2}}g^C_{[s;n]}(t).
\end{equation*}
Therefore,
\begin{equation*}
\sum_{k=0}^n \frac{\Cay_{n-k}(s;n)}{(n-k)!} t^k
=(-1)^{\frac{n-s}{2}} \sum_{k=0}^n \frac{\Cay_k(s;n)}{k!} t^k,
\end{equation*}
which yields
\begin{equation*}
\frac{\Cay_k(x;n)}{k!} 
= (-1)^{\frac{n-s}{2}} \frac{\Cay_{n-k}(x;n)}{(n-k)!} 
\quad
\text{for all $k \leq n$}.
\end{equation*}
The second identity is the case $k=n$. This concludes the theorem.
\end{proof}

We close this section by showing a proof 
of Corollary \ref{cor:Cay0} in the introduction.

\begin{proof}[Proof of Corollary \ref{cor:Cay0}]
As $\Cay_n(x;n)=\Sylv(x;n)$, 
it follows from Sylvester's factorization formula 
\eqref{eqn:Sylv0} or \eqref{eqn:factorSylv2} that
\begin{align*}
\ESol_{n+1}(\Cay;n) 
&=\{s \in \C : \Sylv(s;n)=0\}\\
&=\begin{cases}
\{0, \pm 2, \pm 4, \ldots, \pm n\} & \text{if $n$ is even},\\
\{\pm 1, \pm 3, \pm 5, \ldots, \pm n\} & \text{if $n$ is odd}.
\end{cases}
\end{align*}
\noindent
Now the proposed identities are direct consequences of \eqref{eqn:Ufactorial}.
\end{proof}

\subsection{Krawtchouk polynomials $\{\Cal{K}_k(x;y)\}_{k=0}^\infty$}
\label{subsec:PalinK}

From the palindromic property of the 
Cayley continuants $\{\Cay_k(x;y)\}_{k=0}^\infty$,
one may deduce that of a certain sequence
$\{\Cal{K}_k(x;y)\}_{k=0}^\infty$ of polynomials,
where
the polynomials $\Cal{K}_k(x;n)$
for $y=n \in \Z_{\geq 0}$
are Krawtchouk polynomials
in the sense of \cite[p.\ 130]{MS77}  
(equivalently, the case of $p=2$ for \cite[p.\ 151]{MS77}).
We close this section by discussing the palindromic property of
$\{\Cal{K}_k(x;y)\}_{k=0}^\infty$.
\vskip 0.1in

We start with the definition of $\Cal{K}_k(x;y)$.

\begin{defn}
\label{def:Kraw}
For $k \in \Z_{\geq 0}$, we define 
a 
polynomial 
$\Cal{K}_k(x;y)$ of $x$ and $y$ as
\begin{equation*}
\Cal{K}_k(x;y) :=\sum_{j=0}^k(-1)^j \binom{x}{j}\binom{y-x}{k-j},
\end{equation*}
where the binomial coefficient $\binom{a}{m}$ is defined as
\begin{equation*}
\binom{a}{m}=
\begin{cases}
\frac{a(a-1) \cdots (a-m+1)}{m!} & \text{if $m \in 1+\Z_{\geq 0}$},\\
1 & \text{if $m=0$},\\
0 & \text{otherwise}.
\end{cases}
\end{equation*}
\end{defn}

For $y=n\in \Z_{\geq 0}$, 
the polynomials $\Cal{K}_k(x;n)$ 
are called Krawtchouk polynomials (\cite[p.\ 130]{MS77}). 
In this paper, we also call the polynomials $\Cal{K}_k(x;y)$
of two variables Krawtchouk polynomials.

\begin{rem}
Symmetric Krawtchouk polynomials
(see, for instance, \cite[p.\ 237]{KLS10} ) are used 
to show Sylvester's factorization formula \eqref{eqn:USylv} in \cite{Askey05}.
We remark that 
the definition of the Krawtchouk polynomial $\Cal{K}_k(x;n)$ in this paper is different
from one in the cited paper; in particular, $\Cal{K}_k(x;n)$ is non-symmetric.
\end{rem}

It readily follows from Definition \ref{def:Kraw} that
the Krawtchouk polynomials $\{\Cal{K}_k(x;y)\}_{k=0}^\infty$ have 
the following generating function:
\begin{equation}\label{eqn:genKraw}
(1+t)^{y-x}(1-t)^x = \sum_{k=0}^\infty \Cal{K}_k(x;y) t^k.
\end{equation}

By comparing \eqref{eqn:genKraw} with the generating function
\eqref{eqn:generateU} of the Cayley continuants $\{\Cay_k(x;y)\}_{k=0}^\infty$,
we have
\begin{equation}\label{eqn:CK}
\Cal{K}_k(x;y) = \frac{\Cay_k(y-2x;y)}{k!}.
\end{equation}
\noindent
Then $\Cay_k(x;y)$ may be expressed explicitly as follows.
Here we remark that in Proposition \ref{prop:CayKraw} below we use the 
notation 
$(m)^{\overline{j}}:=m(m+1)\cdots (m+j-1)$
for the raising factorial (shifted factorial)
instead of $(m,j)$ in \eqref{eqn:Csn} to make the comparison with 
the falling factorial 
$(m)^{\underline{j}}:=m(m-1)\cdots (m-j+1)$
clearer.

\begin{prop}\label{prop:CayKraw}
We have
\begin{equation}\label{eqn:CayKraw}
\Cay_k(x;y)=
\sum_{j=0}^k\binom{k}{j}
\left(\frac{x+y}{2}\right)^{\underline{j}}
\left(\frac{x-y}{2}\right)^{\overline{k-j}}.
\end{equation}
\end{prop}

\begin{proof}
By \eqref{eqn:generateU} and \eqref{eqn:genKraw}, we have
\begin{align*}
\Cay_k(x;y)
=(k!)\cdot \Cal{K}_k(\frac{y-x}{2};y)
=(k!)\cdot \sum_{j=0}^k(-1)^j
\binom{\frac{y-x}{2}}{j}\binom{\frac{y+x}{2}}{k-j}.
\end{align*}
As $(-1)^j(m)^{\underline{j}}=(-m)^{\overline{j}}$,
this concludes the proposed identity.
\end{proof}

\begin{rem}
The identity \eqref{eqn:CayKraw}
is also shown in \cite[p.\ 356]{MT05} slightly differently.
\end{rem}

Via the identity \eqref{eqn:genKraw}, we are now going to give 
the factorization formula, the palindromic property, and the factorial
identity for the Krawtchouk polynomials $\{\Cal{K}_k(x;y)\}_{k=0}^\infty$.

\begin{prop}[Factorization formula of $\Cal{K}_{n+1}(x;n)$]
\label{prop:Kraw1}
For given $n \in \Z_{\geq 0}$, we have
\begin{equation*}
\Cal{K}_{n+1}(x;n)
=\frac{(-2)^{n+1}}{(n+1)!}\prod_{\ell=0}^n(x-\ell)
=(-2)^{n+1}\binom{x}{n+1}.
\end{equation*}
\end{prop}

\begin{proof}
The proposed identity simply follows from \eqref{eqn:CK} and 
Sylvester's factorization formula \eqref{eqn:USylv}.
\end{proof}

We next show the palindromic property and factorial identity 
for Krawtchouk polynomials
$\{\Cal{K}_k(x;y)\}$. We denote by $\{1\}_{k=0}^\infty$ the sequence
$\{a_k\}_{k=0}^\infty$ such that $a_k=1$ for all $k$.

\begin{thm}[Palindromic property of 
$\{\Cal{K}_k(x;y)\}_{k=0}^\infty$ and factorial identity of $\Cal{K}_n(x;n)$]
\label{thm:PalinK}
The pair $(\{\Cal{K}_k(x;y)\}_{k=0}^\infty, \{1\}_{k=0}^\infty)$ is 
a palindromic pair  
with degree $n$.
Namely, let $n \in \Z_{\geq 0}$ and $s \in \ESol_{n+1}(\Cal{K};n)$.
Then we have $\Cal{K}_k(s;n) =0$ for $k \geq n+1$ and
\begin{equation*}
\Cal{K}_k(s;n)
=
(-1)^{s}
\Cal{K}_{n-k}(s;n)
\quad
\emph{for $k \leq n$}.
\end{equation*}
In particular, we have 
\begin{equation}\label{eqn:Kfactorial}
\Cal{K}_n(s;n) = (-1)^s
\quad
\emph{for $s \in \ESol_{n+1}(\Cal{K};n)$}.
\end{equation}
\end{thm}

\begin{proof}
The proof is the same as that of Theorem \ref{thm:recip}, we omit the proof.
\end{proof}

We close this section by giving a proof of Corollary \ref{cor:Kraw0} 
in the introduction.

\begin{proof}[Proof of Corollary \ref{cor:Kraw0}]
It follows from Proposition \ref{prop:Kraw1} that
\begin{align*}
\ESol_{n+1}(\Cal{K};n)
&=\{s \in \C : \Cal{K}_{n+1}(s;n)=0\}\\
&=\{0, 1, 2, \dots, n\}.
\end{align*} 
Then \eqref{eqn:Kfactorial} concludes the proposed identity.
\end{proof}


\section{Appendix A: local Heun functions}
\label{sec:appendix}

In this appendix we collect several facts and lemmas
for the local Heun function 
$Hl(a,q;\ga,\gb,\gamma,\delta;z)$ at $z=0$
that are used in the main part of this paper.

\subsection{General facts}\label{subsec:Heun}
As in \eqref{eqn:HeunOp}, we set
\begin{align*}
\D_H(a,q;\ga, \gb, \gamma,\delta;z)
:=\frac{d^2}{dz^2}+\left(\frac{\gamma}{z} + \frac{\delta}{z-1}+\frac{\eps}{z-a}\right)\frac{d}{dz}
+\frac{\ga\gb z-q}{z(z-1)(z-a)}
\end{align*}
\vskip 0.1in
\noindent
with $\gamma + \delta+\eps = \ga + \gb +1$.
Then the equation
\begin{equation*}
\D_H(a,q;\ga,\gb,\gamma,\delta;z)f(z)=0
\end{equation*}
is called \emph{Heun's differential equation}. The $P$-symbol is

\begin{equation*}
P\left\{
\begin{matrix}
0 & 1 & a & \infty & &\\
0 & 0 & 0 & \ga & z & q\\
1-\gamma & 1-\delta & 1-\eps & \beta & &
\end{matrix}
\right\}.
\end{equation*}
\vskip 0.1in

Let 
$Hl(a,q;\ga, \gb,\gamma,\delta;z)$
stands for the local Heun function at $z=0$ 
(\cite{Ron95}).
As in \cite{Ron95}, we normalize 
$Hl(a,q;\ga, \gb,\gamma,\delta;z)$
so that $Hl(a,q;\ga, \gb,\gamma,\delta;0)=1$.
It is known that, for $\gamma \notin \Z$, the functions
\begin{equation}\label{eqn:HeunSol}
Hl(a,q;\ga,\gb,\gamma,\delta;z)
\quad \text{and} \quad
z^{1-\gamma}Hl(a,q';\ga-\gamma+1,\beta-\gamma+1, 2-\gamma, \delta;z)
\end{equation}

\noindent
with $q':=q+(1-\gamma)(\eps+a\delta)$
are two linearly independent solutions at $z=0$ to the Heun equation
$\D_H(a, q;\ga, \gb,\gamma,\delta;z)f(z)=0$.
(See, for instance, \cite{Maier05} and \cite[p.~99]{SL00}.)

\begin{rem}
In \cite[p.~99]{SL00},
there seems to be a typographical error on the formula
\begin{equation*}
\lambda'=\lambda+(a+b+1-c)(1-c).
\end{equation*}
This should read as
\begin{equation*}
\lambda'=\lambda+(a+b+1-c+(t-1)d)(1-c).
\end{equation*}
\end{rem}

Let 
$Hl(a,q;\ga,\gb,\gamma,\delta;z)=\sum_{k=0}^\infty c_k z^k$ 
be the power series expansion at $z=0$.  
In our normalization we have $c_0=1$.
Then $c_k$ for $k\geq 1$ satisfy the following recurrence relations
(see, for instance, \cite[p.\ 34]{Ron95}).
\begin{align}
-q +a\gamma c_1 &=0, \label{eqn:PQR1}\\[1pt]
P_kc_{k-1}-(Q_k+q)c_k+R_kc_{k+1}&=0, \label{eqn:PQR2}
\end{align}
where
\begin{align*}
P_k&=(k-1+\ga)(k-1+\gb),\\[2pt]
Q_k&=k[(k-1+\gamma)(1+a)+a\delta+\eps],\\[2pt]
R_k&=(k+1)(k+\gamma)a.
\end{align*}

Lemma \ref{lem:Kris} below shows that each coefficient $c_k$
for $Hl(a,q;\ga,\gb,\gamma,\delta;z)=\sum_{k=0}^\infty c_k z^k$
has a determinant representation.

\begin{lem}[{see, for instance, \cite[Lem.\ B.1]{Kris10}}]\label{lem:Kris}
Let $\{a_k\}_{k\in \Z_{\geq 0}}$ be a sequence with $a_0=1$ generated by the following relations:
\begin{itemize}
\item $a_1 = A_0$;
\item $a_{k+1} = A_k a_k + B_k a_{k-1}$ $(k\geq 1)$
\end{itemize}
for some $A_0$, $A_k$, $B_k \in \C$. Then $a_k$ can be expressed as 
\begin{equation*}
a_k = 
\begin{small}
\begin{vmatrix}
A_0& -1    &              &            &             &      \\
B_1& A_1    &     -1       &            &             &    \\
  & B_2 &     A_2       &  -1        &             &       &\\
  &       &  \dots  &  \dots & \dots  &       &\\
  &       &              &B_{k-2}          &A_{k-2}            & -1 \\
  &       &              &            &B_{k-1}            & A_{k-1} \\
\end{vmatrix}.
\end{small}
\end{equation*}
\end{lem}

\subsection{The local solutions $u_{[s;n]}(t)$ and $v_{[s;n]}(t)$}

Now we consider the following local solutions to \eqref{eqn:Heun1}.
\begin{align}
&u_{[s;n]}(t)=
Hl(-1, -\frac{ns}{4}; -\frac{n}{2}, -\frac{n-1}{2}, \frac{1}{2}, \frac{1-n-s}{2};t^2),
\label{eqn:Uapp}\\[3pt]
&v_{[s;n]}(t)=
tHl(-1, -\frac{n-2}{4}s; -\frac{n-1}{2}, -\frac{n-2}{2}, \frac{3}{2}, \frac{1-n-s}{2};t^2).
\label{eqn:Vapp}
\end{align}


\begin{lem}\label{lem:u}
Let $u_{[s;n]}(t)=\sum_{k=0}^\infty U_k(s;n) t^{2k}$ be the power series expansion of 
$u_{[s;n]}(t)$ at $t=0$.
Then the coefficients $U_k(s;n)$ for $k\geq 1$
satisfy the following recurrence relations.
\begin{align*}
U_1(s;n)&=E^u_0,\\[2pt]
U_{k+1}(s;n)&=E^u_k U_k(s;n) +F^u_k U_{k-1}(s;n),
\end{align*}
where
\begin{equation}\label{eqn:uk0}
E^u_0=\frac{ns}{2},\quad
E^u_k=\frac{(n-4k)s}{(2k+1)(2k+2)}
\quad \text{and} \quad
F^u_k=\frac{(n-2k+1)(n-2k+2)}{(2k+1)(2k+2)}.
\end{equation}
\end{lem}

\begin{proof}
This follows from \eqref{eqn:PQR1} and \eqref{eqn:PQR2} 
for the specific parameters in \eqref{eqn:Uapp}.
\end{proof}

\begin{lem}\label{lem:v}
Let $v_{[s;n]}(t)=\sum_{k=0}^\infty V_k(s;n) t^{2k+1}$ be the power series expansion of 
$v_{[s;n]}(t)$ at $t=0$. 
Then the coefficients $V_k(s;n)$ $(k\geq 1)$
satisfy the following recurrence relations.
\begin{align*}
V_1(s;n)&=E^v_0,\\[2pt]
V_{k+1}(s;n)&=E^v_k V_k(s;n) +F^v_k V_{k-1}(s;n),
\end{align*}
where
\begin{equation}\label{eqn:vk0}
E^v_0=\frac{(n-2)s}{6}, \quad
E^v_k=\frac{(n-4k-2)s}{(2k+2)(2k+3)}
\quad \text{and} \quad
F^v_k=\frac{(n-2k)(n-2k+1)}{(2k+2)(2k+3)}.
\end{equation}
\end{lem}

\begin{proof}
This follows from \eqref{eqn:PQR1} and \eqref{eqn:PQR2} 
for the specific parameters in \eqref{eqn:Vapp}.
\end{proof}

It follows from Lemmas \ref{lem:u}, \ref{lem:v}, and \ref{lem:Kris} that 
$U_k(s;n)$ and $V_k(s;n)$ have determinant representations
\begin{equation}\label{eqn:uk}
U_k(s;n) = 
\begin{small}
\begin{vmatrix}
E^u_0& -1    &              &            &             &      \\
F^u_1& E^u_1    &     -1       &            &             &    \\
  & F^u_2 &     E^u_2       &  -1        &             &       &\\
  &       &  \dots  &  \dots & \dots  &       &\\
  &       &              &F^u_{k-2}          &E^u_{k-2}            & -1 \\
  &       &              &            &F^u_{k-1}            & E^u_{k-1} \\
\end{vmatrix}
\end{small}
\end{equation}
and
\begin{equation}\label{eqn:vk}
V_k(s;n) = 
\begin{small}
\begin{vmatrix}
E^v_0& -1    &              &            &             &      \\
F^v_1& E^v_1    &     -1       &            &             &    \\
  & F^v_2 &     E^v_2       &  -1        &             &       &\\
  &       &  \dots  &  \dots & \dots  &       &\\
  &       &              &F^v_{k-2}          &E^v_{k-2}            & -1 \\
  &       &              &            &F^v_{k-1}            & E^v_{k-1} \\
\end{vmatrix}.
\end{small}
\end{equation}
\vskip 0.1in


\section{Appendix B: A proof of Sylvester's formula}
\label{sec:appendixB}

In this short appendix
we provide a proof of Sylvester's factorization formula
\begin{equation}\label{eqn:appSylv} 
\Sylv(x;n)=\prod_{\ell=0}^n(x-n+2\ell),
\end{equation}
based on a general theory of 
$\f{sl}(2,\C)$ representations.
For other proofs and related topics, 
see the remark after Theorem \ref{thm:factorP} and 
Proposition \ref{prop:genCay8}.
For possible wide audience of this appendix,
we shall discuss a proof in two ways;
one of which is abstract in nature
(Section \ref{subsec:Sylv1})
and the other uses a concrete realization 
of irreducible represenations
(Section \ref{subsec:Sylv2}).
We remark that the arguments in this appendix can be thought of as
a baby case of the one discussed in 
Section \ref{subsec:Omega12}.

\subsection{Sylvester determinant $\Sylv(x;n)$ and $\f{sl}(2,\C)$ representations I}
\label{subsec:Sylv1}

Let $E_+$, $E_-$, and $E_0$
be the elements of $\f{sl}(2,\C)$ defined in \eqref{eqn:sl2E},
and let $(d\rho_n, V_n)$ be an irreducible representation of 
$\f{sl}(2,\C)$ with $\dim_\C V_n = n+1$. 
Then there exists an ordered basis 
$\Cal{B}$ of $V_n$ such that 
the matrix $d\rho_n(E_++E_-)_{\Cal{B}}$ 
with respect to $\Cal{B}$ is given as
\begin{equation}\label{eqn:SylvE}
d\rho_n(E_++E_-)_{\Cal{B}} =
\begin{small}
\begin{pmatrix}
0& -1    &              &            &             &       &\\
-n& 0    &     -2       &            &             &       &\\
  & -(n-1) &     0       &  -3        &             &       &\\
  &       &  \dots  &  \dots & \dots  &       &\\
  &       &              &-3          &0           & -(n-1) &\\
  &       &              &            &-2            & 0   & -n\\
  &       &              &            &              &  -1   & 0
\end{pmatrix}\\[5pt]
\end{small}
\end{equation}
(cf.\ \cite[Sect.\ 7.2]{Hum72}).
Then the Sylvester determinant
$\Sylv(x;n)$ (\eqref{eqn:factorSylv}) is the characteristic polynomial of 
$d\rho_n(E_++E_-)_{\Cal{B}}$. As the leading coefficient of $\Sylv(x;n)$ is one,
this implies that 
\begin{equation}\label{eqn:SylvSpec}
\Sylv(x;n)=\prod_{\lambda\in\Spec(d\rho_n(E_++E_-))}(x-\lambda),
\end{equation}
where $\Spec(T)$ denotes the set of eigenvalues of 
a linear map $T$. Thus, to show \eqref{eqn:appSylv},
 it suffices to find the eigenvalues of 
$d\rho_n(E_++E_-)$. 
Since $E_++E_-$ is conjugate $E_0$ via an element of $SU(2)$
(cf.\ \cite[Thm.\ 4.34]{Knapp02}),
it is further sufficient to determine $\Spec(d\rho_n(E_0))$.
It is well-known that 
the eigenvalues of $d\rho_n(E_0)$ are 
$n-2j$ for $j=0, 1, 2 , \ldots, n-1, n$
(cf.\ \cite[Sect.\ 7.2]{Hum72}).
Therefore we have
\begin{equation*}
\Spec(d\rho_n(E_++E_-))=\{n-2j : j=0, 1, 2 , \ldots, n-1, n\}.
\end{equation*}
Now the desired factorization formula \eqref{eqn:appSylv} follows
from \eqref{eqn:SylvSpec}.

\subsection{Sylvester determinant $\Sylv(x;n)$ and $\f{sl}(2,\C)$ representations II}
\label{subsec:Sylv2}

We now give some details of the arguments given in Section \ref{subsec:Sylv1}
with a concrete realization of irreducible $\f{sl}(2,\C)$ representations.

Let $(\pi_n, \Pol_n[t])$ be the irreducible representation of $SU(2)$
defined in \eqref{eqn:pin} and $(\dpin, \Pol_n[t])$
the corresponding irreducible representation of $\f{sl}(2,\C)$,
so that $\dpin(E_j)$ for $j= +, -, 0$ are
the differential operators given in \eqref{eqn:Epm}.
It is easily checked that the matrix $\dpin(E_++E_-)_{\Cal{B}}$
with respect to the ordered basis
$\Cal{B}:=\{1,\, t,\,  t^2,\, \ldots,\,  t^n\}$
is given as \eqref{eqn:SylvE}; therefore, we have
\begin{equation}\label{eqn:EES}
\Sylv(x;n)=\prod_{\lambda\in\Spec(\dpin(E_++E_-))}(x-\lambda).
\end{equation}
On the other hand,  the matrix 
 $\dpin(E_0)_{\Cal{B}}$ is given as
\begin{equation*}
\dpin(E_0)_{\Cal{B}}=
\textnormal{diag}(n, \, n-2, \,  n-4, \,  \ldots, \,  -n+2, \, -n),
\end{equation*}
where $\textnormal{diag}(m_0, m_1, \ldots, m_n)$ denotes 
a diagonal matrix. Thus
the eigenvalues of $\dpin(E_0)$ on $\Pol_n[t]$
are $n-2j$ for $j=0, 1, 2 , \ldots, n-1, n$. 

For
\begin{equation*}
g_0 :=\frac{1}{\sqrt{2}}\begin{pmatrix} 1 & 1 \\ -1 & 1 \end{pmatrix}
\in SU(2),
\end{equation*}
we have
\begin{equation*}
\Ad(g_0)(E_++E_-) = E_0.
\end{equation*}
Thus,
\begin{equation*}
\dpin(E_0) = \pi_n(g_0)\dpin(E_++E_-)\pi_n(g_0)^{-1}.
\end{equation*}
Therefore,
\begin{equation}\label{eqn:eigenE2}
\Spec(\dpin(E_++E_-))=
\{n-2j : j=0, 1, 2 , \ldots, n-1, n\}.
\end{equation}
Now \eqref{eqn:EES} and \eqref{eqn:eigenE2} conclude the 
desired formula.

\vskip 0.1in


\textbf{Acknowledgements.}
Part of this research was conducted during a visit of the first author
at the Department of Mathematics of Aarhus University.
He is appreciative of their support and warm hospitality during his stay.

The authors are grateful to Anthony Kable for his suggestion 
to study the $K$-type formulas for the Heisenberg ultrahyperbolic operator.
They would also like to express their gratitude to 
Hiroyuki Ochiai for the helpful discussions on the hypergeometric equation 
and Heun equation, especially for Lemma \ref{lem:Heun}.
Their deep appreciation also goes to Hiroyoshi Tamori for sending his Ph.D. dissertation, from which they thought of an idea for the hypergeometric model.


The first author was partially supported by JSPS
Grant-in-Aid for Young Scientists (JP18K13432).




\bibliographystyle{amsplain}
\bibliography{KuOrsted}


\end{document}